\documentclass{amsart}
\usepackage{showkeys}
\usepackage{amsmath} 
\usepackage{amssymb}
\usepackage{mathrsfs}

\newtheorem{theorem}{Theorem}[section] 
\newtheorem{claim}[theorem]{Claim}
\newtheorem{tft}[theorem]{The Freeness Theorem}

\newtheorem{conclusion}[theorem]{Conclusion}
\newtheorem{observation}[theorem]{Observation}

\theoremstyle{definition}
\newtheorem{definition}[theorem]{Definition}
\newtheorem{explanation}[theorem]{Explanation}

\newtheorem{discussion}[theorem]{Discussion}
\newtheorem{fact}[theorem]{Fact}

\newtheorem{convention}[theorem]{Convention}

\theoremstyle{remark}
\newtheorem{remark}[theorem]{Remark}
\newtheorem{question}[theorem]{Question}
\newtheorem{notation}[theorem]{Notation}
\newtheorem{context}[theorem]{Context}

\newcommand{\pp}{{\rm pp}}
\newcommand{\cp}{{\rm c.p.}}

\newcommand{\Ax}{{\rm Ax}}

\newcommand{\fr}{{\rm fr}}

\newcommand{\add}{{\rm add}}

\newcommand{\Ker}{{\rm Ker}}
\newcommand{\Reg}{{\rm Reg}}
\newcommand{\Ext}{{\rm Ext}}
\newcommand{\ZFC}{{\rm ZFC}}
\newcommand{\End}{{\rm End}}
\newcommand{\NTDC}{{\rm NTDC}}
\newcommand{\otp}{{\rm otp}}
\newcommand{\pcf}{{\rm pcf}}

\newcommand{\ini}{{\rm ini}}

\newcommand{\Ord}{{\rm Ord}}

\newcommand{\MA}{{\rm MA}}

\newcommand{\TDC}{{\rm TDC}}

\newcommand{\BB}{{\rm BB}}

\newcommand{\Hom}{{\rm Hom}}

\newcommand{\ann}{{\rm ann}}

\newcommand{\bd}{{\rm bd}}

\newcommand{\id}{{\rm id}}

\newcommand{\tcf}{{\rm tcf}}

\newcommand{\fin}{{\rm fin}}

\newcommand{\Rang}{{\rm Rang}}

\newcommand{\rest}{{\restriction}}
\newcommand{\dom}{{\rm dom}}

\newcommand{\wilog}{{\rm without loss of generality}}
\newcommand{\Wilog}{{\rm Without loss of generality}}

\newcommand{\then}{{\underline{then}}}
\newcommand{\when}{{\underline{when}}}
\newcommand{\oor}{{\underline{or}}}
\newcommand{\where}{{\underline{where}}}
\newcommand{\Then}{{\underline{Then}}}

\newcommand{\Iff}{{\underline{iff}}}
\newcommand{\mn}{{\medskip\noindent}}
\newcommand{\sn}{{\smallskip\noindent}}

\newcommand{\cA}{{\mathscr A}}

\newcommand{\bbI}{{\mathbb I}}

\newcommand{\gB}{{\mathfrak B}}
\newcommand{\gA}{{\mathfrak A}}

\newcommand{\cH}{{\mathscr H}}

\newcommand{\cF}{{\mathscr F}}

\newcommand{\bbN}{{\mathbb N}}

\newcommand{\bbP}{{\mathbb P}}
\newcommand{\cP}{{\mathscr P}}

\newcommand{\gU}{{\mathfrak U}}
\newcommand{\bbQ}{{\mathbb Q}}
\newcommand{\bbZ}{{\mathbb Z}}

\newcommand{\bbR}{{\mathbb R}}

\newcommand{\cS}{{\mathscr S}}

\newcommand{\gx}{{\mathfrak x}} 
\newcommand{\varp}{{\varepsilon}} 
\newcommand{\cU}{{\mathscr U}}

\newcommand{\cf}{{\rm cf}}
\newcommand{\pr}{{\rm pr}}

\newcount\skewfactor
\def\mathunderaccent#1#2 {\let\theaccent#1\skewfactor#2
\mathpalette\putaccentunder}
\def\putaccentunder#1#2{\oalign{$#1#2$\crcr\hidewidth
\vbox to.2ex{\hbox{$#1\skew\skewfactor\theaccent{}$}\vss}\hidewidth}}

\newenvironment{PROOF}[2][\proofname.]
   {\begin{proof}[#1]\renewcommand{\qedsymbol}{\textsquare$_{\rm #2}$}}
   {\end{proof}}

\begin{document}

\title[Abelian groups, pcf and black boxes]
{Quite free complicated abelian groups, pcf and Black Boxes}
\author {Saharon Shelah}
\address{Einstein Institute of Mathematics\\
Edmond J. Safra Campus, Givat Ram\\
The Hebrew University of Jerusalem\\
Jerusalem, 91904, Israel\\
 and \\
 Department of Mathematics\\
 Hill Center - Busch Campus \\ 
 Rutgers, The State University of New Jersey \\
 110 Frelinghuysen Road \\
 Piscataway, NJ 08854-8019 USA}
\email{shelah@math.huji.ac.il}
\urladdr{http://shelah.logic.at}
\thanks{The author thanks the Israel Science Foundation for support of
this paper, Grant No. 1053/11. Publication 1028.  References like
\cite[Th0.2=Ly5]{Sh:950} means the label of Th 0.2 is y5.
The author thanks Alice Leonhardt for the beautiful typing.  The
reader should note that the version in my website is usually more
updated that the one in the mathematical archive.
First typed January 25, 2012}



% Previous version: June 17, 2018

% Previous version 2013/Dec/31; formerly F1200 which was copied to
%  Sh1028 in July 2013

% Previous version: January 8, 2019

\date{January 18, 2019}

\begin{abstract}
We would like to build
Abelian groups (or $R$-modules) which on the one hand are quite free,
say $\aleph_{\omega +1}$-free, and on the other hand, are complicated
in suitable sense.  We choose as our test problem having 
no non-trivial homomorphism to
$\bbZ$ (known classically for $\aleph_1$-free, recently for
$\aleph_n$-free).  We succeed to prove the existence of 
even $\aleph_{\omega_1 \cdot n}$-free ones.
This requires building $n$-dimensional black boxes, which are quite
free.  This combinatorics is of self interest and we believe will be
useful also for other purposes.
On the other hand, modulo suitable large cardinals, we prove that
it is consistent that every
$\aleph_{\omega_1 \cdot \omega}$-free Abelian group has non-trivial
homomorphisms to $\bbZ$.
\end{abstract}

\maketitle
\numberwithin{equation}{section}
\setcounter{section}{-1}
\newpage

\centerline {Anotated Content}
\bigskip

\noindent
\S0 \quad Introduction, (labels y,z), pg.\pageref{Introduction}

\S(0A) \quad Abelian groups and modules, pg.\pageref{Abelian}
%\mn
%\begin{enumerate}
%\item[${{}}$]  [  ]
%\end{enumerate}

\S(0B) \quad Notation, pg.\pageref{Notation}
%\mn
%\begin{enumerate}
%\item[${{}}$]  [  ]
%\end{enumerate}
\bigskip

\noindent
\S1 \quad Black Boxes, (label a), pg.\pageref{Black}
\mn
\begin{enumerate}
\item[${{}}$]  [We prove the existence of $n$-dimensional black boxes
  as in \cite{Sh:883}, which are, e.g. $\aleph_{\omega \cdot n}$-free
  and even $\aleph_{\omega_1 \cdot n}$-free.  It is self contained except some
  quotation concerning pcf.]
\end{enumerate}
\bigskip

\noindent
\S2 \quad Building Abelian groups, (label d), pg.\pageref{Building}
\mn
\begin{enumerate}
\item[${{}}$]  [Here we prove the existence of
 $\aleph_{\omega_1 \cdot n}$-free Abelian group $G$ with no 
non-trivial homomorphism into $\bbZ$.]
\end{enumerate}
\bigskip

\noindent
\S3 \quad Forcing, (label g), pg.\pageref{3}
\mn
\begin{enumerate}
\item[${{}}$]  [We prove the consistency of ``for every
  $\aleph_{\omega_1 \cdot \omega}$-free Abelian group $G,\Hom(G,\bbZ) \ne
  0$".  Moreover, every such $G$ is a Whitehead group.]
\end{enumerate}
\newpage

\section {Introduction} \label{Introduction}
\bigskip

\subsection {Abelian Groups} \label{Abelian} \
\bigskip

We would like to determine the supremum of all $\lambda$ for which we can prove
 $\TDC_\lambda$, so, dually, the  minimal $\lambda$ such that consistently we
have $\NTDC_\lambda$ which means the failure of $\TDC_\lambda$,
the trivial dual conjecture for $\lambda$, where:
\mn
\begin{enumerate}
\item[$(\TDC_\lambda)$]  \quad there is a $\lambda$-free Abelian
  group $G$ such that $\Hom(G,\bbZ)=0$.
\end{enumerate}
\mn
This seems the weakest algebraic statement of this kind; it is consistent
that the number is $\infty$,
as if $\bold V = \bold L$ then $\TDC_\lambda$ holds for every
$\lambda$ (see, e.g. \cite{GbTl12}).  On the one hand by 
Magidor-Shelah \cite{MgSh:204}, for
$\lambda = \min\{\lambda: \lambda$ is a fixed point, that is $\lambda = 
\aleph_\lambda\}$, $\NTDC_\lambda$ is
consistent,  as more is proved there:
consistently ``$\lambda$-free $\Rightarrow$ free".  On
the other hand, since long ago we know the following 
for $\lambda = \aleph_1$ and
recently by \cite{Sh:883}, we know that for $\lambda = \aleph_n$
there are examples using the $n$-BB ($n$-dimensional
black boxes) introduced there (for every $n$).  
Subsequently those were used for more
complicated algebraic relatives in G\"obel-Shelah \cite{GbSh:920},
G\"obel-Shelah-Str\"ungman \cite{GbShSm:981} 
and G\"obel-Herden-Shelah \cite{GbHeSh:970}.  In
\cite{Sh:898} we have several close approximations to proving in ZFC the
existence for $\aleph_\omega$, that is $\TDC_{\aleph_\omega}$
using 1-black boxes.

Here we finally fully prove that $\TDC_{\aleph_\omega}$ holds and much
more;  $\lambda = \aleph_{\omega_1 \cdot \omega}$
is the first cardinal for which $\TDC_\lambda$ cannot be proved in
$\ZFC$.  The existence proof for $\lambda' < \lambda$ is a major
result here, relying on existence proof of quite free $n$-black boxes,
(in \S1) which use results on $\pcf$ (see \cite{Sh:1008}).  
For complementary consistency results we start with the universe forced in
\cite{MgSh:204} and then we force with a c.c.c. forcing notion making 
``$\MA + 2^{\aleph_0}$ large" but we have to work to show the desired result.

Of course, we can get better results ($\mu^+$-free) when $\mu \in
\bold C_\theta$ (see Definition \ref{y12}) is so called 1-solvable or 
$2^\mu = 2^{< \Upsilon} < 2^\Upsilon$ and $\Upsilon < 2^\mu$.

Note a point which complicates our work relative to previous ones: 
the amount of freeness
(i.e. the $\kappa$ such that we demand $\kappa$-free) and the
cardinality of the structure are markedly different.  In \cite{Sh:898}
this point is manifested when we construct say $G$ of cardinality
$\lambda$ which is $\mu^+$-free where $\mu \in \bold C_{\aleph_0}$ or
$\mu \in \bold C_{\aleph_1}$ and $\lambda = 2^\mu$ or 
$\min\{\lambda:2^\lambda > 2^\mu\}$.  The ``distance" is even larger
in \cite{Sh:883}.

An interesting point here is that for many non-structure problems we
naturally end up with two incomparable proofs.  One is when we have a
$\mu^+$-free $\cF \subseteq {}^\partial \mu$ of cardinality
$\lambda,\lambda$ as above.  In this case the amount of freeness is
large.  In the other, we use the black box from Theorem \ref{a51}.
But we may like to use more sophisticated black boxes, say start with
$\lambda_\ell,\mu_\ell(\ell \le \bold k)$, a black box $\bold x$ as in
Theorem \ref{a51} and combine it with \cite{Sh:775}.  The quotients
$G/G_{\delta +1},\delta$ a limit ordinal are close to being 
$\lambda^+_{\bold k}$-free,
replacing free by direct sums of small subgroups.

Recall from \cite[\S3]{Sh:309}: if we are given BB approximating
models with universe, e.g. $\kappa_2$ by ``guesses of cardinality
$\kappa_1$", and usually models $\kappa_2 = \kappa^{\kappa_1}_2$ then 
we can construct models of cardinality $\kappa_2$ quite freely except
the ``corrections" toward avoiding, e.g. undesirable endomorphisms,
i.e. for each approximation of such endomorphisms given by the BB is
seen as a ``task" how to avoid that in the end there will be an
endomorphism extending the one given by the approximation.  The
``price" is that we make the construction not free, but between the
various approximations there is little interaction.  
This will hopefully help in
\cite{Sh:F918}, which follows \cite{Sh:897} to use $\partial >
\aleph_0$ and here to try to sort out the complicated cases like
$\End(G) \cong R$.  Maybe we can get a neater proof.

In \cite{Sh:54}, \cite{Sh:51} we suggested that combinatorial proofs
from \cite[Ch.VIII]{Sh:a}, \cite[Ch.VIII]{Sh:c}, should be useful for
proving the existence of many non-isomorphic structures, as well as
rigid and indecomposable ones.  The most successful case were black
boxes applied to Abelian groups and modules first applied in
\cite{Sh:172}, \cite{Sh:227}, that is, 
\mn
\begin{enumerate}
\item[$(A)$]   for separable Abelian $p$-groups $G$, 
proving the existence of ones of
cardinality $\lambda = \lambda^{\aleph_0}$ with only so called small
endomorphisms; (\cite{Sh:172}); 
\sn
\item[$(B)$]  Let $R$ be a ring whose additive group $R^+$ is
  cotorsion-free, i.e. $R^+$ is reduced and has no subgroups
  isomorphic to $\bbZ/p \bbZ$ or to the $p$-adic integers.  For
  $\lambda = \lambda^{\aleph_0} > |R|$ there is an abelian group $G$
  of cardinality $\lambda$ whose endomorphism ring is isomorphic to
  $R$ and as an $R$-module it is $\aleph_1$-free (\cite[Th.0.1,pg.40]{Sh:227}).
\end{enumerate}
\mn
We can relax the demands on $R^+$ and may require that $G$ extends a
suitable group $G_0$ such that $R$ is realized by $\End(G)$ modulo a
suitable ideal of ``small" endomorphisms.
\mn
\begin{enumerate}
\item[$(C)$]  Let $R$ be a ring whose additive group is the completion
  of a direct sum of copies of the $p$-adic integers.  If
 $\lambda^{\aleph_0} \ge |R|$ then there exists a separable Abelian $p$-group
$G$ with so called basic subgroup of cardinality $\lambda$ and $R =
\End(G)/\End_s(G)$.  As usual we get $\End(G) = \End_s(G) \oplus R$;
(\cite[Th.0.2,pg.41]{Sh:227}).
\end{enumerate} 
\mn
On previous history of those algebraic problems see \cite{EM02}, \cite{GbTl12}.
Quite many works using black boxes follow, starting with Corner-G\"obel
\cite{CoGb85},  see again \cite{EM02}, \cite{GbTl12}.  On Black boxes
in set theory with weak versions of choice see \cite[\S3A]{Sh:1005}
with no choice \cite[\S3B]{Sh:1005} and for $\bold k$-dimensional
maybe see \cite{Sh:F1303}.

On further applications of those black boxes continuing the present
work, mainly representation of a ring $R$ and the endomorphism ring of
a quite free Abelian group, see \cite{Sh:1045}.
\bigskip

\begin{discussion}  
\label{y6}
1) Note that usually, the known constructions were either for
$\lambda$-free $R$-module of cardinality $\lambda$ using a
non-reflecting $S \subseteq S^\lambda_{\aleph_0}$ with diamond \underline{or}
$\aleph_1$-free of some cardinality $\lambda$ (mainly $\lambda =
(\mu^{\aleph_0})^+$ but also in some other cases) many 
times using a black box (see \cite{Sh:309}) or ``the elevator"
(see \cite{GbTl12}).  
In the former we use induction on $\alpha < \lambda$ and 
each $\alpha$ has ``one task".

Using black boxes in the nice versions, we have for each 
$\delta \in S$ a perfect
set of pairwise isomorphic tasks.

To deal with getting an $\aleph_n$-free Abelian group $G$ with
$\Hom(G,\bbZ) = 0$; the $n$-dimensional black boxes 
actually constructed and used in
\cite{Sh:883} were
products of black boxes from \cite{Sh:309}, each black box separately
is only $\aleph_1$-free but the product of $k$ gives
$\aleph_k$-freeness.  Here things are more complicated.

\noindent
2) Here cardinality and freeness differ.

\noindent
3) Note that the versions of freeness of BB 
in \cite{Sh:898} and here are not the same.
\end{discussion}
\bigskip

\subsection {Notation} \label{Notation}\
\bigskip

\noindent
Recall (on $\pp$ see \cite{Sh:g} but the reader can just use \ref{y16} below).
\begin{definition}
\label{y12}
Let $\bold C = \{\mu:\mu$ strong limit singular and pp$(\mu) =^+ 2^\mu\}$

\[
\bold C_\kappa = \{\mu \in \bold C:\cf(\mu) = \kappa\}.
\]
\end{definition}

\begin{claim}
\label{y16}
\mn
\begin{enumerate}
\item[$(a)$]  $\mu \in \bold C$ if $\mu$ is strong limit singular
of uncountable cofinality
\sn
\item[$(b)$]  if $\mu = \beth_\delta > \cf(\mu)$ and
$\delta = \omega_1$ or just $\cf(\delta) > \aleph_0$, \then \, 
$\mu \in \bold C_{\cf(\mu)}$ and for a club (= a closed unbounded subset) of 
$\alpha < \delta$ we have $\beth_\alpha \in \bold C$.
\end{enumerate}
\end{claim}

\begin{PROOF}{\ref{y16}}
Clause (a) holds by \cite[Ch.II,\S2]{Sh:g} and clause (b) by 
\cite[Ch.IX,\S5]{Sh:g}.  
\end{PROOF}

\begin{explanation}
\label{y29}
1) A reader, particularly with algebraic background may wonder how the
ideals defined in Definition \ref{y32} below are used in the algebraic
construction.  For an ideal $J$ on a set $S$ we may try to find an
Abelian group $G_1$ extending the free Abelian group $G_0 =
\oplus\{\bbZ x_s:s \in S\}$ such that the quotient
$G_1/\oplus\{\bbZ_s:s \in S_1\}$ is free for every $S_1 \in J$.  In
particular we would like to have some $h_0 \in \Hom(G_0,\bbZ)$ which cannot be
extended to a homomorphism from $G_1$ to $\bbZ$.  Copies of such
tuples $(S,J,G_1,G_0,h_0)$ are used as ``the building block" in the
constructions, so finding such examples is crucial; we find ones, see
in \S2; more in \cite{Sh:1045}.

\noindent
2) Concerning Observation \ref{y34}, note that the product $J_1 \times
J_2$ is not symmetric (even up to isomorphisms) because if
e.g. $\partial < \kappa$ then $J_\partial \times J_\kappa = \{A
\subseteq \partial \times \kappa$: for some $i < \partial,j < \kappa$
we have $A \subseteq (i \times \kappa) \cup (j \times \partial)\}$;
but $J_\kappa \times J_\partial$ has no such representation. 
\end{explanation}

\begin{definition}
\label{y32}
1) For a set $S$ of ordinals with no last member let $J^{\bd}_S$ be
 the ideal consisting of the bounded subsets of $S$.

\noindent
2) If $J_\ell$ is an ideal on $S_\ell$ for $\ell=1,2$ \then \, $J_1 \times
   J_2$ is the ideal on $S_1 \times S_2$ consisting of the $S
   \subseteq S_1 \times S_2$ such that $\{s_1 \in S_1:\{s_2 \in
   S_2:(s_1,s_2) \in S\} \notin J_2\}$ belongs to $J_1$.

\noindent
3) If $\delta_1,\delta_2$ are limit ordinals, $J_\ell$ is an ideal on
   $\delta_\ell$ and $\delta_1 \cdot \delta_2 = \delta_3$ \then \,
$J_1 * J_2$ is the following ideal on $\delta_3$: it consists of 
$\{\{\delta_1 \cdot i+j:i < \delta_2,j < \delta_1$ and 
$(j,i) \in A\}:A \in J_1 \times J_2\}$.

\noindent
4) If $\delta_1,\delta_2$ are limit ordinals, $J_\ell$ is an ideal on
$\delta_\ell$ for $\ell=1,2$ and $\delta_1 \cdot
\delta_2 = \delta_3$ \then \, $J_1 \odot J_2$ is the following ideal on
$\delta_3$: it consists of $\{\{\delta_1 \cdot i+j:i < \delta_2,j <
\delta_1$ and $(i,j) \in A\}:A \in J_2 \times J_1\}$.
\end{definition}

\begin{observation}
\label{y34}
If $\partial \ge \kappa$ are regular cardinals 
\then \, $J^{\bd}_\partial \times
J^{\bd}_\kappa$ is isomorphic to $J^{\bd}_\partial * J^{\bd}_\kappa$
which include in $J^{\bd}_\partial \odot J^{\bd}_\kappa$ which is
isomorphic to $J^{\bd}_\kappa \times J^{\bd}_\partial$. 
\end{observation}

\begin{PROOF}{\ref{y34}}
Should be clear but we elaborate the first equivalence.

Why $J' = J^{\bd}_\partial \times J^{\bd}_\kappa$ is isomorphic to
$J'' = J^{\bd}_\partial * J^{\bd}_\kappa$?

Note that $J'$ is an ideal on $\partial \times \kappa$ and $J''$ is an
ideal on $\partial \cdot \kappa$.  We define a function $\pi:\partial
\times \kappa \rightarrow \partial \cdot \kappa$ by:
\mn
\begin{enumerate}  
\item[$(*)$]  $\pi((i,j)) = \partial \cdot j +i$, so $\pi$ is a
  one-to-one function from $\partial \times \kappa$ onto $\partial
  \cdot \kappa$ by the rules of ordinal division.
\end{enumerate}
\mn
It suffices to prove that for any $A \subseteq \partial \times \kappa$
that $A \in J' \Leftrightarrow \pi''(A) \in J''$; so fix $A
\subseteq \partial \times \kappa$ and below we have $\bullet_i
\Leftrightarrow \bullet_{i+1}$ hence $\bullet_1 \Leftrightarrow \bullet_4$
which suffices, when:
\mn
\begin{enumerate}
\item[$\bullet_1$]  $A \in J'$
\sn
\item[$\bullet_2$]  $\{s_1 \in \partial:\{s_2 \in \kappa:(s_1,s_2) \in
  A\} \notin J^{\bd}_\kappa\} \in J^{\bd}_\partial$  
\sn
\item[$\bullet_3$]  $\{i < \partial:\{j<\kappa:\partial j +i \in
  \pi''(A)\} \notin J^{\bd}_\kappa\} \in J^{\bd}_\partial$ 
\sn
\item[$\bullet_4$]  $\pi''(A) \in J''$.
\end{enumerate}
\mn
That is, $\bullet_1 \leftrightarrow \bullet_2$ by the definition of
$J'$ and $\bullet_2 \leftrightarrow \bullet_3$ by the choice of $\pi$
and $\bullet_3 \leftrightarrow \bullet_4$ by the definition of $J''$.
\end{PROOF}

\begin{definition}
\label{y37}
1) We say $\cF \subseteq {}^S X$ is $(\theta,J)$-free
\when\footnote{E.g. \cite{Sh:g}, this version is used.
 Sometimes we even demand $\alpha < \alpha_* \Rightarrow 
\{s \in S:\eta_\alpha(s) \in \{\eta_\beta(t):\beta < \alpha,t \in I\}\} 
\in J$.  But in the main case ``$J$ is a $\theta$-complete filter 
on $\theta$", the versions in \ref{y37}(1),(2) are equivalent, see 
\ref{a25}.} \, $J$ is an ideal on $S$ and for every $\cF' \subseteq
\cF$ of cardinality $< \theta$ there is a sequence
$\langle w_\eta:\eta \in \cF' \rangle$ such that: $\eta \in \cF'
\Rightarrow w_\eta \in J$ and if $\eta_1 \ne \eta_2
\in \cF'$ and $s \in S \backslash (w_{\eta_1} \cup w_{\eta_2})$ 
then $\eta_1(s) \ne \eta_2(s)$.

\noindent
2) We say $\cF \subseteq {}^S X$ is $[\theta,J]$-free \when \, 
$J$ is an ideal on $S$ and for
every $\cF' \subseteq \cF$ of cardinality $< \theta$ there is a list $\langle
   \eta_\alpha:\alpha < \alpha_*\rangle$ of $\cF'$ such that: if
$\alpha < \alpha_*$ then the set $w_\alpha := \{s \in S:\eta_\alpha(s) \in
\{\eta_\beta(s):\beta < \alpha\}\}$ belongs to $J$.

\noindent
3) Let $\theta$-free or $(\theta)$-free mean $(\theta,J)$-free
\when \, $S \subseteq \Ord,J = J^{\bd}_S$.

\noindent
4) We say $\mu$ is 1-solvable \when \, $\mu$ is singular strong limit
   and there is a $\mu^+$-free family $\cF \subseteq {}^{\cf(\mu)}\mu$ of
cardinality $2^\mu$.

\noindent
5) We say $\mu$ is $(\theta,1)$-solvable \when \, above we weaken
 ``$\mu^+$-free" to ``$\theta$-free".

\noindent
6) We say $\cF \subseteq {}^S X$ is weakly ordinary \when \, each $\eta
   \in \cF$ is a one-to-one function.  We say $\cF \subseteq {}^\gamma
   \Ord$ is ordinary \when \, each $\eta \in \cF$ is an increasing function.
\end{definition}

\begin{claim}
\label{z53}
Assume $\theta > \partial$ and $\partial$ is regular, $J$ is an ideal 
on $\partial$ extending $[\partial]^{< \partial}$ 
and $\cF \subseteq {}^\partial\Ord$ 
\underline{and}\footnote{We can replace ``$< \partial$" by ``$\in J'$"
  when $J' \subseteq J$ is a $\partial$-complete ideal} $\eta \ne \nu
\in \cF \Rightarrow |\{i < \partial:\eta(i) \in \Rang(\nu)\}| < \partial$.

\noindent
1) We have $\cF$ is $(\theta,J)$-free \Iff \, $\cF$ is $[\theta,J]$-free.

\noindent
2) If every $\eta \in \cF$ is one-to-one, \then \, we can 
add in Definition \ref{y37}(2), $\eta_\alpha(s) \notin 
\{\eta_\beta(t):\beta < \alpha,t \in S\}$. 
\end{claim}

\begin{remark}
\label{z55}
1) We may consider only the case $i \ne j \Rightarrow 
\eta(i) \ne \nu(j)$ in \ref{y37}(1), \ref{a3}(6), \ref{a12}(1).

\noindent
2) Compare with \cite{Sh:g}, \cite{Sh:898}.

\noindent
3) Because of \ref{z53} the difference between $(\theta,J)$-free and
   $[\theta,J]$-free is not serious.  For $\bold k$-c.p. $\bold x$ see
Definition \ref{a6}; there we use only the latter version so do not write
   $[\theta,J]$. 
\end{remark}

\begin{PROOF}{\ref{z53}}
1) It is enough to prove for every $\cF \subseteq {}^\partial \Ord$
of cardinality $< \theta,\cF$ is $(\theta,J)$-free iff $\cF$ is
$[\theta,J]$-free.

First, if $\cF$ is $[\theta,J]$-free, then there is a sequence
$\langle \eta_\alpha:\alpha < \alpha_*\rangle$ enumerating $\cF$ as in
Definition \ref{y37}(2), i.e. $\alpha < \alpha_* \Rightarrow
w^1_\alpha := \{i < \partial:\eta_\alpha(i) \in \{\eta_\beta(i):\beta
< \alpha\}\} \in J$.
Define $w_\eta$ by $\eta = \eta_\alpha \Rightarrow w_\eta =
w^1_\alpha$; easily $\langle w_\eta:\eta \in \cF\rangle$ is as required
in Definition \ref{y37}(1).  

Second, if $\cF$ is $(\theta,J)$-free,
then there is $\langle w_\eta:\eta \in \cF\rangle$ which is as required
in Definition \ref{y37}(1).

Let $\langle \eta^1_\alpha:\alpha < \alpha_*\rangle$ list $\cF$ and by
induction on $n$ for each $\alpha$ we define $u_{\alpha,n}$ as follows:
\mn
\begin{enumerate}
\item[$(*)^1_\alpha$]  $(a) \quad u_{\alpha,0} = \{\alpha\}$
\sn
\item[${{}}$]  $(b) \quad u_{\alpha,n+1} = u_{\alpha,n} \cup \{\beta <
\alpha_*$: for some $i \in \partial \backslash w_\beta$ we have

\hskip25pt  $\eta_\beta(i) \in \{\eta_\gamma(i):\gamma \in u_{\alpha,n}\}\}$.
\end{enumerate}
\mn
Now
\mn
\begin{enumerate}
\item[$(*)^2_\alpha$]  $|u_{\alpha,n}| \le \partial$ and $u_{\alpha,n}
  \subseteq \alpha_*$.
\end{enumerate}
\mn
[Why?  Trivially $u_{\alpha,n} \subseteq \alpha_*$.  Also
$|u_{\alpha,0}| = 1 \le \partial$ and if $|u_{\alpha,n}| \le \partial$
then $|u_{\alpha,n+1}| \le |u_{\alpha,n}| + \sum\limits_{i < \partial}
\, \sum\limits_{\gamma \in u_{\alpha,n}} |\{\beta < \alpha_*:\beta$
satisfies $i \notin w_\beta \wedge \eta_\beta(i) = \eta_\gamma(i)\}|
= |u_{\alpha,n}| + \sum\limits_{i < \partial} \, 
\sum\limits_{\gamma \in u_{\alpha,n}} \, 1 \le \partial + \partial
\cdot \partial \cdot 1 = \partial$.]

We define $u_\alpha$ by induction on $\alpha < \alpha_*$ as follows:
$u_\alpha = \bigcup\limits_{n} u_{\alpha,n} \backslash
\bigcup\limits_{\beta < \alpha} u_\beta$, so $\langle u_\alpha:\alpha
< \alpha_*\rangle$ is a partition of $\alpha_*$ to sets each of
cardinality $\le \partial$, so we can 
let $\langle \beta_{\partial \alpha +i}:i < i_\alpha \le \partial\rangle$ list
$u_\alpha$.  Let $\cU = \{\partial \alpha +i:\alpha < \alpha_*,i <
i_\alpha$ and $\beta_{\partial \alpha +i} \notin \cup\{u_\gamma:\gamma <
\alpha\}\}$, so $\{\beta_\gamma:\gamma \in \cU\}$ lists $\alpha_*$
with no repetitions and 
easily $\langle \eta_{\beta_\zeta}:\zeta \in \cU\rangle$
is a list as required in Definition \ref{y37}(2).  That is, let $\beta
= \beta_{\partial \alpha + i} = \beta(\gamma_\alpha +i),i <
i_\alpha$.  So $\{i < \partial:\eta_\beta(j) \in
\{\eta_\gamma(j):\gamma \in \cU \cap \beta\}$ if $J$ is a
$\partial$-complete ideal, then the union of the following sets:
$w^2_\beta := \{j < \partial:\eta_\beta(j) \in \{\eta_\gamma(j):\gamma
\in \cU \cap \beta_{\partial,\alpha}\}$ and $w^2_{\beta,\iota} = \{j
< \partial:\eta_\beta(j) = \eta_{\partial \alpha + \iota}(j)\}$ for
$\iota < i$.  Now each of those sets belong to $J$ (Why?  $w^2_\beta$
by the choice of the $u_{\gamma,n}$'s and the $u_\gamma$'s;
$w^2_{\beta,\iota}$ as it is included in 
$w_{\eta_{\beta(\partial \alpha +i)}}$).  So if $J$ is a
$\partial$-complete ideal we are done, and if not, by $\bullet$ of the
assumptio nof the claim, $\iota < i \Rightarrow
|w^2_{\beta,\iota}|< \partial$, so recalling $\partial$ is regular
$\bigcup\limits_{\iota < i} w^2_{\beta,\iota}$ has cardinality
$< \partial$ hence belongs to $J$, so as $J$ is an ideal we are done.

Pedantically
$\langle \eta'_\gamma:\gamma < \otp(\cU)\rangle$ is such a list when we
define $\eta'_\gamma$ for $\gamma < \otp(\cU)$ by $\eta'_{\otp(\zeta \cap \cU)}
= \eta_{\beta_\zeta}$.

\noindent
2) Similarly to the ``Second" in the proof that \ref{y37}(1) holds,
except that $(*)^1_\alpha(b)$ is:
\mn
\begin{enumerate}
\item[$(b)'$]  $u_{\alpha,n+1} = u_{\alpha,n} \cup \{\beta <
  \alpha_*$: for some $i \in \partial \backslash w_\beta,\eta_\beta(i)
  = \{\eta_\gamma(j):\gamma \in u_{\alpha,n},j < \partial\}\}$.
\end{enumerate}
\end{PROOF}

\begin{question}
1) If $\mu$ is strong limit $\aleph_0 = \cf(\mu) < \mu$ (but not
necessarily $\mu \in \bold C$), can we get the
freeness results of \cite{Sh:1008}.

\noindent
2) In the cases we have can we strengthen the $\chi$-BB by having
$F:\Lambda_{\bold x} \rightarrow \chi$ and demand $\eta_m(i) \in F(\bar\eta
\upharpoonleft (m,<i))$?

\noindent
2A) Is this preserved by products?
\end{question}
\newpage

\section {Black Boxes} \label{Black}

We generalize the $\bold k$-dimensional black box from 
\cite{Sh:883}, where we deal with the special case when
$\ell < \bold k \Rightarrow \partial_\ell = \aleph_0$ because this
seems natural for Abelian groups; the black boxes earlier to
\cite{Sh:883} were for $\bold k=1$.  

But here for Abelian groups the most interesting cases are when
$\{\partial_\ell:\ell < \bold k\} \subseteq \{\aleph_0,\aleph_1\}$.
In the cases we prove existence, the $\bold k$-dimensional black box 
is the product of black boxes, i.e. the ones for $\bold k = 1$.

The main result is Theorem \ref{a51} telling us that there are 
$\bold k$-dimensional black boxes which are quite free.

The central notion here is a combinatorial parameters, those objects
$(\bold x)$ consist of the relevant finitely many cardinals
$(\langle \partial_\ell:\ell < \bold k\rangle)$ and sets $(\langle
S_\ell:\ell < \bold k\rangle)$ and a family $(\Lambda)$ of sequences
$\langle \eta_\ell:\ell < \bold k\rangle$ with $\eta_\ell$ a sequence of
length $\partial_\ell$ of members of $S_\ell$.  Such objects are used
in the construction of Abelian groups $G$.  The point is that, on the
one hand, the relevant (algebraic) freeness of the Abelian group $G$
is deduced from (set theoretic) freeness of $\bold x$, i.e. of
$\Lambda$, and on the other hand, e.g. $\Hom(G,\bbZ)=0$ is deduced by
using the $\bold x$ having a black box (which is used in the construction).  
See more in \ref{a5d}.

\begin{convention}  
\label{a0}
1) $\bar\partial$ will denote a sequence $\langle \partial_\ell:\ell
< \bold k\rangle$ of regular cardinals or just limit ordinals 
of length $\bold k \ge 1$ 
and then $\partial(\ell) = \partial_\ell$ but note that 
$k = \bold k -1$ was used in \cite{Sh:883}; a major case is
$\bar\partial$ is constant,
i.e. $\bigwedge\limits_{\ell} \partial_\ell = \partial$ for some $\partial$.

\noindent
2) Let $\bold x,\bold y,\bold z$ denote combinatorial parameters, see
Definition \ref{a6} below.
\end{convention}

\begin{notation}  
\label{a3}
0) Here $\bar S = \langle S_\ell:\ell < \bold k\rangle$ and
$\bar\partial = \langle \partial_\ell = \partial(\ell):\ell < \bold k\rangle$. 

\noindent
1) Let $\bar S^{[\bar\partial]} = \prod\limits_{\ell < \bold k}
   {}^{\partial(\ell)}(S_\ell)$
and $\bar S^{[\bar\partial,u]} = \prod\limits_{\ell \in u}
{}^{\partial(\ell)}S_\ell$ for $u \subseteq \{0,\dotsc,\bold k -1\}$
 and if each $S_\ell$ is a set of ordinals let 
$\bar S^{< \bar\partial>} = \{\bar\eta \in 
\bar S^{[\bar\partial]}$: each $\eta_\ell$ is increasing$\}$ and
   similarly $\bar S^{<\bar\partial,u>}$.

\noindent
2) If $\bar\eta \in \bar S^{[\bar\partial]},m < \bold k$ and $i < 
\partial_m$ then\footnote{It is sometimes natural to replace ``$i <
\partial_\ell$" by ``$i$ a subset of $\partial_\ell$ from some
family $\cP_\ell$ and $\eta'_\ell = \eta_\ell \rest i$ when $\ell = m$", 
say using $J^{\bd}_{\aleph_1} *
J^{\bd}_{\aleph_1}$ as in \cite{Sh:898}.  In \cite{Sh:883} this
version was used.} \, $\bar\eta \upharpoonleft (m,i) = 
\bar\eta \upharpoonleft_{\bold x}(m,i)$ is $\langle
\eta'_\ell:\ell < \bold k\rangle$ where $\eta'_\ell$ is
   $\eta_\ell$ when $\ell < \bold k \wedge \ell \ne m$ and is
$\eta_\ell \rest \{i\}$ if $\ell = m$; this is closed to but not the same as 
in\footnote{but if we use
  tree like $\Lambda \subseteq \bar S^{[\bar\partial]}$, see
  \ref{a3}(6), the difference is small, what we use there is called
  here $\bar\eta \upharpoonleft (m,=i)$.} \cite{Sh:883}.  
Also for $w \subseteq \partial_m,\bar\eta
\upharpoonleft (m,= w)$ is defined as $\langle \eta'_\ell:\ell <
\bold k\rangle$ where $\eta'_\ell = \eta_\ell$ if $\ell < \bold k \wedge
\ell \ne m$ and $\eta'_\ell = \eta_\ell \rest w$ if $\ell = m$.  
Let $\bar\eta \upharpoonleft (m) = 
\langle \eta_\ell:\ell \ne m,\ell < \bold k\rangle$.

\noindent
3) If $\Lambda \subseteq \bar S^{[\bar\partial]},m < \bold k$ and $i
   < \partial_m$ \then \, $\Lambda \upharpoonleft_{\bold x}(m,i) = \{\bar\eta
   \upharpoonleft (m,i):\bar\eta \in \Lambda\}$; we define similarly
   $\Lambda \upharpoonleft_{\bold x}(\eta,=w)$.

\noindent
4) If $\Lambda \subseteq \bar S^{[\bar\partial]},m < \bold k$ and $i
   \le \partial_m$ \then \, $\Lambda \upharpoonleft_{\bold x}(m,< i) 
= \cup\{\Lambda \upharpoonleft (m,i_1):i_1 < i\}$.

\noindent
5) $\Lambda_{\bold x,\in u} = \cup\{\Lambda_{\bold x}
   \upharpoonleft(m,i):m \in u,i < \partial_m\}$ for 
$u \subseteq \{0,\dotsc,\bold k-1\}$.  We may write ``$< m$" instead of
``$\in m$" when ``$u = \{0,\dotsc,m-1\}$" and let
$\Lambda_{\bold x,m} = \Lambda_{\bold x,\in \{m\}}$.

\noindent
6) We say $\Lambda \subseteq \bar S^{[\bar\partial]}$ is tree-like
\when \, $\bar\eta,\bar\nu \in \Lambda,\bar\eta \upharpoonleft (m,i) = \bar
\nu \upharpoonleft (m,j)$ implies $\eta_m \rest i = \nu_m \rest j$ so
in particular it implies $i=j$.

\noindent
7) We say $\Lambda \subseteq \bar S^{<\bar\partial>}$ is normal \when \,:
   if $\bar\eta,\bar\nu \in \Lambda,m < \bold k,i,j < \partial_m$ and
   $\eta_m(i) = \nu_m(j)$ then $i=j$, (hence each $\nu_m$ is one-to-one;
   this follows from being tree-like).
\end{notation}

\noindent
We now define in Definition \ref{a5} the standard $\bold x$, as it is
more transparent than the general case (in \ref{a6}) but will 
not use it as the ZFC-existence results are not standard; see
explanation after Definition \ref{a5}.  The main difference is that in
the general, (i.e. not necessarily standard) version, we have the
extra parameter $J_\ell$, ideal on $\partial_\ell$.
\begin{definition}  
\label{a5}
1) We say $\bold x$ is a standard $\bar\partial$-c.p. (combinatorial
$\bar\partial$-parameter) \when \, $\bold x = (\bold
   k,\bar\partial,\bar S,\Lambda) = (\bold k_{\bold x},
\bar\partial_{\bold x},\bar S_{\bold x},\Lambda_{\bold x})$ and it satisfies:
\mn
\begin{enumerate}
\item[$(a)$]  $\bold k \in \{1,2,\ldots\}$ and let $k = k_{\bold x} =
  \bold k-1$, (this to fit the notation in \cite{Sh:883})
\sn
\item[$(b)$]  $\bar\partial = \langle \partial_\ell:\ell < \bold
  k\rangle$ is a sequence of regular cardinals; so 
$\partial_\ell = \partial_{\bold x,\ell}$
\sn
\item[$(c)$]  $\bar S = \langle S_\ell:\ell < \bold k\rangle,S_\ell$ a
set of ordinals so $S_\ell = S_{\bold x,\ell}$
\sn
\item[$(d)$]  $\Lambda \subseteq \bar S^{[\bar\partial]} =
\prod\limits_{\ell < \bold k} {}^{\partial(\ell)}(S_\ell)$, see \ref{a3}(1). 
\end{enumerate}
\mn 
2) If $\ell < \bold k \Rightarrow \partial_\ell = \partial$ we may
write $\partial$ instead of $\bar\partial$ in $(\bold
k,\bar\partial,\bar S,\Lambda)$ and may say combinatorial
$(\partial,\bold k)$-parameter.  
 If $\ell < \bold k \Rightarrow
\partial_\ell = \aleph_0$ we may omit $\bar\partial$ and write
``$\bold x$ is a combinatorial $\bold k$-parameter".  If $\ell < \bold k
\Rightarrow S_\ell = S$ we may write $S$ instead of $\bar S$.  Also we
may write $\bold k(\bold x)$ for $\bold k_{\bold x}$.

\noindent
3) We say $\bold x$ (or $\Lambda$) is ordinary \when \, (each $S_\ell$
is a set of ordinals and) $\bar\eta \in \Lambda \Rightarrow$ each
   $\eta_\ell$ is increasing.  We say $\bold x$ (or $\Lambda$) is
weakly ordinary \when \, $\bar\eta \in \Lambda \wedge m < \ell
g(\bar\eta) \Rightarrow \eta_m$ is one-to-one.
We say $\bold x$ is disjoint \when \,
$\langle S_{\bold x,m}:m < \bold k\rangle$ is a sequence of pairwise
disjoint sets.   We say $\bold x$ is ordinarily full
\when \, it is ordinary and $\Lambda_{\bold x} = 
\{\langle \eta_\ell:\ell < \bold k\rangle:\eta_\ell 
\in {}^{\partial(\ell)}(S_\ell)$ is increasing
for $\ell < \bold k\}$.  Similarly for weakly ordinary.

\noindent
4) We say $\bold y$ is a permutation of $\bold x$ \when \, for some
permutation $\pi$ of $\{0,\dotsc,\bold k -1\}$ we have $m < k
\Rightarrow \partial_{\bold x,m} = \partial_{\bold y,\pi(m)}$ and $m < k
\Rightarrow S_{\bold x,m} = S_{\bold y,\pi(m)}$ and $\Lambda_{\bold y} 
= \{\langle \eta_{\pi(m)}:m < \bold k\rangle:\langle \eta_m:m <
   \bold k\rangle \in \Lambda_{\bold x}\}$.

\noindent
5) We say $\bar\pi$ is an isomorphism from $\bold x$ onto $\bold y$
\when \,:
\mn
\begin{enumerate}
\item[$(a)$]  $\bold k_{\bold y} = \bold k_{\bold x}$ call it $\bold k$
\sn
\item[$(b)$]  $\bar\pi = \langle \pi_m:m \le \bold k \rangle$
\sn
\item[$(c)$]   $\pi_{\bold k}$ is a permutation of $\{0,\dotsc,\bold k
-1\}$
\sn
\item[$(d)$]  $\partial_{\bold x,m} = \partial_{\bold y,\pi_{\bold
      k}(m)}$ for $m < \bold k$
\sn
\item[$(e)$]  $\pi_m$ is a one-to-one function from $S_{\bold x,m}$
onto $S_{\bold y,\pi_{\bold k}(m)}$ for $m < \bold k$
\sn
\item[$(f)$]  $\langle \nu_m:m < \bold k \rangle \in \Lambda_{\bold
  y}$ \Iff \, for some 
$\langle \eta_m:m < \bold k\rangle \in \Lambda_{\bold x}$ we have
  $\nu_{\pi_{\bold k}(m)} = 
\langle \pi_m(\eta_m(i)):i < \partial_{\bold x,m}\rangle$.
\end{enumerate}
\end{definition}

\begin{discussion}
\label{a5d}
It may be helpful to the reader to indicate how such $\bold x$ helps
to construct, e.g. Abelian groups, for simplicity each $\partial_\ell$
is $\aleph_0$ (this suffices for constructing an $\aleph_{\omega \cdot n}$-free
$G$, which already is new).

First, let $\langle x_{\bar\eta}:\bar\eta \in \Lambda_{\bold x}
\upharpoonleft (m,i)$ for some $m$ and $i\rangle$ freely generate an
Abelian group $G_0$ and for such $\bar\eta \in \Lambda_{\bold x}$ 
we add elements
like $y_{\bar\eta,n} = \Sigma\{(\frac{i!}{n!}) 
(x_{\bar\eta \upharpoonleft (m,i)} + a_{\bar\eta,m} x_{\nu_{\bar\eta}}):
m < \bold k_{\bold x},i$ finite $\ge n\}$ for some $\nu_{\bar\eta} \in
\Lambda_{\bold x_1},n \in \bold k$ and $a_{\bar\eta,m} \in \bbZ$
getting $G_1 \supseteq G_0$.  Now on the one hand,
we like $G_1$ to be $\theta$-free and on the other hand, we like it,
e.g. to have no non-zero homomorphism into $\bbZ$.  For the second
task, we need a BB (black box) property, that is, for each possible
$\nu_{\bar\eta}$ to have for each
$\bar\eta \in \Lambda$, a homomorphism $h_{\bar\eta}$ from
$\Sigma\{\bbZ x_{\bar\eta \upharpoonleft (m,i)}:m < \bold k,i$
finite$\} \oplus \bbZ x_{\nu_{\bar\eta}}$ into $\bbZ$ 
such that $\{h_{\bar\eta}:\bar\eta \in
\Lambda\}$ is dense (or see Definition \ref{a9}(1) called
$\bar\alpha_{\bar\eta}$ there) and choose the $a_{\bar\eta,n}$'s to ``defeat
$h_{\bar\eta}$", i.e. to ensure no $h \in \Hom(G_1,\bbZ)$ extends
$h_{\bar\eta}$.

Concerning the first task, we like to ensure $\bold x$ is
$\theta$-free meaning that for any $\Lambda \subseteq 
\Lambda_{\bold x}$ of cardinality $< \theta$ we can list its members
as $\langle \bar\eta_\alpha:\alpha < \alpha_*\rangle$ such that for
every $\alpha$ for some $m,i$ we have $j \ge i \Rightarrow
\bar\eta_\alpha \upharpoonleft (m,j) \notin \{\bar\eta_\beta
\upharpoonleft (m,j):\beta < \alpha\}$, see Definition \ref{a9}(3).

In the existence proofs the novel main point is getting enough
freeness relying on the pcf theory, i.e. in \S1 we prove the
existence of suitable $\cp \, \bold x$.
\end{discussion}

\begin{definition}
\label{a6}
1) We say $\bold x$ is a $\bar\partial$-c.p. 
(combinatorial $\bar \partial$-parameter) \when \, $\bold x 
= (\bold k,\bar\partial,\bar S,\Lambda,\bar J) = 
(\bold k_{\bold x},\bar\partial_{\bold x},\bar S_{\bold
   x},\Lambda_{\bold x},\bar J_{\bold x})$ and
   they satisfy: (in the standard case $J_m = \{w \subseteq
   \partial_\ell:w$ is bounded$\}$):
\mn
\begin{enumerate}
\item[$(a)$]  $\bar\partial = \langle \partial_m:m < \bold k\rangle$,
  a sequence of limit ordinals
\sn
\item[$(b)$]  $\bar J = \langle J_m:m < \bold k\rangle$
\sn
\item[$(c)$]  $J_m$ is an ideal on $\partial_m$ 
\sn
\item[$(d)$]  $\bar S = \langle S_m:m < \bold k\rangle,S_m$ a set of
  ordinals if not said otherwise
\sn
\item[$(e)$]  $\Lambda  \subseteq \bar S^{[\bar\partial]}$.
\end{enumerate}
\mn
2) We adopt the conventions and definitions in \ref{a5}(2)-(5).
\end{definition}

\begin{convention}
\label{a7}
1) If $\bold x$ is clear from the context we may write $\bold k$ for 
$\bold k(\bold x)$, $k$ for $k(\bold x)$ and $S,\Lambda,\bar J$ 
instead of $\bold k_{\bold x},k_{\bold x},
\bar S_{\bold x},\Lambda_{\bold x},\bar J_{\bold x}$ respectively.

\noindent
2) If not said otherwise $\bold x$ is weakly ordinary, see \ref{a5}(3).
\end{convention}

\begin{definition}
\label{a9}
Assume $\bold x$ is a $\bar\partial$-c.p.

\noindent
1)  We say $\bold x$ has $(\bar \chi,\bold k,1)$-Black Box or
$\bar\chi$-pre-black box \when \, some $\bar \alpha$ is a $(\bar\chi,\bold
k,1)$-black box for $\bold x$
or $(\bold x,\bar \chi)$-pre-black box, which means:
\mn
\begin{enumerate}
\item[$(a)$]   $\bar \chi = \langle \chi_m:m < \bold k_{\bold
    x}\rangle$ is a sequence of cardinals
\sn
\item[$(b)$]   $\bar \alpha = \langle \bar \alpha_{\bar \eta}:\bar
\eta \in \Lambda_{\bold x}\rangle$
\sn
\item[$(c)$]   $\bar \alpha_{\bar \eta} = \langle \alpha_{\bar \eta,m,i}:m
< \bold k_{\bold x},i < \partial_m \rangle$ and 
$\alpha_{\bar\eta,m,i} < \chi_m$
\sn
\item[$(d)$]   if $h_{m}:\Lambda_{\bold x,m} \rightarrow \chi_m$ for
$m < \bold k_{\bold x}$ recalling \ref{a3}(5) \then \, for 
some $\bar \eta \in \Lambda_{\bold x}$ we
have: $m < \bold k_{\bold x} \wedge i < \partial_m \Rightarrow h_m(\bar \eta
\upharpoonleft \langle m,i \rangle) = \alpha_{\bar \eta,m,i}$.
\end{enumerate}
\mn
2)  For $\Lambda \subseteq \Lambda_{\bold x}$ we define $\bold x
\rest \Lambda$ naturally as $(\bold k_{\bold x},\bar\partial_{\bold x},\bar
S_{\bold x},\Lambda,\bar J)$.

\noindent
3) We may write $\bar\alpha$ as $\bold b$, a function with domain
$\{(\bar\eta,m,i):\bar\eta \in \Lambda_{\bold x},m < \bold k,i <
\partial_m\}$ such that $\bold b_{\bar\eta}(m,i) = 
\bold b(\bar\eta,m,i) = \alpha_{\bar\eta,m,i}$.
We may replace $\bar \chi$ by $\chi$ if $\bar \chi =\langle
\chi:\ell < \bold k_{\bold x}\rangle$ or by $\bar C = \langle
C_\ell:\ell < \bold k\rangle$ when $|C_\ell| = \chi_\ell$ and we demand
$\Rang(h_\ell) \subseteq C_\ell$.  
We may replace $\bold x$ by $\Lambda = \Lambda_{\bold x}$ (so say 
$\bar \alpha$ is a $(\Lambda,\bar \chi)$-pre-black box).

\noindent
4) Omitting the ``pre" in part (1) means that there is a partition
$\bar\Lambda = \langle \Lambda_\alpha:\alpha < 
|\Lambda_{\bold x}|\rangle$ of $\Lambda_{\bold x}$
such that each $\bold x \rest \Lambda_\alpha$ has 
$\bar\chi$-pre-black box and some $\langle \bar\nu_\alpha:\alpha <
|\Lambda_{\bold x}|\rangle$ witnesses it which means that:
\mn
\begin{enumerate}
\item[(a)]  $\{\bar\nu_\alpha:\alpha < |\Lambda_{\bold x}|\} =
  \Lambda_{\bold x}$
\sn
\item[(b)]  letting $\mu$ be maximal such that $(\forall \ell < \bold
  k)2^{< \mu} \le \chi_\ell$ we have $\alpha < \beta < \alpha + \mu
  \Rightarrow \bar\nu_\alpha = \bar\nu_\beta$
\sn
\item[(c)]  if $\alpha \le \beta < |\Lambda_{\bold x}|$ and $\bar\eta
  \in \Lambda_\beta$ then $\nu_{\alpha,\bold k-1} < \eta_{\bold k-1}
  \mod J_{\bold x,\bold k-1}$.
\end{enumerate}
\mn
5) We may write BB instead of black box.

\noindent
6) We say $\bold x$ essentially has a $\bar\chi$-black box when some
$(\bar\Lambda,\bold n)$ witnessing it which means\footnote{See the
  proof of 2.10(2).}:
\mn
\begin{enumerate}
\item[(a)]  $\bar\Lambda = \langle \Lambda_\alpha:\alpha < 
|\Lambda_{\bold x}|\rangle$ is a sequence of pairwise disjoint subsets
of $\Lambda_{\bold x}$
\sn
\item[(b)]  $\bold x \rest \Lambda_\alpha$ has a
  $\bar\chi$-pre-black-box
\sn
\item[(c)]  $\bold n = \langle \bar\nu_\alpha:\alpha < 
|\Lambda_{\bold x}|\rangle$
\sn
\item[(d)]  if $\bar\nu \in \Lambda_{\bold x}$ then $\bar\nu \in
  \{\bar\nu_\alpha:\alpha < |\Lambda_{\bold x}|\}$
\sn
\item[(e)]  if $\mu = \sup\{\mu:2^\mu < \min\{|S_{\bold x,\ell}|:
\ell < \bold k_{\bold x}\}$ then $\alpha < \beta < \alpha + \mu
\Rightarrow \bar\nu_\alpha = \bar\nu_\beta$ and $\alpha \le \beta <
\lambda \wedge \bar\eta \in \Lambda_{\bold x_\alpha} \Rightarrow
\nu_{\alpha,\bold k-1} <_{J_{\bold x,\ell}} \eta_{\bold k-1}$ (we can
use variant of this) but this suffices presently.
\end{enumerate}
\end{definition}

\noindent
We shall use freely
\begin{observation}
\label{a9d}
If (A) then (B):
\mn
\begin{enumerate}
\item[(A)]  $\bold x$ is a $\bar\partial-\cp$ and $(\bar\Lambda,\bold
  n)$ witness $\bold x$ essentially have $\bar\chi$-black box
\sn
\item[(B)] there is $\bold y = \bold x \rest \Lambda$ for some
  $\Lambda \subseteq \Lambda_{\bold x}$ which has $\bar\chi$-black box
\end{enumerate}
\end{observation}

\begin{PROOF}{\ref{a9d}}
We choose $\Omega_n \subseteq \Lambda_{\bold x}$ by induction on $n$
by:
\mn
\begin{enumerate}
\item[$(*)$]
\begin{enumerate}
\item[(a)]  if $n=0$ then $\Omega_0 = \Lambda_0 \cup \{\bar\nu_0\}$
\sn
\item[(b)]  if $n=m+1$ then $\Omega_n = \cup\{\Lambda_\alpha:\alpha <
  \lambda = |\Lambda_{\bold x}|$ and $\bar\nu_\alpha \in \Omega_m\} \cup \Omega_m$.
\end{enumerate}
\end{enumerate}
\mn
Now $\bold x \rest \bigcup\limits_{n} \Omega_n$ is as required.
\end{PROOF}

\begin{observation}
\label{a10}
1) In Definition \ref{a9}(4) we may use $\Lambda_{\bold x}$ as the
index set of $\bar\Lambda$ instead of $|\Lambda_{\bold x}|$.

\noindent
2) If $\bold x$ is a $\bar\partial$-c.p., $\bar\chi = \langle
\chi_\ell:\ell < \bold k_{\bold x}\rangle$ and $|\Lambda_{\bold x}| =
\max\{\chi_\ell:\ell < \bold k_{\bold x}\}$ then $\bold x$ has a
$\bar\chi$-black box \Iff \, $\bold x$ has a $\bar\chi$-pre-black box.
\end{observation}

\begin{remark}
\label{a10}
Concerning the variants below our aim is to have ``$\bold x$ is
$(\theta)$-free", but to get it we use the other versions.
\end{remark}

\begin{definition}
\label{a12}
1) For $\Lambda_* \subseteq \bar S^{[\bar\partial]}$, we 
say ``$\bold x$ is $(\theta,u)$-free over $\Lambda_*$"
\when\footnote{so if $k_{\bold x}$ is 1 then ``$\bold x$ is
$(\theta,\{0\})$-free" has closer meaing to ``$\{\eta:\langle \eta\rangle \in
\Lambda_{\bold x}\}$ is $[\theta,J_{\bold x,0}]$-free" than to
$(\theta,J_{\bold x,0})$-free, see Definition \ref{z53}} \, $\bold x$ is 
weakly ordinary, $u \subseteq 
\{0,\dotsc,\bold k_{\bold x}-1\}$ and
for every $\Lambda \subseteq \Lambda_{\bold x} \backslash \Lambda_*$
of cardinality $< \theta$ there is a list $\langle \bar\eta_\alpha:\alpha <
\alpha_*\rangle$ of $\Lambda$ such that: for every $\alpha$ for some
$m \in u$ and $w \in J_{\bold x,m}$ we\footnote{If $\Lambda_{\bold x}$ 
is normal, we can restrict ourselves to $i=j$ and this is the usual case.}
 have $\bar\nu \in \{\bar\eta_\beta:\beta
< \alpha\} \cup \Lambda_* \wedge \bar\nu \upharpoonleft (m) = \bar\eta_\alpha
\upharpoonleft(m) \wedge j < \partial_{\bold x,m} \wedge 
i \in \partial_{\bold x,m} \backslash w 
\Rightarrow \nu_m(j) \ne \eta_{\alpha,m}(i)$.

\noindent
2) If $\theta > |\Lambda_{\bold x}|$ we may (in part (1)) write
$(\infty,u)$-free or $u$-free; we may omit ``over $\Lambda_*$" when 
$\Lambda_* =\emptyset$.

\noindent
3) If $u = \{0,\dotsc,\bold k -1\}$ we may omit it.

\noindent
4) Suppose we are given cardinals $\theta_1 \le \theta_2$, combinatorial
$\bar\partial$-parameter $\bold x$ and $\Lambda_*$ (usually $\subseteq
\Lambda_{\bold x}$) and $u \subseteq \{0,\dotsc,\bold k_{\bold x}-1\}$.

We say $\bold x$ is $(\theta_2,\theta_1,u,k)$-free over $\Lambda_*$
\when \,:
\mn
\begin{enumerate}
\item[$(a)$]   $\theta_2 \ge \theta_1 \ge 1$
\sn
\item[$(b)$]  $1 \le k \le \bold k_{\bold x}$, if $k=1$ we may omit it.
\sn
\item[$(c)$]  $u \subseteq \{0,\dotsc,\bold k_{\bold x}-1\}$ has $\ge
  k$ members
\sn
\item[$(d)$]  for every $\Lambda \subseteq \Lambda_{\bold x}
  \backslash \Lambda_*$ of cardinality $< \theta_2$ there is a witness
  $(\bar\Lambda,g,\bar h)$ which means
\sn
\begin{enumerate}
\item[$(\alpha)$]  $\bar\Lambda = \langle \Lambda_\gamma:\gamma <
  \gamma(*)\rangle$ is a partition of $\Lambda$ to sets each of
  cardinality $< \theta_1$, so $\gamma(*)$ is an
  ordinal $< \theta_2$
\sn
\item[$(\beta)$]  $g:\gamma(*) \rightarrow [u]^k$; when $k=1$ we usually use
$g':\gamma(*) \rightarrow u$ where $g(\gamma) = \{g'(\gamma)\}$ for
  $\gamma < \gamma(*)$ or even use $g'':\Lambda \rightarrow [u]^1$
  where $g''(\bar\eta) = g'(\gamma)$ when $\bar\eta \in
  \Lambda_\gamma$.   Occasionally (when the meaning of $\bar\eta_\beta$ is
clear we may write $g(\bar\eta_\beta)$ or
  $g'(\bar\eta_\beta)$ instead of $g(\beta)$ and $g'(\beta)$ (so we
  consider $\lambda$ as the domain of $g,g'$ instead of $\gamma(*)$)
\sn
\item[$(\gamma)$]  $\bar\eta,\bar\nu \in \Lambda_\gamma \wedge 
m \in (\bold k_{\bold x} \backslash g(\gamma)) 
\Rightarrow \eta_m = \nu_m$
\sn
\item[$(\delta)$]  $\bar h = \langle h_m:m \in u\rangle$
\sn
\item[$(\varepsilon)$]  $h_m:\Lambda \rightarrow J_m$, really just $h_m
  \rest \{\bar\eta \in \Lambda$: if $\gamma < \gamma(*)$ and $\bar\eta \in
  \Lambda_\gamma$ then $m \in g(\gamma)\}$ matters.  Here again, 
we may write $h_m(\beta)$ instead of $h_m(\bar\eta_\beta)$
\sn
\item[$(\zeta)$]  if $\bar\eta \in \Lambda_\beta$ and $m \in
  g(\beta)$ and $\bar\nu \in \cup\{\Lambda_\alpha:\alpha < \beta\}
  \cup \Lambda_*$ and $\bar\nu \upharpoonleft (m,=\emptyset) = \bar\eta 
\upharpoonleft (m,=\emptyset)$ then $i \in \partial_m \backslash 
h_m(\bar\eta) \Rightarrow \eta_m(i) \ne \nu_m(i)$.
\end{enumerate}
\end{enumerate}
\mn
5) In (4) if $\theta_2 > |\Lambda_{\bold x}|$ we may write
$(\infty,\theta_1,u,k)$-free; we may omit 
$\Lambda_*$ if $\Lambda_* = \emptyset$ and if $k=1$ we may omit $k$.

\noindent
6) We say $\bold x$ is $(\theta,u)$-free over $\Lambda_*$ respecting
$\bar\Lambda$ when (so we may write $\bold k$ instead of $u = \{\ell:\ell
< \bold k\}$ and $\theta$-free instead of $(\theta,\{\ell:\ell < \bold
k\})$: $\bar\Lambda = \langle \Lambda_{\bar\nu}:\bar\nu
\in \Lambda_{\bold x}\rangle,\Lambda_{\bar\nu} \subseteq
\Lambda_{\bold x}$, and for every $\Lambda \subseteq \Lambda_{\bold x}
\backslash \Lambda_*$ of cardinality $< \theta$ there is a list
$\langle \bar\eta_\alpha:\alpha < \alpha_*\rangle$ of $\Lambda$ such
that:
\mn
\begin{enumerate}
\item[$\bullet_1$]  if $\bar\eta_\alpha \in \Lambda_{\bar\nu}$ so $\bar\nu
\in \Lambda_{\bold x}$ then $\bar\nu \in \{\bar\eta_\beta:\beta <
  \alpha\} \cup \Lambda_*$
\sn
\item[$\bullet_2$]  for every $\alpha < \alpha_*$ for some $m \in u$
  and $w \in J_{\bold x,m}$ we have $\bar\nu \in \{\bar\eta_\beta:\beta <
  \alpha\} \cup \Lambda_* \cap \bar\nu \upharpoonleft (m) = 
\bar\eta_\alpha \upharpoonleft m \wedge j \in \partial_{\bold x,m}
  \backslash w \wedge i < \partial_{\bold x,m} \Rightarrow \nu_m(i)
  \ne \eta_m(j)$.
\end{enumerate}
\mn
7) For $\bold x,\theta_1,\theta_2,\Lambda_*,u$ as in Definition
\ref{a12}(4) and a sequence $\bar\Lambda^* = \langle
\Lambda^*_{\bar\rho}:\bar\rho \in \Lambda_{\bold x}\rangle$ of subsets
of $\Lambda_{\bold x}$ we say $\bold x$ is
$(\theta_2,\theta_1,u,k)$-free over $\Lambda_*$ respecting
$\bar\Lambda^*$, when clauses (a)-(d) of Definition \ref{a12}(4) hold
and we add to clause (d)
\mn
\begin{enumerate}
\item[$(\eta)$]  if $\bar\eta \in \Lambda_\alpha$ and $\bar\eta \in
  \Lambda^*_{\bar\rho}$ then $\bar\rho \in \cup\{\Lambda_\beta:\beta <
  \alpha\} \cup \Lambda_*$.
\end{enumerate}
\end{definition}

\begin{claim}
\label{a15}
Assume $\bold x$ is a $\bar\partial$-c.p. and $u \subseteq
\{0,\dotsc,\bold k_{\bold x}-1\}$ is not empty.

\noindent
1) $\bold x$ is $(\theta_2,2,u,1)$-free over $\Lambda_*$ \Iff \, $\bold x$ is
$(\theta_2,u)$-free over $\Lambda_*,\theta_2 \ge 2$.

\noindent
2) If $\partial > \max\{\partial_\ell:\ell < \bold k_{\bold x}\},\bold
   x$ is $(\theta,\partial,u)$-free over $\Lambda_*$ 
and for each $\ell \in u,\bold x$ is 
$(\partial,2,\{\ell\})$-free \then \, $\bold x$ is $(\theta,2,u)$-free
over $\Lambda_*$ (equivalently $(\theta,u)$ free over $\Lambda_*$).
\end{claim}

\begin{PROOF}{\ref{a15}}
Should be clear but we elaborate.

\noindent
1) It is enough to deal with the case $|\Lambda_{\bold x} \backslash
\Lambda_*| < \theta_2$.  First, assume $\theta_2 \ge 2$ and $\bold x$
is $(\theta_2,u)$-free over $\Lambda_*$, let $\langle
\bar\eta_\alpha:\alpha < \alpha_*\rangle$ listing $\Lambda_{\bold x}
\backslash \Lambda_*$ be as in Definition \ref{a12}(1).  Let
$\Lambda_\alpha = \{\bar\eta_\alpha\}$ for $\alpha < \alpha_*$ and
define $g':\alpha_* \rightarrow u$ by $g'(\alpha) =$ the minimal $m
\in u$ such that for some $w \in J_m$ the condition in Definition
\ref{a12}(1) holds.  By the assumption that $\bold x$ is
$(\theta_2,u)$-free over $\Lambda_*,g'$ is well defined.  Let 
$g:\alpha_* \rightarrow
[u]^1$ be $g(\alpha) = \{g'(\alpha)\}$.  Also we define $h_m:\alpha_*
\rightarrow J_m$ for $m \in u$ such that: if $\alpha < \alpha_*$ and
$m = g'(\alpha)$ \then \, $h_m(\alpha)$ is any $w \in J_m$ such that
the condition in Definition \ref{a12}(1) holds.   Now clearly in
Definition \ref{a12}(4), clause (a) holds (letting $\theta_1=2$ as
$\theta_2 \ge 2 = \theta_1$), clause (b) holds as $k=1 \in 
[1,\bold k_{\bold x}]$ and clause (c) is obvious.  
We shall check clauses $(d)(\alpha)-(\zeta)$
hence finishing proving the ``if" implication.

Let $\gamma(*) = \alpha_*$ and $\bar\Lambda = \langle \Lambda_\alpha:\alpha <
\alpha_* \rangle$.  This definition takes care of $(d)(\alpha)$ and
the above definition of $g,g'$ ensures $(d)(\beta)$.  Clause
$(d)(\gamma)$ is immediate since each $\Lambda_\alpha$ is a
singleton.  Clauses $(d)(\delta),(d)(\varp)$ follow from the
definition of the $h_m$'s.  Finally, clause $(d)(\zeta)$ follows from
Definition \ref{a12}(1).

Second, assume $\bold x$ is $(\theta_2,2,u,1)$-free and let
$(\bar\Lambda,g,\bar h)$ witness this so $\theta_1=2$; 
note that $\theta_2 \ge 2$, since $\theta_1=2$ and $\theta_2 \ge
\theta_1$ by Definition \ref{a12}(4)(a).
So $\bar\Lambda =
\langle \Lambda_\alpha:\alpha < \alpha_*\rangle$ and $\bar h = \langle
h_m:m \in u\rangle$ and $g:\alpha_* \rightarrow [u]^1$, so for some
function $g':\alpha_* \rightarrow u$ we have $\alpha < \alpha_*
\Rightarrow g(\alpha) = \{g'(\alpha)\}$.  As $|\Lambda_\alpha| <
\theta_2 =2$ we have $|\Lambda_\alpha| \le 1$; \wilog \,
$\bigwedge\limits_{\alpha} \Lambda_\alpha \ne \emptyset$ hence there
is a unique $\bar\eta_\alpha \in \Lambda_{\bold x} \backslash \Lambda_*$
such that $\Lambda_\alpha = \{\bar\eta_\alpha\}$.  So $\langle
\bar\eta_\alpha:\alpha < \alpha_*\rangle$ lists $\Lambda_{\bold x}
\backslash \Lambda_*$, and it suffices to check that for every $\alpha
< \alpha_*$ the condition in Definition \ref{a12}(1) holds.  We choose
$m=g'(\alpha)$ so $m \in u$ and we choose $w = h_m(\alpha)$ so $w \in
J_m$ indeed, and the condition there holds for $m,w$ by clause
$(d)(\zeta)$ of Definition \ref{a12}(4) as $\Lambda_\alpha = 
\{\bar\eta_\alpha\},\beta < \alpha \Rightarrow
\Lambda_\beta = \{\bar\eta_\beta\}$.

\noindent
2) As $\bold x$ is $(\theta,\partial,u)$-free over $\Lambda_*$ there
is a triple $(\bar\Lambda^*,g^*,\bar h^*)$ witnessing it, as in
Definition \ref{a10}(4) and let
$\bar\Lambda^* = \langle \Lambda^*_\alpha:\alpha < \alpha_*\rangle$
and $\bar h^* = \langle h^*_m:m \in u\rangle$.  For each $\ell \in u$ and
$\alpha < \alpha_*$ we know that $\bold x$ is
$(\partial,2,\{\ell\})$-free and $\Lambda^*_\alpha$ is a subset of
$\Lambda_{\bold x} \backslash \Lambda_*$ of cardinality $< \partial$
hence there is a triple $(\bar\Lambda_\alpha,g_\alpha,\bar h_\alpha)$
witnessing it, let $\bar\Lambda_\alpha = \langle
\Lambda_{\alpha,\beta}:\beta < \beta_\alpha\rangle$ and so
$|\Lambda_{\alpha,\beta}| < 2$ and \wilog \, $\Lambda_{\alpha,\beta}
\ne \emptyset$, so let $\Lambda_{\alpha,\beta} =
\{\bar\eta_{\alpha,\beta}\}$ and (as $k=1$, see end of \ref{a12}(5))
$g_\alpha(\beta) = \{g'_\alpha(\beta)\}$, where
$g'_\alpha:\beta_\alpha \rightarrow u$ and let $\bar h_\alpha =
\langle h_{\alpha,m}:m \in u\rangle$.

Let $\gamma_\alpha = \Sigma\{\beta_{\alpha_1}:\alpha_1 < \alpha\}$ for
$\alpha < \alpha_*$, so clearly
$\langle \gamma_\alpha:\alpha \le \alpha_*\rangle$ is increasing
continuous and $\gamma_0 = 0$ and let $\gamma_* = \gamma_{\alpha_*}$;
we define $\bar\eta_\gamma$ for $\gamma < \gamma_*$ by: if $\gamma =
\gamma_\alpha + \beta,\beta < \beta_\alpha$ then we let
$\bar\eta_\gamma = \bar\eta_{\alpha,\beta}$.  Also let $g':\gamma_*
\rightarrow [u]^1$ be defined by $g' \rest [\gamma_\alpha,\gamma_{\alpha
  +1})$ is constantly $\{g^*(\alpha)\}$, let $\bar\Lambda = \langle
\Lambda_\gamma:\gamma < \gamma_*\rangle$ where $\Lambda_\gamma =
\{\bar\eta_\gamma\}$ and let $\bar h = \langle h_m:m \in
u\rangle,h_m:\gamma_* \rightarrow J_m$ be $h_m(\gamma_\alpha + \beta)
= h_{\alpha,m}(\beta)$ if $\alpha < \alpha_*,\beta < \beta_\alpha$.
So it is enough to check that $(\bar\Lambda,g',\bar h)$ witnesses
$\Lambda_{\bold x}$ is $(\theta,2,u)$-free over $\Lambda_*$.  E.g. why
clause $(\zeta)$ of Definition \ref{a12}(d) holds.

Let $\bar\eta \in \Lambda_\gamma,m \in g'(\gamma)$ and $\bar\nu \in
\cup\{\Lambda_\alpha:\alpha < \gamma\} \cup \Lambda_*$.  So
$\bar\eta = \bar\eta_\gamma$ and one of the following cases occur,
letting $\gamma = \gamma_\alpha + \beta,\beta < \beta_\alpha$.
\medskip

\noindent
\underline{Case 1}:  $\bar\nu \in \cup\{\Lambda^*_{\alpha'}:\alpha' <
\alpha\} \cup \Lambda_*$

Use ``$(\bar\Lambda^*,g^*,\bar h^*)$ witness $\Lambda_{\bold x}$ is
$(\theta,\partial,u)$-free over $\Lambda_*$".
\medskip

\noindent
\underline{Case 2}:  $\bar\nu \in \Lambda^*_\alpha$

Use ``$(\bar\Lambda_\alpha,g_\alpha,\bar h_\alpha)$ witness 
$\Lambda_\alpha$ is $(\partial,2,\{\ell\})$-free" for $\ell = g^*(\alpha)$.
\end{PROOF}

\begin{definition}  
\label{a18}
We say $(\bold x,\bar\Lambda)$ witness $\BB^3_{\bold
  k}(\lambda,\Theta,\bar\chi,\bar\partial)$ \when \,:
\mn
\begin{enumerate}
\item[$(a)$]  $\bold x$ is a $\bar\partial$-c.p. 
with $|\Lambda_{\bold x}| = \lambda$ and $\bold k = \bold k_{\bold
  x}$, i.e. $= \ell g(\bar\partial)$
\sn
\item[$(b)$]  $\bar\Lambda = \langle \Lambda_{\bar\nu}:\bar\nu \in
  \Lambda_{\bold x}\rangle$ is a sequence\footnote{In \cite{Sh:883} we
  use $\Lambda_{\bold x,<k}$ as index set which if $k=1$ may have
  smaller cardinality; so far not a significant difference.} of pairwise
  disjoint subsets of $\Lambda_{\bold x}$
\sn
\item[$(c)$]  $\bold x \rest \Lambda_{\bar\nu}$ has
  $\bar\chi$-pre-black box for every $\bar\nu \in \Lambda_{\bold x}$
\sn
\sn
\item[$(d)$]  $\Theta$ is a collection of cardinals and pairs of
  cardinals
\sn
\item[$(e)$]  if $\theta \in \Theta$ \then \, $\bold x$ is 
  $(\theta,\bold k)$-free respecting $\bar\Lambda$, see \ref{a12}(6) 
which means that in the list
$\langle \bar\eta_\alpha:\alpha < \alpha_*\rangle$ in Definition
  \ref{a12}(1), we have 
$\bar\eta_\alpha \in \Lambda_{\bar\nu} \Rightarrow \bar\nu
 \in \{\bar\eta_\beta:\beta < \alpha\}$
\sn
\item[$(f)$]  if $(\theta_2,\theta_1) \in \Theta$ then $\bold x$ is
  $(\theta_2,\theta_1,\bold k,1)$-free 
respecting $\bar\Lambda$, see \ref{a12}(7).
\end{enumerate}
\end{definition}

\begin{remark}
\label{a19}
Note that in Definition \ref{a18} necessarily we have 
$\Sigma\{\chi_\ell:\ell < \bold k\} \le |\Lambda_{\bold x}|$.
\end{remark}

\noindent
Clearly
\begin{claim}
\label{a20}
Assume $\mu$ is strong limit $> \cf(\mu) = \partial,\cF \subseteq
{}^\partial \mu$ has cardinality $\lambda = 2^\mu$ and $\cF$ is
$\theta$-free (i.e. $(\theta,J^{\bd}_\partial)$-free)), moreover,
$[\theta,J^{\bd}_\partial]$-free and weakly ordinary, see \ref{y37}(1),(2),(6).

\Then\, there is a $\langle \partial \rangle$-c.p. $\bold x$ with
$\Lambda_{\bold x} = \cF$ which is $\theta$-free and has the
$\lambda$-BB (i.e. $(\langle \lambda \rangle,1,1)$-BB).
\end{claim}

\begin{PROOF}{\ref{a20}}
The point is that the set of functions from ${}^{\partial >}\mu$ to
$\lambda$ has cardinality $\lambda = |\cF|$, see more in 
\cite[2.2=Ld.6]{Sh:898}.
\end{PROOF}

\begin{claim}
\label{a25}
1) Assume $\bold x$ is a $ \bold k$-c.p., 
$\theta_2 \ge \theta_1 = \cf(\theta_1) > 
\max\{\partial_{\bold x,\ell}:\ell < \bold k_{\bold x}\}$ and $u \subseteq
 \{0,\dotsc,\bold k_{\bold x}-1\},|u| = k \ge 1$.

The following conditions (A),(B),(C) on $\bold
x,\theta_2,\theta_1,u,k$ are equivalent:
\mn
\begin{enumerate}
\item[$(A)$]  $\bold x$ is $(\theta_2,\theta_1,u,k)$-free over
  $\Lambda_*$
\sn
\item[$(B)$]  as in Definition \ref{a12}(4) omitting clause
  $(d)(\gamma)$, in this case we call 
$(\bar\Lambda,g,\bar h)$ an almost witness
\sn
\item[$(C)$]  for every $\Lambda \subseteq \Lambda_{\bold x}
  \backslash \Lambda_*$ of cardinality $< \theta_2$ there is a weak
witness $(g,\bar h)$ which means: clauses $(\delta),(\varepsilon)$ 
of \ref{a12}(4)(d) and
\sn
\begin{enumerate}
\item[$(\beta)'$]  $g:\Lambda \rightarrow [u]^k$
\sn
\item[$(\zeta)'$]  if $\bar\eta_1 \in \Lambda$ and $m \in u$ then
  for all but $< \theta_1$ of the sequences $\bar\eta_2 \in \Lambda$ we
  have
\begin{itemize}
\item   if $\bar\eta_1 \ne \bar\eta_2,
\bar\eta_1 \upharpoonleft (m,= \emptyset) = \bar\eta_2 \upharpoonleft
(m,= \emptyset)$ and $m \in g(\bar\eta_1) \cap g(\bar\eta_2)$ 
and $i \in \partial_m \backslash (h_m(\bar\eta_1) \cup
h_m(\bar\eta_2))$ \then \, $\eta_{1,m}(i) \ne \eta_{2,m}(i)$
\end{itemize}
\sn
\item[$(\eta)'$]  if $\bar\eta_1 \in \Lambda$ and $\bar\eta_2 \in
  \Lambda_*$ \then \, $\bullet$ of $(\zeta)'$ holds demanding only $m
  \in g(\bar\eta_1)$.
\end{enumerate}
\end{enumerate}
\mn
2) If in addition $\bold x$ is normal (see \ref{a3}(7)) we can add:
\mn
\begin{enumerate}
\item[$(D)$]  like (C) but we replace $\bullet$ inside $(\zeta)'$ (and
  similarly in $(\eta)'$) by
\sn
\begin{itemize}
\item   if $\bar\eta_1 \ne \bar\eta_2 \in \Lambda,\bar\eta_1
  \upharpoonleft (m,=\emptyset) = \bar\eta_2 \rest (m,\emptyset)$ and
 $m \in g(\bar\eta_1) \cap g(\bar\eta_2)$ and $i,j \in \partial_m
  \backslash (h_m(\bar\eta_1) \cup h_m(\bar\eta_2))$ \then \,
$\eta_{1,m}(i) \ne \eta_{2,m}(j)$.
\end{itemize}
\end{enumerate}
\mn
3) If in addition $\Lambda_* \subseteq \Lambda_{\bold x}$ and
each $J_{\bold x,\ell}$ is $\sigma$-complete, \then \,
$\{\Lambda:\Lambda \subseteq \Lambda_{\bold x} \backslash \Lambda_*$
is $(\theta_2,\theta_1,u,k)$-free over $\Lambda_*\}$
is a $\sigma$-complete ideal on $\Lambda_{\bold x} \backslash
\Lambda_*$.
\end{claim}

\begin{PROOF}{\ref{a25}}
1) \underline{$(A) \Rightarrow (B)$}:

Obvious by the formulation of (B).
\medskip

\noindent
\underline{$(B) \Rightarrow (C)$}:

Let $\Lambda \subseteq \Lambda_{\bold x} \backslash \Lambda_*$
have cardinality $< \theta_2$; by clause (B) we can choose
$(\bar\Lambda,g,\bar h)$, an almost witness (for $\Lambda$).  As $|u|=k$,
necessarily $g$ is constantly $u$, so let $g':\gamma(*) \rightarrow
[u]^k$ be constantly $u$, hence it is enough to prove that $(g',\bar
h)$ is a weak witness; clearly clause $(\beta)'$ of (C) holds.  
So by the phrasing of (B) and (C) it
is enough to prove clauses $(\zeta)',(\eta)'$ of (C).  But clause
$(\eta)'$ follows from clause $(\zeta)$ of (B), i.e. $(d)(\zeta)$ of
Definition \ref{a12}(4).  Now for clause $(\zeta)'$, let $\bar\Lambda =
\langle \Lambda_\gamma:\gamma < \gamma(*)\rangle$ and assume
$\bar\eta_\iota \in \Lambda_{\beta_\iota}$ for $\iota = 1,2$ and
$\beta_1 \ne \beta_2 < \gamma(*)$ and $m \in u$ and it suffices to
prove $\bullet$ of $(\zeta)'$.  Clearly $m \in g'(\beta_1) 
\cap g'(\beta_2)$.  So assuming $\bar\eta_1 \upharpoonleft (m, = \emptyset)
= \bar\eta_2 \upharpoonleft (m,=\emptyset)$ and $i \in \partial_m \backslash
(h_m(\bar\eta_1) \cup h_m(\bar\eta_2))$ we should prove that
$\eta_{1,m}(i) \ne \eta_{2,n}(i)$.  By the symmetry \wilog \, $\beta_1
< \beta_2$ and we apply clause $(\zeta)$ of (B) with
$\bar\eta_1,\bar\eta_2,\beta_1,\beta_2,m$ here standing for
$\bar\nu,\bar\eta,\beta,m$ there and get $\eta_{1,m}(i) \ne
\eta_{2,m}(i)$ as promised.
\medskip

\noindent
\underline{$(C) \Rightarrow (A)$}:

So assume that $\Lambda \subseteq \Lambda_{\bold x} \backslash
\Lambda_*$ has cardinality $< \theta_2$ and let
$(g,\bar h)$ be a weak witness for it; (actually we have
no further use of $|\Lambda| < \theta_2$), again necessarily 
$g$ is constantly $u$.  
So for $m \in u,i < \partial_m$ and
every $\bar\eta \in \Lambda$ let $\Omega^1_{i,m,\bar\eta} = \{\bar\nu
\in \Lambda:\bar\nu \upharpoonleft (m,=\emptyset) = \bar\eta
\upharpoonleft (m,= \emptyset)$ and $i \in \partial_m \backslash 
h_m(\bar\nu)$ and $\bar\nu_m(i) = \bar\eta_m(i)\}$.

By the choice of $(g,\bar h)$ and the definition of
$\Omega^1_{i,m,\bar\eta}$ we have:
\mn
\begin{enumerate}
\item[$\bullet_1$]   if $\bar\nu,\bar\rho \in \Omega^1_{i,m,\bar\eta}$
  \then \, $\bar\nu \rest (m,=\emptyset) = \bar\rho \upharpoonleft 
(m,= \emptyset)$ and $\bar\eta_m(i) = \bar\nu_m(i)$ and 
$i \in \partial_m \backslash (h_m(\bar\nu) \cup h_m(\bar\rho))$
\end{enumerate}
\mn
hence applying clause $(\zeta)'$ of (C) to any $\bar\eta_1 \in
\Omega_{i,m,\bar\eta}$ we have
\mn
\begin{enumerate}
\item[$\bullet_2$]   $\Omega^1_{i,m,\bar\eta}$ has $< \theta_1$ members.
\end{enumerate}
\mn
Let $\Omega^1_{\bar\eta} = \cup\{\Omega^1_{i,m,\bar\eta}:m \in u,i <
\partial_m\}$, so recalling the claim assumption $\theta_1 = 
\cf(\theta_1) > \sum\limits_{m} \partial_m$ clearly
\mn
\begin{enumerate}
\item[$\bullet_3$]   if $\bar\eta \in \Lambda$ then
  $\Omega^1_{\bar\eta}$ has cardinality $< \theta_1$.
\end{enumerate}
\mn
By transitivity of equality
\mn
\begin{enumerate}
\item[$\bullet_4$]  if $\bar\nu \in \Omega^1_{\bar\eta}$ then $m <
  \bold k \wedge m \notin u \Rightarrow \bar\nu_m = \bar\eta_m$.
\end{enumerate}
\mn
For $\bar\eta \in \Lambda$ let $\Omega^2_{\bar\eta}$ be the minimal
subset $\Omega$ of $\Lambda$ such that $\bar\eta \in \Omega$ and
$\bar\nu \in \Omega \Rightarrow \Omega^1_{\bar\nu} \subseteq
\Omega$, so recalling $\theta_1$ is regular 
necessarily $|\Omega^2_{\bar\eta}| < \theta_1$.

Let $\langle \bar\eta^*_\gamma:\gamma < \gamma(*)\rangle$ list
$\Lambda$.  We now choose $\Lambda^1_\gamma$ for $\gamma < \gamma(*)$ by 
$\Lambda^1_\gamma = \cup\{\Omega^2_{\bar\eta^*_\beta}:\beta \le \gamma\}$ 
so $\{\bar\eta^*_\gamma\} \subseteq \Lambda^1_\gamma
\subseteq \Lambda$ so clearly $\cup\{\Lambda^1_\gamma:\gamma <
\gamma(*)\} = \Lambda$.

Lastly, let $\Lambda^2_\gamma = \Lambda^1_\gamma \backslash
\cup\{\Lambda^1_\beta:\beta < \gamma\}$, so obviously $\bar\Lambda^2 =
\langle \Lambda^2_\gamma:\gamma < \gamma(*)\rangle$ is a partition of
$\Lambda$.  Let $g_*:\gamma(*) \rightarrow [u]^k$ be constantly $u$
and $\bar h = \langle h_m:m \in u\rangle$ and we shall show that
the triple $(\bar\Lambda^2,g_*,\bar h)$ is as required in \ref{a12}(4)(d).

Now clauses $(\alpha)-(\varepsilon)$ hold by our choices noting that
by $\bullet_4$ we have: if $\bar\eta,\bar\nu \in \Lambda^2_\gamma$ and $m
< \bold k,m \notin u$ then $\bar\eta_m = \bar\nu_m$.  As for
clause $(\zeta)$ let $\bar\eta \in \Lambda_\beta,m \in g(\beta),\alpha
< \beta$ and $\bar\nu 
\in \Lambda^2_\alpha,\bar\nu \upharpoonleft (m,=\emptyset) =
\bar\eta \upharpoonleft (m,=\emptyset)$ and $i \in \partial_m \backslash
h_m(\bar\eta)$ and we should prove that $\bar\nu_m(i) \ne \bar\eta_m(i)$.  But
otherwise $\bar\eta \in \Omega^1_{\bar\nu} \subseteq
\Omega^2_{\bar\eta^*_\alpha} \subseteq
\cup\{\Lambda^2_{\alpha_1}:\alpha_1 \le \alpha\}$, contradiction.

\noindent
2) Similarly.

\noindent
3) By part (1) if we can use as definition clause (C) of (1), so assume
$\Lambda = \bigcup\limits_{i < i(*)} \Lambda_i \subseteq
\Lambda_{\bold x} \backslash \Lambda_*$ and $i(*) < \sigma$ and
$h_{i,m}:\Lambda \rightarrow J_m$ and $(g_i,\bar h_i)$ weakly
witnesses $\Lambda_i$.  As $|u| = k$ necessarily $g_0 := \bigcup\limits_{i}
g_i$ is the constant function from $\Lambda$ into $\{u\}$ 
and let $h_m:\Lambda \rightarrow \cP(\partial_m)$ be
$h_m(\bar\eta) = \cup\{h_{i,m}(\bar\eta):i < i(*)$ and $\bar\eta \in 
\Lambda_i\}$.  Now $h_m$ is into $J_m$
as $J_m$ is a $\sigma$-complete ideal and $i(*) < \sigma$.  Lastly,
clearly $(g_0,\langle h_m:m \in u\rangle)$ is a weak witness for
$\Lambda$ so we are done.
\end{PROOF}

\begin{remark}
\label{a27h}
Why the demand $|u|=k$ in the claim?

Our problem is: in (A) we promise that the function $g$ gives (for a
fixed one $\gamma$) for all $\bar\eta \in \Lambda_\gamma$ the same $u$
whereas in clause (C) this is not the case, in fact, not well defined.  
It is natural then to
divide $\Lambda_\gamma$ to $\le 2^{\bold k}$ cases according to
the value of $g$, but then it is not clear that clause $(\zeta)$ of
(A) holds.  To avoid this we assume $|u|=k$.  Maybe \ref{a25}(3) helps
but this is not crucial.
\end{remark}

\begin{definition}  
\label{a27}
If $\ell g(\bar\partial_\iota) = \bold k_\iota$ and $\bold x_\iota$
is a combinatorial $\bar\partial_\iota$-parameter for
$\iota = 1,2,3$ then we say $\bold x_1 \times \bold x_2 = \bold x_3$ 
\when \,:
\mn
\begin{enumerate}
\item[$(a)$]  $\bar\partial_3 = \bar\partial_1 \char 94
  \bar\partial_2$ hence $\bold k_3 = \bold k_1 + \bold k_2$
\sn
\item[$(b)$]  $\bar J_{\bold x_3} = \bar J_{\bold x_1} 
\char 94 \bar J_{\bold x_2}$
\sn
\item[$(c)$]  $\bar S_{\bold x_3}$ is $\bar S_{\bold x_1} \char 94
  \bar S_{\bold x_2}$, that is
\sn
\begin{enumerate}
\item[$\bullet$]  $S_{\bold x_1,\ell}$ if $\ell < \bold k_1$
\sn
\item[$\bullet$]  $S_{\bold x_2,\ell - \bold k_1}$ if $\ell \ge \bold k_1$
\end{enumerate}
\sn
\item[$(d)$]  $\Lambda_{\bold x_3}$ is the set of $\bar\eta \in
  \prod\limits_{\ell < \bold k_3} {}^{\partial_3(\ell)}(S_{\bold
  x_3,\ell})$ such that for some $\bar\nu \in \Lambda_{\bold x_1}$ and
  $\bar\rho \in \Lambda_{\bold x_2}$ we have:
\sn
\begin{enumerate}
\item[$\bullet$]  if $\ell < \bold k_1$ \then \, $\eta_\ell = \nu_\ell$
\sn
\item[$\bullet$]  if $\ell \ge \bold k_1$ \then \, $\eta_\ell 
= \rho_{\ell - \bold k_1}$.
\end{enumerate}
\end{enumerate}
\end{definition}

\begin{explanation}
\label{a29}
What is the role of the next claim?  We shall prove 
for $(\partial,J) = (\aleph_0,J^{\bd}_\omega)$ and
$(\aleph_1,J^{\bd}_{\aleph_1} \times J^{\bd}_{\aleph_0})$, that for many
strong limit singular $\mu$, there is a $1-\cp \, \bold x$ such that 
$(\partial_{\bold x,0},J_{\bold x,0}) = (\partial,J)$ and 
$\bold x$ has $2^{\mu}$-BB and $\bold x$ is quite free.  But we do not know how
  to get one which is even just $\aleph_{\omega +1}$-free.  But such
  freeness is needed in \S2!  However, using long enough finite
  products we can get enough freeness.  More fully, first by
  \ref{a30}, the product gives combinational parameter of the expected
  length (the sum) and weak ordinariness, ordinariness and normality
  are preserved.

Second, by \ref{a34} the products have the appropriate (pre)-black-box
if each product has one.

Third, in \ref{a34} - \ref{a41} we get that for each $\bold x_\ell$
satisfies enough cases of $(\theta_2,\theta_1,u)$-freeness 
conditions then their product satisfies more.

Fourth, in Theorem \ref{a51} we prove the existence of $\bold x_\ell
\, (\ell < \bold k)$ as required relying on \cite{Sh:1008}.

Lastly, in Conclusion \ref{a52} we get the desired conclusion used in \S2.
\end{explanation}

\begin{claim}  
\label{a30}
1) If $\bold x_\iota$ is a combinatorial
   $\bar\partial_\iota$-parameter for $\iota=1,2$ \then \, there is
   one and only one combinatorial parameter $\bold x_3$ 
such that $\bold x_1 \times \bold x_2 = \bold x_3$.

\noindent
2) The product in Definition \ref{a27} is associative.

\noindent
3) If $\bold x_1 \times \bold x_2 = \bold x_3$ \then \, $\bold x_2 \times
   \bold x_1$ is a permutation of $\bold x_3$, see Definition
\ref{a5}(4).

\noindent
4) If in Definition \ref{a27}, $\bold x_1,\bold x_2$ are [weakly]
ordinary and/or normal, see \ref{a5}(3), \ref{a3}(7), 
\then \, so is $\bold x_3$.
\end{claim}

\begin{PROOF}{\ref{a30}}
Straightforward.
\end{PROOF}

\begin{claim}  
\label{a34}
1) $\bold x_3$ has $\bar\chi_3$-pre-black box \when \,: 
\mn
\begin{enumerate}
\item[$(a)$]  $\bold x_\iota$ is a combinatorial
$\bar\partial_\iota$-parameter for $\iota=1,2,3$
\sn
\item[$(b)$]  $\bold x_1 \times \bold x_2 = \bold x_3$
\sn
\item[$(c)$]  $\bold x_\iota$ has $\bar\chi_\iota$-pre-black box 
for $\iota= 1,2$
\sn
\item[$(d)$]  $\bar\chi_3 = \bar \chi_1 \char 94 \bar\chi_2$
\sn
\item[$(e)$]  if $\ell < \ell g(\bar\partial_2)$ then $\chi_{2,\ell} =
(\chi_{2,\ell})^{|\Lambda_{\bold x_1}|}$.
\end{enumerate}
\mn
2) Moreover, $\bold x_3$ has a $\bar \chi_3$-black box \when \, in
addition
\mn
\begin{enumerate}
\item[$(c)^+$]  $\bold x_2$ has a $\bar\chi_2$-black box and
  $\chi_{2,n} = (\chi_{2,n})^{|\Lambda_{\bold x_1}|}$. 
\end{enumerate}
\end{claim}

\begin{PROOF}{\ref{a34}}  
1) For each $m < \bold k_{\bold x_2}$ let 
$\bar F^m = \langle F^m_\alpha:\alpha <
\chi_{2,m}\rangle$ list $\{F:F$ a function from $\Lambda_{\bold x_1}$
into $\chi_{2,m}\}$.
By clause (e) of the assumption, such sequence exists.  
Let $\bar\alpha^1$ be a $\bar\chi_1$-pre-black box
for $\bold x_1$ and let $\bar\alpha^2$ be a 
$\bar\chi_2$-pre-black box for $\bold x_2$; they exist by clause (c)
of the assumption.

Lastly, we define $\bar\alpha = \langle \bar\alpha_{\bar\eta}:\bar\eta
\in \Lambda_{\bold x_3}\rangle$ where $\bar\alpha_{\bar\eta} = \langle
\alpha_{\bar\eta,m,i}:m < \bold k_{\bold x_3},i < \partial_m\rangle$
as follows: for $\bar\eta \in \Lambda_{\bold x_3},
m < \bold k_{\bold x_3}$ and $i < \partial_{\bold x_3,m}$ we let
\mn
\begin{enumerate}
\item[$\bullet$]  if $m < \bold k_{\bold x_1}$ then
  $\alpha_{\bar\eta,m,i} = \alpha^1_{\bar\eta \rest \bold k(\bold x_1),m,i}$
\sn
\item[$\bullet$]  if $m = \bold k_{\bold x_1} + \ell$ and 
$\ell < \bold k_{\bold x_2}$ then $\alpha_{\bar\eta,m,i} =
F^m_{\alpha^2_{\bar\nu,\ell,i}}(\bar\eta \rest \bold k_{\bold x_1})$
  where $\bar\nu = \bar\eta \rest [\bold k_{\bold x_1},\bold k_{\bold
    x_3})$, i.e. $\bar\nu = \langle \eta_{\bold k(\bold x_1)+ n}:n < \bold
  k_{\bold x_2}\rangle$.
\end{enumerate}
\mn
Clearly $\bar\alpha$ is of the right form, but is it really a
$\bar\chi_3$-pre-black box?  So assume $h_m:\Lambda_{\bold x_3,m}
\rightarrow \chi_{3,m}$ for $m < \bold k_{\bold x_3}$ and we should
find $\bar\eta \in \Lambda_{\bold x_3}$ as in Definition \ref{a9}(1).
Now first we define $h^2_m:\Lambda_{\bold x_2,m} \rightarrow
\chi_{2,m}$ for $m < \bold k_{\bold x_2}$ as follows:
$h^2_m(\bar\nu)$ is the unique $\alpha < \chi_{2,m}$ such that:
$\bar\rho \in \Lambda_{\bold x_1} \Rightarrow h_{\bold k(\bold
  x_1)+m}(\bar\rho \char 94 \bar\nu) = F^m_\alpha(\bar\rho)$, possible by
the choice of $\bar F^m$ above.  As
$\bar\alpha^2$ is a $\bar\chi_2$-pre-black box, clearly there is
$\bar\nu \in \Lambda_{\bold x_2}$ such that $m < \bold k_{\bold x_2}
\wedge i < \partial_{\bold x_2,m} \Rightarrow h^2_m(\bar\nu
\upharpoonleft (m,i)) = \alpha^2_{\bar\nu,m,i}$.

Fix a sequence $\bar\nu \in \Lambda_{\bold x_2}$ as in the former paragraph.
Now for $m < \bold k_{\bold x_1}$ we define $h^1_m:
\Lambda_{\bold x_1,m} \rightarrow \chi_{1,m}$ by $h^1_m(\bar\rho) =
h_m(\bar\rho \char 94 \bar\nu)$ for $\bar\rho \in \Lambda_{\bold
  x_1,m}$, it is well defined as by our assumptions on $h_m$, it has domain
$\Lambda_{\bold x_3,m}$ and as $\bar\nu \in \Lambda_{\bold x_2}$, clearly
$\bar\rho \char 94 \bar\nu \in \Lambda_{\bold x_3,m}$ by the
  definition of $\bold x_3$.  As $\bar\alpha^1$ is a
  $\bar\chi_1$-pre-black box for $\bold x_1$ there is $\bar\rho \in
  \Lambda_{\bold x_1}$ such that $m < \bold k_{\bold x_1} \wedge i <
  \partial_{\bold x_1,m} \Rightarrow h^1_m(\bar\rho) =
  \alpha^1_{\bar\rho,m,i}$.  We shall show that
  $\bar\eta := \bar\rho \char 94 \bar\nu$ is as required.

First, $\bar\eta \in \Lambda_{\bold x_3}$ because $\bold x_3 = \bold
x_1  \times \bold x_2,\bar\rho \in \Lambda_{\bold x_1}$ and $\bar\nu
\in \Lambda_{\bold x_2}$.

Second, if $m < \bold k_{\bold x_1} \wedge i < \partial_{\bold x_3,m}
= \partial_{\bold x_1,m}$ then
\mn
\begin{enumerate}
\item[$(*)$]  
\begin{enumerate}
\item[(a)]  $h_m(\bar\eta \upharpoonleft (m,i)) 
= h^1_m(\bar\rho \upharpoonleft (m,i))$ by
the choices of $\bar\eta$ and $h^1_m$ 
\sn
\item[(b)]  $h^1_m(\bar\rho \upharpoonleft
(m,i)) = \alpha^1_{\bar\rho,m,i}$ by the choice of $\bar\rho$ 
\sn
\item[(c)]  $\alpha^1_{\bar\rho,m,i} = 
\alpha_{\bar\eta,m,i}$ by the choice of $\alpha_{\bar\eta,m,i}$ 
\newline
\qquad so together
\sn
\item[(d)]  $h_m(\bar\eta \upharpoonleft (m,i)) = \alpha_{\bar\eta,m,i}$.
\end{enumerate}
\end{enumerate}
\mn
Third, if $m \in [\bold k_{\bold x_1},\bold k_{\bold x_3}) 
\wedge i < \partial_{\bold x_3,n}$
  then $m = \bold k_{\bold x_1} + \ell,\ell < \bold k_{\bold x_2}$ for
some $\ell$ and use the choices of $\alpha_{\bar\eta,m,i}$ and of $\bar\nu$. 

\noindent
2) We have to deal with the black box case.  So recalling
Definition \ref{a9}(4) we are assuming:
\mn
\begin{enumerate}
\item[$(*)$] 
\begin{enumerate}
\item[(a)]  $\bar\Lambda^2 = \langle
  \Lambda^2_\gamma:\gamma < |\Lambda_{\bold x_2}|\rangle$ is a partition
  of $\Lambda_{\bold x_2}$
\sn
\item[(b)]  if $\gamma < |\Lambda_{\bold x_2}|$, then
$\bold x_2 \rest \Lambda^2_\gamma$ has a $\bar\chi_2$-pre-black box.
\end{enumerate}
\end{enumerate}
\mn
Now repeating the proof above, note:
\mn
\begin{enumerate}
\item[(c)]  $\langle \bar\nu_\alpha:\alpha < |\Lambda_{\bold
    x_2}|\rangle$ list $\Lambda_{\bold x}$ as required in Definition
  \ref{a9}(4). 
\end{enumerate}
\mn
We can choose $\bar\alpha^2$ such
that not only it is a $\bar\chi_2$-pre-black box but also
$\bar\alpha^2 \rest \Lambda^2_\gamma = \langle
\bar\alpha^2_{\bar\nu}:\bar\nu \in \Lambda^2_\gamma\rangle$ is a
$\bar\chi_2$-pre-black box for each $\gamma < |\Lambda_{\bold x_2}|$.

Having defined $\bar\alpha = \langle \bar\alpha_{\bar\eta}:\bar\eta
\in \Lambda_{\bold x}\rangle$ note that:
\mn
\begin{enumerate}
\item[$(*)$]  $|\Lambda_{\bold x_1}| \le \chi_{\bold x_2,0}$ (by clause (e) of
  the claim) and $\chi_{\bold x_2,0} \le |\Lambda_{\bold x_2}|$ (by \ref{a19})
  and $|\Lambda_{\bold x_2}|$ is infinite (otherwise the
  $\bar\chi_2$-black box fails) hence $|\Lambda_{\bold x_3}| =
  |\Lambda_{\bold x_2}| \times |\Lambda_{\bold x_2}| = |\Lambda_{\bold
    x_2}|$
\sn
\item[$(*)$]  letting $\Lambda_\gamma = \Lambda_{\bold x_1} \times
  \Lambda^2_\gamma$ the sequence $\langle \Lambda_\gamma:\gamma <
  |\Lambda_{\bold x_3}|\rangle$ is a partition of $\Lambda_{\bold x_3}$.
\end{enumerate}
\mn
Mainly we need to prove that: if $\gamma < |\Lambda_{\bold x_3}|$ then
$\bar\alpha \rest \Lambda_\gamma$ is a $\bar\chi$-pre-black box.  
This proof is exactly as in the proof of the first part.

Lastly, we choose $\langle \bar\nu_\alpha:\alpha < |\Lambda_{\bold
  x_3}|\rangle$ as required.  Toward this let $\mu_\iota =
\max\{\mu:(\forall \ell < \bold k_\iota)(2^{< \mu} \le
\chi_{\iota,\ell})$ note that necessarily $|\Lambda_{\bold x_2}| =
|\Lambda_{\bold x_3}|$ and $\mu_1 \le |\Lambda_{\bold x_1}| < \mu_2$.
Now choose $\langle \bar\nu^1_\alpha:\alpha < |\Lambda_{\bold
  x_1}|\rangle$ be such that $\{\bar\nu^1_\alpha:\alpha <
|\Lambda_{\bold x_1}|\} = \Lambda_{\bold x_1}$ and $\alpha \le \beta <
\alpha + \mu_1 \Rightarrow \bar\nu^1_\alpha = \bar\nu^1_\beta$.

To finish, define $\langle \bar\nu_\alpha:\alpha < 
|\Lambda_{\bold x_3}|\rangle$ by:
\mn
\begin{itemize}
\item  if $\gamma = |\Lambda_{\bold x_1}| \cdot \alpha + \beta$
and $\beta < |\Lambda_{\bold x_1}|$, then \, 
 $\bar\nu_\gamma = \bar\nu^1_\beta \char 94 \bar\nu^2_\gamma$.
\end{itemize}
\mn
Recalling $\gamma_1 < \gamma_2 < \gamma_1 + \mu_2 \Rightarrow
\bar\nu^2_{\gamma_1} = \bar\nu^2_{\gamma_2}$ we are easily done.
\end{PROOF}

\noindent
The following Definition is somewhat similar to \cite{Sh:883}, 
different notation than earlier.
\begin{definition}  
\label{a36}
Let $\bold x = (\bold k,\bar\partial,\bar S,\Lambda,\bar J)$ be disjoint for
notational transparency (see \ref{a5}(3)).

\noindent
0) For $u \subseteq \{0,\dotsc,\bold k-1\}$ let 
$u^\perp = \{\ell < \bold k:\ell \notin u\}$.

\noindent
1) For $\cU \subseteq \bigcup\limits_{\ell < \bold k}{}^{\partial(\ell)}
(S_{\bold x,\ell})$ let $\Lambda_{\cU} = \Lambda_{\bold
   x,\cU} = \Lambda_{\bold x}(\cU) = \{\bar\eta \in \Lambda_{\bold
   x}:\eta_\ell \in \cU$ for every $\ell < \bold k\}$.

\noindent
2) For $\cU \subseteq \bigcup\limits_{\ell < \bold k} 
{}^{\partial(\ell)}(S_{\bold x,\ell})$ and $u 
\subseteq \{0,\dotsc,\bold k-1\}$ let:
\mn
\begin{enumerate}
\item[$(a)$]   $\add_{\bold x}(u) = \{\bold u:\bold u \subseteq
\bigcup\limits_{\ell \in u} {}^{\partial(\ell)}(S_{\bold x,\ell})$ satisfies
$|\bold u \cap {}^{\partial(\ell)}(S_{\bold x,\ell})|=1$ for $\ell \in
u\}$, note that $\bold u \in \add_{\bold x}(u) \Rightarrow |\bold u| = |u|$
\sn
\item[$(b)$]  for $\bold u \in \add_{\bold x}(u)$ let
$\Lambda_{\cU,\bold u} = \Lambda_{\bold x}(\cU,\bold u) := 
\{\bar\eta \in \Lambda_{\bold x}$: for some $m \in u$
   we have $\ell < \bold k \wedge \ell \ne m \Rightarrow \eta_\ell \in
   (\cU \cup \bold u) \cap {}^{\partial(\ell)}(S_{\bold x,\ell})$ and $\ell <
   \bold k \wedge \ell = m \Rightarrow \eta_\ell \in \cU \cap
   {}^{\partial(\ell)}(S_{\bold x,\ell})\}$
\sn
\item[$(c)$]  $\Lambda^*_{\bold x}(\cU,\bold u) := \Lambda_{\bold x}
(\cU \cup \bold u) \backslash \Lambda_{\bold x}(\cU,\bold u)$; this
set is interesting, i.e. non-empty only when $\cU \cap \bold u = \emptyset$
  and then it is equal to $\{\bar\eta \in \Lambda_{\bold x}$: if $\ell \in
  u$ then $\eta_\ell \in \bold u$ and if $\ell \in \bold k \backslash
  u$ then $\eta_\ell \in \cU\}$.
\end{enumerate}
\mn
3) For non-empty $u \subseteq \{0,\dotsc,\bold k-1\}$ we say $\bold x$
   is $\theta-(u,k)$-free \when \,: if $\cU \subseteq \bigcup\limits_{\ell
   < \bold k} {}^{\partial(\ell)}(S_{\bold x,\ell})$ has cardinality $< \theta$
   and $\bold u \in \add_{\bold x}(u^\perp)$ is disjoint to $\cU$ 
then $\Lambda_{\bold x}(\cU \cup \bold u)$ is 
$(\infty,2,u,k)$-free over $\Lambda_{\bold x}(\cU,\bold u)$ 
recalling \ref{a12}(4),(5).

\noindent
3A) If $\theta > |\Lambda_{\bold x}|$ we may write $\infty$ instead of 
$\theta$ in part (3).

\noindent
4) For non-empty $u \subseteq \{0,\dotsc,\bold k-1\}$ we say $\bold x$
is $(\theta_2,\theta_1)-(u,k)$-free \when \,: if $\cU \subseteq
\bigcup\limits_{\ell < \bold k} {}^{\partial(\ell)}(S_{\bold x,\ell})$ and
   $\bold u \in \add_{\bold x}(u^\perp)$ is disjoint to $\cU$ 
\then \, $\Lambda_{\bold x}(\cU \cup \bold u)$ is 
$(\theta_2,\theta_1,u,k)$-free over $\Lambda_{\bold x}(\cU,\bold u)$
recalling \ref{a12}(4).
\end{definition}

\begin{observation}
\label{a39}
1) If Definition \ref{a36}(3), the conclusion is equivalent to
``$\Lambda^*_{\bold x}(\cU,\bold u) = \Lambda_{\bold x}(\cU \cup
   \bold u) \backslash \Lambda_{\bold x}(\cU,\bold u)$ is 
$(\infty,2,u,k)$-free". 

\noindent
2) Similarly in \ref{a36}(4); that is, assume $u \subseteq \{0,\dotsc,\bold
   k-1\},\cU \subseteq \bigcup\limits_{\ell < \bold k}
   {}^{\partial(\ell)}(S_{\bold x,\ell})$ and $\bold u \in
   \add(u^\perp)$ is disjoint to $\cU$, \then \,: $\Lambda_{\bold x}(\cU
   \cup \bold u)$ is $(\theta_2,\theta_1,u,k)$-free over
   $\Lambda_{\bold x}(\cU < \bold u)$ \Iff \, $\Lambda^*_{\bold x}
(\cU,\bold u) = \Lambda_{\bold x}(\cU \cup \bold u) \backslash
   \Lambda_{\bold x}(\cU,\bold u)$ is $(\theta_2,\theta_1,u,k)$-free..

\noindent
3) If $\bold x$ is $\theta-(u,k)$-free \then \, $\bold x$ is
   $(\theta,u,k)$-free; see Definition \ref{a36}(3), \ref{a12}(4),(5)
   respectively.

\noindent
4) If $\bold x$ is $(\theta_2,\theta_1)-(u,k)$-free \then \, $\bold x$
   is $(\theta_2,\theta_1,u,k)$-free; see Definition \ref{a36}(4),
   \ref{a12}(4) respectively.
\end{observation}

\begin{PROOF}{\ref{a39}}
1) As if $\bar\eta \in \Lambda_{\bold x}(\cU \cup \bold u) \backslash
\Lambda_{\bold x}(\cU,\bold u)$ and $\bar\nu \in \Lambda_{\bold
  x}(\cU,\bold u)$ as $\bold u \in \add_{\bold x}(u^\perp)$ it follows
that $(\exists m \in u^\perp)[\eta_m \ne \nu_m]$.

\noindent
2) Similarly.

\noindent
3),4) Straightforward.
\end{PROOF}

\noindent
The gain in the following theorem is that taking products of
combinatorial parameters, we gain new cases of freeness.
\begin{tft}  
\label{a41}
If $\boxplus$ below holds, \then \, $\bold x$ is $(\theta_{\bold
  m},\theta^+_0)-(u,1)$-free.  If, in addition, every $\bold x_\ell$
is $\theta^+_0$-free, then $\bold x$ is $(\theta_{\bold m},u)$-free:
\mn
\begin{enumerate}
\item[$\boxplus$]
\begin{enumerate}
\item[(a)]  $\bold x_\ell$ is a combinatorial
  $\langle \partial_\ell \rangle$-parameter for $\ell < \bold k$
\sn
\item[(b)]  $\bold x = \bold x_0 \times \ldots \times \bold x_{\bold k -1}$
\sn
\item[(c)]  $u \subseteq \{0,\dotsc,\bold k-1\}$ and
  $\bold m = |u| > 0$, hence $\bold m \le \bold k$
\sn
\item[(d)]  $\theta_0 < \theta_1 < \ldots < \theta_{\bold m}$ are
  regular except possibly $\theta_0$
\sn
\item[(e)]  $\partial_{\bold x_\ell} \le \theta_0$ for $\ell < \bold k$
\sn
\item[(f)]  $\bold x_k$ is $(\theta_{m+1},
\theta^+_m)$-free when $k \in u \wedge m < \bold m$.
\end{enumerate}
\end{enumerate}
\end{tft}

\begin{PROOF}{\ref{a41}}  
%\label{}
\Wilog \, $\bold x$ is disjoint, i.e. the sets $S_\ell := 
S_{\bold x,\ell}$ are pairwise disjoint for $\ell < \bold k$.  
We prove the claim by induction on $\bold m$ (so fix $\bold k$ but we
  vary $u$ and the $\theta_m$'s).  So let
$\bold u \in \add_{\bold x}(u^\perp)$ and $\cU \subseteq
\bigcup\limits_{\ell < \bold k} {}^{\partial_\ell}(S_\ell)$ has cardinality
$< \theta_{\bold m}$ and we shall prove that 
$\Lambda^*_{\bold x}(\cU,\bold u)$ is $(\infty,\theta^+_0,u,1)$-free.
Clearly this suffices for the first phrase and the second follows recalling
  \ref{a15}(2), \ref{a39}(2).
\bigskip

\noindent
\underline{Case 1}:  $\bold m=1$

So $|u|=1$ and let $u = \{\ell\}$ hence $\bar\eta \mapsto \eta_\ell$
is a one-to-one function from $\Lambda^*_{\bold x}(\cU,\bold u)$ onto
$\cU_\ell : = \cU \cap \Lambda_{\bold x_\ell}$.  We know that $\bold x_\ell$ is
$(\theta_1,\theta^+_0)$-free and $|\cU_\ell| < \theta_1$ hence there
is a partition $\langle \cU_{\ell,\alpha}:\alpha < \alpha(*)\rangle$
of $\cU_\ell$ to sets each of cardinality $\le \theta_0,\alpha(*) \le
|\cU_\ell| < \theta_1$ and
$h_\ell:\cU_\ell \rightarrow J_{\bold x_\ell}$ such that $\alpha < \beta <
\alpha(*) \wedge \eta \in \cU_{\ell,\alpha} \wedge \nu \in
\cU_{\ell,\beta} \wedge \partial_\ell > i \notin h_\ell(\nu)
\Rightarrow \eta(i) \ne \nu(i)$.  For $\alpha < \alpha(*)$ 
let $\Lambda_\alpha = \{\bar\eta \in \Lambda^*_{\bold
  x}(\cU,\bold u):\eta_\ell \in \cU_{\ell,\alpha}\}$, clearly
$\langle \Lambda_\alpha:\alpha < \alpha(*)\rangle$ is a partition of
$\Lambda^*_{\bold x}(\cU,\bold u)$ to sets each of cardinality $\le
\theta_0$.  Let the function $g$ from $\alpha(*)$ to 
$[u]^1 = \{\{\ell\}\}$ be defined by $g(\alpha) = \{\ell\}$,
clearly the partition $\langle \Lambda_\alpha:\alpha <
\alpha(*)\rangle$ and the functions $g,h_\ell$ witness that $\Lambda^*_{\bold
  x}(\cU,\bold u)$ is $(\theta_{\bold m},\theta^+_0)$-free, as required.
\bigskip

\noindent
\underline{Case 2}:  $\bold m > 1$

Let $m = \bold m-1$, as $\bold m > 1$, clearly $m$ is $\ge 1$.  So
for $k \in u$ the c.p. $\bold x_k$ is
$(|\cU|^+,\theta^+_m)$-free and let $\cU_k = \cU \cap
{}^{\partial(k)}(S_k) \subseteq \Lambda_{\bold x_k}$ and by the
induction hypothesis, \wilog \, $|\cU| \ge \theta_m$.
Hence as in earlier cases
we can find a function $h^*_k:\cU_k \rightarrow 
J_{\bold x_k}$ such that in the directed
graph $(\cU_k,R_k)$ each node has out-degree $\le \theta_m$, that is,
$(\forall \eta \in \cU_k)(\exists^{\le \theta_m} \nu \in \cU_k)[\eta
  R_k \nu]$, where
\mn
\begin{enumerate}
\item[$(*)_1$]   $R_k = R_{k,h_k} = 
\{(\eta,\nu):\eta,\nu \in \cU_k$ and for some $i < \partial_k$ we have
$i \notin h^*_k(\nu),\eta(i)=\nu(i)\}$
\sn
\item[$(*)_2$]  let $\Lambda_*$ be $\Lambda^*_{\bold x}(\cU,\bold u) =
\Lambda_{\bold x}(\cU \cup \bold u) \backslash \Lambda_{\bold x}(\cU,\bold u)$
\end{enumerate}
\mn
and let
\mn
\begin{enumerate}
\item[$(*)_3$]  $R_* = \{(\bar\eta,\bar\nu):\bar\eta,\bar\nu \in
  \Lambda_*$ and for some $k \in u$ we have $\eta_k R_k \nu_k$ and
  $\ell < \bold k \wedge \ell \ne k \Rightarrow \eta_\ell =
  \nu_\ell\}$.
\end{enumerate}
\mn
Clearly
\mn
\begin{enumerate}
\item[$(*)_4$]  $(\Lambda_*,R_*)$ is a directed graph with each node having
out-degree $\le \theta_m$.
\end{enumerate}
\mn
Let $\bar\Lambda = \langle \Lambda_\gamma:\gamma < \gamma(*)\rangle$
be such that:
\mn
\begin{enumerate}
\item[$(*)_5$]  $(a) \quad \bar\Lambda$ is a partition of $\Lambda_*$
\sn
\item[${{}}$]  $(b) \quad \Lambda_\gamma$ has cardinality $\le
  \theta_m$
\sn
\item[${{}}$]  $(c) \quad$ if $\bar\eta \in \Lambda_\beta,\bar\nu \in
\Lambda_\gamma$ and $\beta < \gamma < \gamma(*)$ \then \, $\neg(\bar\eta R_*
\bar\nu)$; that is
\sn
\begin{enumerate}
\item[${{}}$]  $\bullet \quad$ if $\ell \in u$ and $\bar\eta
  \upharpoonleft (\ell,< 0) = \bar\nu \upharpoonleft (\ell,< 0)$ 
\then \, $\neg(\eta_\ell R_\ell \nu_\ell)$.
\end{enumerate}
\end{enumerate}
\mn
[Why?   Let $\langle \bar\eta_\alpha:\alpha <
|\Lambda_*|\rangle$ list $\Lambda_*$ with no repetition.  For $\alpha
< |\Lambda_*|$ we define $u_{\alpha,n} \in [|\Lambda_*|]^{\le
  \theta_m}$ by induction on $n$, increasing with $n$ by $u_{\alpha,0}
= \{\alpha\},u_{\alpha,n+1} = \{\beta$: for some $\gamma \in
u_{\alpha,n}$ we have $\bar\eta_\gamma R_* \bar\eta_\beta$ or $\beta =
\gamma\}$.

So $u_\alpha = \bigcup\limits_{n} u_{\alpha,n} \in [|\Lambda_*|]^{\le
  \theta_m},\alpha \in u_\alpha$ and $[\bar\eta_\beta R_*
\bar\eta_\gamma \wedge \beta \in u_\alpha \Rightarrow \gamma \in
u_\alpha]$.  Let $\Lambda_\alpha = \{\bar\eta_\gamma:\gamma \in u_\alpha$
but $(\forall \beta < \alpha)(\gamma \notin u_\beta)\}$, now check
that $\bar\Lambda = \langle \Lambda_\alpha:\alpha <
|\Lambda_*|\rangle$ is as required.]
\mn
\begin{enumerate}
\item[$(*)_6$]  it is enough to prove for each $\gamma < \gamma(*)$
  that $\Lambda_\gamma$ is $(\infty,\theta^+_0,u,1)$-free.
\end{enumerate}
\mn
[Why?  It is enough to prove $\Lambda_*$ is
$(\infty,\theta^+_0,u,1)$-free.

By the assumption of $(*)_6$ for 
each $\gamma < \gamma(*)$ let $\bar\Lambda_\gamma,g_\gamma,\bar
h_\gamma$ witness that $\Lambda_\gamma$ is $(\infty,\theta^+_0,u)$-free,
that is (recall Definition \ref{a12}(4); for $k=1$, see \ref{a12}(5)):
\mn
\begin{enumerate}
\item[$\bullet$]  $\bar\Lambda_\gamma = \langle
  \Lambda_{\gamma,\varepsilon}:\varepsilon <
  \varepsilon_\gamma\rangle$ is a partition of $\Lambda_\gamma$
\sn
\item[$\bullet$]  $\Lambda_{\gamma,\varepsilon}$ has cardinality $\le
  \theta_0$
\sn
\item[$\bullet$]  $g_\gamma:\varepsilon_\gamma \rightarrow u$
\sn
\item[$\bullet$]  if $\bar\eta,\bar\nu \in
  \Lambda_{\gamma,\varepsilon}$ and $k \in u \subseteq \bold k,k \ne
  g_\gamma(\varepsilon)$ then $\eta_k = \nu_k$
\sn
\item[$\bullet$]  $\bar h_\gamma = \langle h_{\gamma,k}:k \in u \rangle$
\sn
\item[$\bullet$]  $h_{\gamma,m}$ is a function from $\Lambda_\gamma$ into $J_m$
\sn
\item[$\bullet$]  if $\bar\eta \in \Lambda_{\gamma,\varepsilon}$ and
$\bar\nu \in \cup\{\Lambda_{\gamma,\xi}:\xi < \varepsilon\},m =
g_\gamma(\varepsilon),\bar\nu \upharpoonleft (m) = \bar\eta \upharpoonleft (m)$
and $i \in \partial_k \backslash h_{\gamma,k}(\bar\eta)$ 
then $\eta_k(i) \ne \nu_k(i)$.
\end{enumerate}
\mn
Let
\mn
\begin{enumerate}
\item[$\bullet$]  $\zeta_\gamma = \sum\limits_{\beta < \gamma}
  \varepsilon_\beta$ for $\gamma \le \gamma(*)$
\sn
\item[$\bullet$]  $\Lambda'_\varepsilon = \Lambda_{\gamma,\varepsilon -
  \zeta_\gamma}$ when $\varepsilon \in [\zeta_\gamma,\zeta_{\gamma +1}]$
\sn
\item[$\bullet$]  $g$ is the function with domain $\zeta_{\gamma(*)}$
\sn
\item[$\bullet$]  $g(\varepsilon) = g_\gamma(\varepsilon - 
\zeta_\gamma)$ when $\varepsilon \in 
[\zeta_\gamma,\zeta_{\gamma +1})$ and $\gamma < \gamma(*)$
\sn
\item[$\bullet$]  $h_k$ is the function with domain $\Lambda_*$
  defined by: if $\bar\eta \in \Lambda_\zeta,\zeta = \zeta_\gamma +
  \varp$ and $\varp < \varp_\gamma$ then $h_k(\bar\eta) =
  h_{\gamma,k}(\bar\eta) \cup h^*_k(\bar\eta)$.
\end{enumerate}
\mn
Now check Definition \ref{a12}(4).]

Fix $\gamma < \gamma(*)$ and we shall prove for it the condition from $(*)_6$.
If $|\Lambda_\gamma| < \theta_m$ the desired statement 
follows from the induction
hypothesis so assume $|\Lambda_\gamma| = \theta_m$.  Let $\langle
\eta_{\gamma,\alpha}:\alpha < \theta_m\rangle$ list
$\{\nu_k:\bar\nu \in \Lambda_\gamma$ and $k \in u\}$.

For $\beta < \theta_m$ let $\cU_{\gamma,\beta} =
\{\eta_{\gamma,\alpha}:\alpha < \beta\}$
and let $k(\beta)$ be the unique $k \in u$ such
that $\eta_{\gamma,\beta} \in {}^{\partial_k}(S_k)$.  Clearly
$|\cU_{\gamma,\beta}| < \theta_m$.
Also $\langle \cU_{\gamma,\beta}:\beta <
\theta_m\rangle$ is $\subseteq$-increasing continuous 
with union $\cup\{\Lambda^*_{\bold x}(\cU_{\gamma,\beta},\bold u):
\beta < \theta_m\} = \Lambda_\gamma$.

By induction on $\beta < \theta_m$ we choose $\langle 
\bar\Lambda_\beta,g_\beta,\bar h^\beta)$ such that
\mn
\begin{enumerate}
\item[$(*)_7$]
\begin{enumerate}
\item[(a)]  $\bar\Lambda_\beta = \langle
  \Lambda_{\gamma,\varepsilon}:\varepsilon < \varepsilon_\beta\rangle$
  is a partition of $\Lambda^*_{\bold x}(\cU_{\gamma,\beta},\bold u)$ so
  $\alpha < \beta \Rightarrow \bar\Lambda_\alpha \triangleleft 
\bar\Lambda_\beta$
\sn
\item[(b)]   each $\Lambda_{\gamma,\varepsilon}$ has
  cardinality $\le \theta_0$
\sn
\item[(c)]  $g_\beta:\varepsilon_\beta \rightarrow u$
  such that $\alpha < \beta \Rightarrow g_\alpha \subseteq g_\beta$
\sn
\item[(d)]  $\bar h^\beta = \langle h^\beta_k:k \in u\rangle$
\sn
\item[(e)]  $h_{\beta,k}:\Lambda_{\bold x}(\cU_{\gamma,\beta},\bold u) 
\rightarrow J_k$ and $\alpha < \beta
  \Rightarrow h^\alpha_k \subseteq h^\beta_k$
\sn
\item[(f)]  if $\varepsilon <
  \varepsilon_\beta,\bar\eta \in
  \Lambda_{\gamma,\varepsilon},g_\beta(\bar\eta) = k$ so $k \in u$ and
  $\bar\nu \in \cup\{\Lambda_{\gamma,\zeta}:\zeta < \varepsilon\}$ and
$\bar\nu \upharpoonleft (k,<0) = \bar\eta \upharpoonleft 
(k,<0)$ \then \, $i \in \partial_{\bold x,k} \backslash h_{\beta,k}(\bar\eta)
\Rightarrow \nu_k(i) \ne \eta_k(i)$.
\end{enumerate}
\end{enumerate}
\mn
For $\beta = 0,\Lambda^*_{\bold x}(\cU_{\gamma,\beta},\bold u) = \emptyset$ so
this is obvious.  For $\beta$ limit take unions.

Lastly, for $\beta = \beta_* +1$, it is enough to show that
$\Lambda^*_{\bold x}(\cU_{\gamma,\beta},\bold u)$ is 
$(\infty,\theta^+_0,u)$-free over
$\Lambda^*_{\bold x}(\cU_{\gamma,\beta_*},\bold u)$.  Now $\cU_{\gamma,\beta}
  \backslash \cU_{\gamma,\beta_*} =
  \{\eta_{\gamma,\beta_*}\},\eta_{\gamma,\beta_*} \in 
{}^{\partial_{k(\beta_*)}}(S_{k(\beta_*)})$, 
hence $\eta_{\gamma,\beta_*} \in \cU$.   
So let $u_{\gamma,\beta} = u \backslash \{k(\beta_*)\},
\bold u_{\gamma,\beta} = \bold u \cup
\{\eta_{\gamma,\beta_*}\}$, so $\bold u_{\gamma,\beta} \in \add_{\bold
  x}(u^\perp_{\gamma,\beta}),\bold u_{\gamma,\beta} \subseteq \{0,\dotsc,\bold
k_{\bold x}-1\}$ has $m$ members because $|\bold u| = \bold m = m+1$.
Recall $\Lambda^*_{\bold x}(\cU_{\gamma,\beta},\bold
  u_{\gamma,\beta}) = \Lambda^*_{\bold x}(\cU_{\gamma,\beta},\bold u) 
\backslash \Lambda^*_{\bold x}(\cU_{\gamma,\beta_*},\bold u)$ 
and by the induction hypothesis on $m$ we know $\Lambda^*_{\bold
  x}(\cU_{\gamma,\beta},\bold u_{\gamma,\beta})$ is
$(\infty,\theta^+_0,u_{\gamma,\beta})$-free so there is a witness
$(\bar\Lambda^*_{\gamma,\beta},g^*_{\gamma,\beta},\bar
h^*_{\gamma,\beta})$, i.e. is as in \ref{a12}(4)(d) for $k=1$,
in particular:
\mn
\begin{enumerate}
\item[$(*)_8$]  $\bar\Lambda^*_{\gamma,\beta} = \langle
  \Lambda^*_{\gamma,\beta,\zeta}:\zeta < \zeta_{\gamma,\beta}\rangle$
  is a partition of $\Lambda^*_{\bold x}(\cU_{\gamma,\beta},\bold
  u_{\gamma,\beta})$. 
\end{enumerate}
\mn
We define
\mn
\begin{enumerate}
\item[$(*)_9$]
\begin{itemize}
\item  $\varepsilon_\beta =
\varepsilon_{\beta_*} + \zeta_{\gamma,\beta}$
\sn
\item  $\Lambda_{\varepsilon_{\beta_*} + \zeta} =
  \Lambda^*_{\gamma,\beta,\zeta}$ for $\zeta < \zeta_{\gamma,\beta}$
\sn
\item  $g_\beta(\varepsilon_{\beta_*} + \zeta) =
  g^*_{\gamma,\beta}(\zeta)$ for $\zeta < \zeta_{\gamma,\beta}$,
i.e. $g_\beta$ is the function with domain
$\varepsilon_\beta$ extending $g_{\beta_*}$ and defined on
$[\varepsilon_{\beta_*},\varepsilon_\beta)$ as above
\sn
\item  $h_{\beta,k}$ is a function with domain
$\Lambda^*_{\bold x}(\cU_{\gamma,\beta,\bold u}) =
\cup\{\Lambda_\varepsilon:\varepsilon < \varepsilon_\beta\}$ 
extending $h_{\gamma,\beta_*,k}$ 
\sn
\item  $h_{\beta,k}(\bar\eta) =
  h^*_{\gamma,\beta,k}(\bar\eta)$ if $\bar\eta \in \Lambda^*_{\bold x}
(\cU_{\gamma,\beta},\bold u_{\gamma,\beta})$.
\end{itemize}
\end{enumerate}
\mn
Now check, notice that if $\xi < \varepsilon_{\beta_*} \le \varepsilon
< \varepsilon_\beta$ and $\bar\nu \in \Lambda_{\gamma,\xi}$ and $\bar\eta
\in \Lambda_{\gamma,\varepsilon} = \Lambda^*_{\gamma,\beta,\varepsilon
  - \varepsilon_{\beta_*}}$ and $m = g_\beta(\varepsilon) =
g^*_{\beta,\gamma}(\varepsilon - \varepsilon_{\beta_*})$ then $m \ne
k(\beta_*)$ and $\eta_{k(\beta_*)} \ne \nu_{k(\beta_*)}$, so no
problem arise and the rest should be clear.
\end{PROOF}

\noindent
In what follows we assume $\ell < \bold k \Rightarrow \partial_\ell = 
\partial$ to simplify, anyhow we have not sorted out what
occurred to $(B)(d)$ when $\bar\partial$ is not constant
\begin{theorem}  
\label{a51}
If (A) then (B) where:
\mn
\begin{enumerate}
\item[$(A)$] 
\begin{enumerate}
\item[(a)]  $\bar\partial = \langle \partial_\ell:\ell <
  \bold k\rangle$ such that $\ell < \bold k \Rightarrow 
\partial_\ell = \partial = \cf(\partial)$
\sn
\item[(b)]  $\mu_\ell \in \bold C_{\partial_\ell}$ for
  $\ell < \bold k$, see \ref{y12}, \ref{y16}
\sn
\item[(c)]  $\mu_\ell < \mu_{\ell +1}$ for $\ell < \bold k$
\sn
\item[(d)]  $\chi_\ell = 2^{\mu_\ell}$
\sn
\item[(e)]  $J_\ell = J^{\bd}_{\partial_\ell}$ for $\ell < \bold k$
  \oor \, for some regular $\sigma, \bigwedge\limits_{\ell} \sigma
  < \partial_\ell$ and $J_\ell = J^{\bd}_{\partial_\ell} \times
J^{\bd}_\sigma$ for $\ell < \bold k$ 
\end{enumerate}
\sn
\item[$(B)$]  there is $\bold x$ such that:
\sn
\begin{enumerate}
\item[(a)]  $\bold x$ is a combinatorial
  $\bar\partial$-parameter of cardinality
 $\le \chi_{\bold k-1}$ with $J_{\bold x,\ell} = J_\ell$
\sn
\item[(b)]  $\bold x$ has a $\bar\chi$-black box
\sn
\item[(c)]  $\bold x$ is $(\theta_*,\theta^+)$-free \when
  \, $n(*) \ge 1,\theta = \cf(\theta) \ge \partial,\theta_* = 
\theta^{+ (\partial \cdot n(*))} < \mu_0$ and 
$3n(*)+4 < \bold k$
\sn
\item[(d)]  $\bold x$ is $\theta_{**}$-free when $\theta_{**} =
  \partial^{+(\partial \cdot n(*)+\partial)} < \mu_0,3n(*) +4 
< \bold k,n(*) \ge 1$.
\end{enumerate}
\end{enumerate}
\end{theorem}

\begin{remark}
\label{a42}
Note that the proof is somewhat easier when
$\theta^{+ \partial(n(*)+1)} < \mu_0$ and the loss is minor.
\end{remark}

\begin{PROOF}{\ref{a51}}  
For each $\ell < \bold k$ we can choose $\bold x_\ell$ such that:
\mn
\begin{enumerate}
\item[$\oplus$] 
\begin{enumerate}
\item[(a)]  $\bold x_\ell$ is a combinatorial $\langle
  \partial_\ell\rangle$-parameter 
\sn
\item[(b)]  $\bold x_\ell$ is $(\theta^{+\partial +1},
\theta^{+4})$-free when $\partial \le \theta < \mu_0$
\sn
\item[(c)]  $\bold x_\ell$ has a $\chi_\ell$-pre-black box, moreover
\sn
\item[(c)$^+$]  $\bold x_\ell$ has $\chi_\ell$-black box
\sn
\item[(d)]  $\Lambda_{\bold x_\ell}$ has cardinality $\chi_\ell$
\sn
\item[(e)]  $\bold x_\ell$ is $\partial^+$-free.
\end{enumerate}
\end{enumerate}
\mn
[Why?  By \cite[0.4,0.5,0.6=y19,y22,y40]{Sh:1008}, when we weaken
clause $\oplus(i)$ to $\bold x_\ell$ has a $\chi_\ell$-pre-black box;
anyhow  we elaborate
(also when $\partial = \aleph_0$ we have to say a little more) so let
$\ell < \bold k$ and $\mu = \mu_\ell,\lambda = \chi_\ell$.

First, assume that there is a $(\mu^+,J^{\bd}_\partial)$-free subset $\cF$
of ${}^\partial(\mu)$ of cardinality $\lambda = 2^\mu$.  We define
$\bold x_\ell$ by $\Lambda_{\bold x_\ell} = \{\langle \eta
\rangle:\eta \in \cF\},J_{\bold x_\ell} = J^{\bd}_\partial$.

Now $\bold x_\ell$ has $\lambda$-black box (by \cite[\S3]{Sh:898});
easy as the number of functions from ${}^{\partial >}(\mu)$ to
$\lambda$ is $\lambda^\mu = \lambda$).  Note also that
$\bold x_\ell$ is tree-like; this is enough for
$\oplus(a),(b),(c),(d),(e)$.  \Wilog \, there is a list $\langle
\eta_\alpha:\alpha < \lambda\rangle$ of the elements of
$\cF$ such that $\alpha < \beta
\Rightarrow \eta_\alpha <_{J^{\bd}_\partial} \eta_\beta$ (see the
proof of \cite[3.10=L1f.28]{Sh:898}.  Let $\langle \cU_\alpha:\alpha <
\lambda\rangle$ be a sequence of pairwise disjoint subsets of
$\lambda$ each of cardinality $\lambda$ such that $\min(\cU_\alpha) >
\mu^\omega \cdot \alpha$ and let $\cF_\alpha = \{\eta_\beta:\beta \in
\cU_\alpha\}$ and $\nu_\alpha = \eta_\beta$ when $\alpha \in [\mu \cdot
\beta,\mu \cdot \beta + \mu)$ and $\cF_* = \bigcup\limits_{\alpha}
\cF_\alpha$.  Now we choose $\Lambda_{\bold x_\ell} = \{\langle
\eta\rangle:\eta \in \cF_*\},\Lambda^*_\alpha = \{\langle
\eta\rangle:\eta \in \cF_\alpha\}$ so $\langle \nu_\alpha:\alpha <
\lambda\rangle$ witness $\bold x_\ell$ has $\bar\chi$-black box.

Second, assume that there is no $\cF$ as above, it follows
that $\lambda = 2^\mu$ is regular (see \cite[0.4]{Sh:1008} or
\cite[\S3]{Sh:898} using the
``no hole claim" combining).  Note that if there is a
$\langle \partial\rangle$-c.p. $\bold x$ which is
$(\theta_2,\theta_1)$-free, $\Lambda_{\bold x} \subseteq {}^\partial
\mu$ pedantically $\Lambda_{\bold x} \subseteq \{\langle
\eta\rangle:\eta \in {}^\partial \mu\}$ and $|\Lambda_{\bold x}| =
2^\mu,J_{\bold x} \supseteq J^{\bd}_\partial$, \then \, there is such
$\bold y$ with $J_{\bold y} = J^{\bd}_\partial$ and as above both have
the $\lambda$-BB.

Now as $\lambda = \cf(\lambda) = 2^\mu,\mu \in \bold C_\partial$,
there is a sequence $\langle \lambda_i:i < \partial\rangle$ of regular
cardinals $< \mu$ and $\partial$-complete ideal $J$ on $\partial$
extending $J^{\bd}_\partial$ such that $\chi = \tcf(\prod\limits_{i
  < \partial} \lambda_i,<_J)$ so let $\langle \eta_\alpha:\alpha <
\chi\rangle$ be $<_J$-increasing cofinal in $(\prod\limits_{i
  < \partial} \lambda_i,<_J)$.  By \cite[0.1=L41]{Sh:1008} there is $S \in
\check I_{\theta^+}[\lambda]$ such that: if $\delta < \lambda \wedge
\cf(\delta) \ge \theta^{+4}$ then $\{\delta_1 < \delta:\cf(\delta_1) =
\theta^{+3}$ and $S \cap \delta_1$ is a stationary subset of
$\delta_1\}$ is stationary in $\delta$, note that there $\cf(\delta) =
\theta^{+4}$, but the general case of $\cf(\delta) \ge \theta^{+4}$
follows.

Assume $\lambda = \cf(\lambda),S \subseteq \lambda,\sup(S) = \lambda$
and we recall some things from \cite{Sh:1008}, $\bar f = \langle
f_\alpha:\alpha < \lambda\rangle$ is $<_J$-increasing, $J$ an ideal on
$\partial,f_\alpha:\partial \rightarrow \Ord$ and $u_\alpha \subseteq
\alpha$ for $\alpha < \lambda$, we say $\bar f$ obeys the sequence of
sets $\bar u = \langle u_\alpha:\alpha < \lambda\rangle$ \when \, for
every $\beta \in A_\alpha$ we have $\bigwedge\limits_{\gamma \in
  \dom(f)} f_\beta(\gamma) < f_\alpha(\gamma)$ and if $\alpha \in S$
is a limit ordinal then $f_\alpha(\gamma) = \sup_{\beta \in A_\alpha}
(f_\beta(\gamma)+1)$ for every $\gamma \in \dom(f)$.

For $\theta = \cf(\theta) < \lambda$, we say $\bar u$ as above is a
witness for $S \in \check I_\theta[\lambda]$ \when \,:
\mn
\begin{itemize}
\item  $\alpha \in \cS \Rightarrow \cf(\alpha) = \theta$
\sn
\item  $\alpha < \lambda \Rightarrow |u_\alpha| < \theta$
\sn
\item  $\alpha \in u_\beta \Rightarrow u_\alpha = u_\beta \cap \alpha$
\sn
\item  there is a club $E$ of $\lambda$ such that
if $\delta \in S \cap E$ then $u_\alpha$ is an unbounded subset
 of $\alpha$ of order type $\theta$.
\end{itemize}
\sn
We say $\bar f$ is good in a limit ordinal $\delta < \lambda$ \when \,
there are $u \subseteq \delta = \sup(\delta)$ and $\bar w = \langle
w_\alpha:\alpha \in u\rangle \in {}^u J$ such that $\alpha \in u
\wedge \beta \in u \wedge \alpha < \beta \wedge i \in \partial
\backslash (w_\alpha \cup w_\beta) \Rightarrow f_\alpha(i) < f_\beta(i)$.
 So \wilog \, $\bar f$
obeys a witness for $S \in \check I_{\partial^+}[\lambda]$, hence is
good in $\delta$ when $\delta \in S \vee (S \cap \delta)$ is a
stationary subset of $\delta$ and $\cf(\delta) \in
(\theta,\theta^{+\partial})$.  Hence $\{f_\alpha:\alpha < \lambda\}$
is $(\theta^{+\partial+1},\theta^{+4})$-free for every $\theta
\ge \partial$, see \cite[0.4=Ly19]{Sh:1008}; in more details, in
\cite[0.4=Ly19]{Sh:1008} we conclude (A) or (B), now (A) there is
stronger whereas if (B) ther holds see \cite[0.6(e)=Ly40(e)]{Sh:1008}.

Together we are done except for $\oplus(c)^+$ which is proved as in
the ``first".]

So we have finished proving $\oplus$. 

Let $\bold x = \bold y_{\bold k}$ where for $m \in \{1,\dotsc,\bold
k\}$ we let $\bold y_m = \bold x_0 \times \bold x_1 \times \ldots \times \bold
x_{m-1}$ and we shall show it is as required.

\underline{Clause (B)(a)} which says ``$\bold x$ is a combinatorial
$\bar\partial$-parameter of cardinality $\chi_{\bold k-1}$", 
holds by \ref{a30}(1), i.e. we can prove
``$\bold y_m$ is a $\langle \partial_\ell:\ell < m\rangle$-c.p. of
cardinality $\chi_{m-1}$" by induction on $m = 1,\dotsc,\bold k$.

\underline{Clause (B)(b)} which says ``$\bold x$ has a $\bar\chi$-BB" holds 
 by \ref{a34}, that is, again by induction on $m = 1,\dotsc,\bold k$ we can
 prove that $\bold y_m$ has the $\langle \chi_\ell:\ell < m\rangle$-BB.

We now shall prove:
\newline
\underline{Clause (B)(c)}: we deduce it from \ref{a41} +
$\oplus(b)$.  We are given $\theta,n(*)$ as there.  
Let $\langle \theta_m:m \le m(*)\rangle$ be defined by: $m(*) =
3n(*)+4,\theta_\iota := \theta^{+ \iota}$ for $\iota = 0,1,2,3$ and
$\theta_{3+3m + \iota} := (\theta_{3+3m})^{+(\partial + \iota)}$ 
for $\iota=1,2,3$ when $m < n(*)$ and 
$\theta_{m(*)} := \theta^{+\partial +1}_{3n(*)+4} < \mu_0$; the ``$\le
\mu_0$" holds by the assumption of clause (B)(c).  Note 
that if $\theta_{m+1} = \theta^+_m$ then ``$\bold x_\ell$ being
$(\theta_{m+1},\theta^+_m)$-free" is trivial.

To apply Theorem \ref{a41} with $\bold x_\ell$ as in $\oplus$
above, $\bold x$ as above, $\bold m = m(*),u = \{0,\dotsc,\bold k-1\}$
has $\bold m$ members and $\theta_\ell$ for $\ell \le \bold m$ as
above; we have to verify clauses (a)-(f) of $\boxplus$ of \ref{a41}.

Now clause (a) stating $\bold x_\ell$ is a combinatorial
$\langle \partial_\ell \rangle$-parameter holds by $\oplus(a)$.

Now clause (b) stating $\bold x = \bold x_0 \times \ldots \bold
x_{\bold k-1}$ holds by the choice of $\bold x$ above.

Clause (c) stating ``$u \subseteq \{0,\dotsc,\bold k-1\}$ and $\bold m
= |u|>0$" holds by the choice of $u$ and the assumption on $m(*)$.

Clause (d) stating ``$\theta_0 < \ldots < \theta_{\bold m}$"
holds by the choice of the $\theta_\ell$'s above.  Notice that each
$\theta_\ell (\ell > 0)$ is a successor and hence regular. 

Clause (e) stating ``$\partial_{\bold x_\ell} \le \theta_0$ for $\ell <
\bold k$", this holds because $\theta_0 = \theta \ge \partial = \partial_\ell$
for $\ell < \bold k$.

Clause (f) stating ``$\bold x_\ell$ is
$(\theta_{m+1},\theta^+_m)$-free" ``when $\ell \in u,m < \bold m$" holds,
we check this by cases.
\medskip

\noindent
\underline{Case 1}:  $\theta_{m+1} = \theta^+_m,(f)$ holds trivially.
\medskip

\noindent
\underline{Case 2}:  $m=3,(\theta_{m+1},\theta^+_m) =
(\theta^{+(\partial +1)},\theta^{+4})$ holds by clause (b) of $\oplus$.
\medskip

\noindent
\underline{Case 3}:  $m=3n+3$ where $n < n(*)$ 
so $(\theta_{m+1},\theta^+_m) =
(\theta^{\partial \cdot (n+1)+1},\theta^{\partial \cdot n+4})$

By clause (b) of $\oplus$ above applied to $\theta
= \partial^{+ \partial \cdot n}$.  

So all clauses of $\boxplus$ of
Theorem \ref{a41} hold, hence its conclusion which says $\bold x$ in
$(\theta_{\bold m},\theta^+_0)$-free but $\theta_{\bold m} =
\theta_*$ and $\theta_0 = \partial$, so we are done proving clause (c)
of \ref{a51}(B).
\medskip

\noindent
\underline{Clause (B)(d)}:  says that ``$\bold x$ is $\theta_*$ free"
assuming $\theta_* = \partial^{+(\partial \cdot n(*) + \partial)} 
< \mu_0,3m +4 < \bold k$ and 
$\bigwedge\limits_{\ell < \bold k} \partial_\ell = \partial$.
We will deduce it from clause (B)(c) by applying it choosing $\theta'_*
= \theta^{+ \partial \cdot n(*)+4},\theta  = \partial$ and
$m(*)=m$.  The assumptions in clause (c) holds: $\theta = \partial^+$ so
$\theta \ge \partial$ and $\theta'_*$ is as $\theta_*$ is
there and $\theta'_* < \mu_0$ by an assumption of clause (d) which
also says $3n(*)+ 4 < \bold k$.

So the conclusion of clause (c) holds, i.e. $\bold x$ is
$(\theta'_*,\theta)$-free.  But $\theta_* \le \theta'_*$ so $\bold x$
is $(\theta_*,\partial^+)$-free.  Also each $\bold x_\ell$ is
$\partial^+$-free by $\boxplus(e)$ hence by \ref{a15} the last two
statements implies $\bold x$ is $\theta_*$-free.
\end{PROOF}

\begin{conclusion}
\label{a52}
1) If $\sigma < \partial$ are regular and $\chi \ge \partial$ and
$n \ge 1$ \then \, there is an $\aleph_{\partial \cdot n}$-free, 
$m$-c.p. $\bold x$ for some $m$ which has the $\chi$-BB and
$|\Lambda_{\bold x}| < \beth_{\partial \cdot \omega}(\chi)$ and
$J_{\bold x,m} = J^{\bd}_\partial \times J^{\bd}_\sigma$.

\noindent
2) If $\sigma = \partial$ is regular and $\chi \ge \partial$ and
$n \ge 1$ \then \, there is an $\aleph_{\partial \cdot n}$-free
m-c.p. (for some $m$), $\bold x$ have the $\chi$-BB which is not free
(really follows) and $\Lambda_{\bold x}$ is not even 
the union of $\le \chi$ free subsets and $\bold x$ has cardinality 
$< \beth_{\partial \cdot \omega}(\chi) + \beth_{\omega_1}(\chi)$ 
and $J_{\bold x,m} = J^{\bd}_\partial$.

\noindent
3) If $m = 3n+5,\sigma = \cf(\sigma) < \partial =
\cf(\partial) < \chi < \mu_0 < \ldots < \mu_{m-1}$ and
$\mu_\ell \in \bold C_\partial$ for $\ell < m,\lambda_\ell =
\cf(2^{\mu_\ell}),S_\ell \subseteq \{\delta <
 \lambda_\ell:\cf(\delta) = \sigma\}$ stationary from 
$\check I_\sigma[\lambda_\ell]$ and $J =  J^{\bd}_\partial \times 
J^{\bd}_\sigma$ \then \, we have (A) or (B), where:
\mn
\begin{enumerate}
\item[$(A)$]  for some $\ell$
\sn
\begin{enumerate}
\item[$(a)$]   there is an $\cF \subseteq
  {}^\partial(\mu_\ell)$ of cardinality $2^{\mu_\ell}$ which is
  $\mu^+_\ell$-free, i.e. is
$(\mu^+_\ell,J^{\bd}_\partial)$-free, see Definition \ref{y37}(1), and
even $(\mu^+_\ell,J)$-free,
\sn
\item[$(b)$]   hence letting $\bold x$ be the
1.-c.p. such that $\Lambda_{\bold x} = \{\langle \eta\rangle:\eta \in
\cF\}$, it is a $2^{\mu_\ell}$-BB for $\bold x$ which is
$\mu^+_\ell$-free and $J_{\bold x} = J$
\end{enumerate}
\sn
\item[$(B)$]   we can choose $\bold x = \bold x_0 \times \ldots \times
   \bold x_{m-1},\bold x_\ell$ is a 1-c.p., $\Lambda_{\bold x_\ell} =
\{\eta_{\ell,\delta}:\delta \in S_\ell\},\lim_{J_{\bold
   x_\ell}}(\eta_{\ell,\delta}) = \delta$, moreover $\eta_{\ell,\delta}$
is increasing with limit $\delta$ and $J_{\bold x_\ell} = J_\sigma
\odot J_\partial$ and $\bold x_\ell$ has the $\chi$-BB if $\chi < \mu_\ell$.
\end{enumerate}
\mn
4) Given $n,m,\sigma < \partial < \chi$ as in part (3), we can find
$\mu_\ell$ (and $\lambda_\ell,S_\ell$) as there such that:
\mn
\begin{enumerate}
\item[$(a)$]  if $\partial > \aleph_0$ then $\mu_\ell =
  \beth_{\partial \cdot (1 + \ell)}(\chi)$, we'll have ``$\bold x$ is
  $\theta_*$-free" we need $\chi \ge \theta$
\sn
\item[$(b)$]  if $\partial = \aleph_0$ for some club $E$ of $\omega_1$ and 
$\mu_\ell  \in \{\beth_\delta(\chi):\delta \in E\}$ are O.K.
\end{enumerate}
\end{conclusion}

\begin{PROOF}{\ref{a52}}
1) Let $\bold k = 3n+5$ and for $\ell < \bold k$ we let $\partial_\ell
= \partial,\mu_\ell = \beth_{\partial \cdot (1 + \ell)}
(\partial^{+(\partial \cdot n+1)} + \chi)$ and $\chi_\ell =
2^{\mu_\ell}$.  So each $\mu_\ell$ is strong limit of cofinality
$\partial = \cf(\partial) > \sigma \ge \aleph_0$, recalling \ref{y16}
we have $\mu_\ell \in \bold C_{\partial_\ell}$, i.e. clause (A)(b) of
Theorem \ref{a51} holds.

Clauses (A)(a),(c),(d),(e) of \ref{a51} are obvious hence there is
$\bold x$ as in clause (B) of \ref{a51}, in particular it is
$\partial^{+(\partial \cdot n+1)}$-free.  Also $\partial^{+(\partial
  \cdot n+1)} < \beth_{\partial \cdot \omega}(\chi)$ hence also
$\mu_\ell = \beth_{\partial \cdot (n+2)}(\partial^{+(\partial n+1)} +
\chi)$ is $< \beth_{\partial \cdot \omega}(\chi)$ hence
$|\Lambda_{\bold x}| \le 2^{\mu_{\bold k}-1} < \beth_{\partial \cdot
  \omega}(\chi)$, so we are done.

\noindent
2) If $\partial > \aleph_0$, the proof of part (1) holds and
$|\Lambda_{\bold x}| < \beth_{\partial \cdot \omega}(\chi)$.  If
$\partial = \aleph_0$, we know (see \cite{Sh:g}) that there is a club
$E$ of $\omega_1$ consisting of limit ordinals 
such that $\delta \in E \Rightarrow
\beth_\delta(\chi) \in \bold C_\partial$.  We define $\bold
k,\partial_\ell$ as above and for $\ell < \bold k$ let $\delta_\ell$
be the $\ell$-th member of $E$ and let $\mu_\ell
=\beth_{\delta_\ell}(\chi)$, and we continue as in the proof of part (1). 

\noindent
3) This is straightforward by \cite{Sh:898} but we elaborate to some
extent.  First assume that for some $\ell < \bold k$ 
clause (A)(a) of \ref{a52}(3) holds, so $\bold x$ from (A)(b) 
is a well defined 1-c.p. and is
$\mu^+_\ell$-free and letting $\chi = 2^{\mu_\ell}$ there is a
$\chi$-BB for $\bold x$ because the number of $h:{}^{\partial
  >}(\mu_\ell) \rightarrow \chi$ is $\le \chi^{\mu_\ell} = \chi$, and
diagonalizing we can choose a $\chi$-pre-BB for $\bold x$ (see
\ref{a20}).  To get a $\chi$-BB we work as in the proof of \ref{a34}(2).

So assume there is no such $\ell$.  Then for each $\ell$, we know that
$\lambda_\ell = 2^{\mu_\ell}$ is regular 
(see \cite[3.10(3)=L1f.28,pg.39]{Sh:898}).  By the proof of $\oplus$ in
the beginning of the proof of \ref{a51}, there is $\bold x_{\ell,1}$ as
there so as $\Lambda_{\bold x_{\ell,1}} \subseteq
{}^\partial(\mu_\ell)$.  
By \cite[3.6=L1f.21]{Sh:898}, we know that $\alpha < \lambda_\ell
\Rightarrow |\alpha|^\sigma < \lambda_\ell$, hence obviously there is a
stationary set $S_\ell \subseteq \check I_\sigma[\lambda_\ell]$, 
(in fact, $\{\delta < \lambda_\ell:\cf(\delta) = \sigma\}$ belongs to
$\check I_\sigma[\lambda_\ell]$), see \cite[Claim 2.14]{Sh:420})
and \wilog \, $\delta \in S_\ell \Rightarrow \mu^\omega_\ell|\delta$.  

Hence we can find $\bar\nu = \langle \nu_\delta:\delta \in
S_\ell\rangle$ such that:
\mn
\begin{enumerate}
\item[$\bullet$]  $\nu_\delta \in {}^\sigma \delta$ is increasing with
  limit $\delta$
\sn
\item[$\bullet$]  $\nu_{\delta_1}(i_1) = \nu_{\delta_2}(i_2)
  \Rightarrow \ell_1 = \ell_2 \wedge \nu_{\delta_1} \rest i_1 =
  \nu_{\delta_2} \rest i_2$
\sn
\item[$\bullet$]  $\nu_\delta(i)$ is divisible by $\mu_\ell$.
\end{enumerate}
\mn
Let $\langle \rho_\delta:\delta \in S_\ell\rangle$ list
$\Lambda_{\bold x_\ell,1}$ and for $\delta \in S_\ell$ let $\eta_\delta
\in {}^\partial \delta$ be: if $i < \partial,j < \sigma$ then
$\eta_\delta(\sigma i+j) = \nu_\delta(j) + \rho_\delta(i)$.  We define
$\bold x_\ell$ by $\Lambda_{\bold x_\ell} = \{\eta_\delta:\delta \in
S_\ell\},J_{\bold x_\ell} = J_\sigma \circledast J^\delta_\partial$, etc.
Now
\mn
\begin{enumerate}
\item[$(*)$]  $\bold x_\ell$ is a $\langle \partial \rangle$-pre-BB of
  cardinality $\chi_\ell$, with the freeness properties from \ref{a51}.
\end{enumerate}
\mn
What about $\chi$-pre-BB?  By \cite[\S3]{Sh:898} this holds whenever $\chi
< \mu_\ell$, which is enough for applying.  To get $\chi$-BB let
$\langle \delta(\zeta):\zeta < \lambda\rangle$ list $S_\ell$ in increasing
order and let $\langle S_\alpha:\alpha < \lambda_\ell\rangle$ be a
sequence of pairwise disjoint stationary subsets of $S_\ell$ such that
$\min(S_\alpha) > \delta(\alpha)$.  Let $\nu_\xi =
\eta_{\delta(\zeta)}$ when $\zeta \cdot \mu \le \xi < \zeta \cdot \mu
+ \mu$.  

We define: $\Lambda_\alpha = \Lambda^\ell_\alpha = 
\{\eta_\delta:\delta \in S_\alpha\}$ so
for each $\alpha$ there is a $\chi_\ell$-pre-BB for $\Lambda_\alpha$
and we continue as in the proof of \ref{a51}.
We now continue as in part (1) by inside the proof of \ref{a51}.

\noindent
4) By the proofs above this should be clear.
\end{PROOF}

\begin{discussion}  
\label{a53}
1) The following statement appears in \cite[0.4=Ly19]{Sh:1008}.  If
$\sigma = \cf(\sigma) < \kappa = \cf(\kappa)$ and $\mu \in \bold
C_\kappa$, then at least one of the following holds:
\mn
\begin{enumerate}
\item[(A)]   there exists a $\mu^+$-free $\cF \subseteq {}^\kappa \mu$
  of cardinality $\lambda = 2^\mu$
\sn
\item[(B)]   $\lambda = 2^\mu$ is regular and there is a
  $(\lambda,\mu,\sigma,\kappa)-5$-solution.
\end{enumerate}
\mn
If (A) holds, then we get more than promised
(i.e. $\mu^+_\ell$-freeness).  Hence we may assume, without loss of
generality, that (B) holds.  We shall return to this point (and then
recall the definition of 5-solution).

\noindent
2) We can vary the definition of the BB, using values in $\chi$ or using
models.

\noindent
3) We can use just product of two combinatorial parameters but with
   any $\bold k_{\bold x}$.  At present this makes no real difference.
\end{discussion}

\begin{discussion}  
\label{a54}
Assume $\bold x$ is a combinatorial $\bar\partial$-parameter, $\bar\partial =
\bar\partial_{\bold x}$ and $\bar\partial' = \langle
\partial'_\ell:\ell < \bold k_{\bold x}\rangle$ is a sequence of
limit ordinals such that $\ell < \bold k \Rightarrow \cf(\partial'_\ell) =
\partial_\ell$.

It follows that there is $\bold y$ such that:
\mn
\begin{enumerate}
\item[$(*)$]
\begin{enumerate}
\item[(a)]  $\bold y$ is a combinatorial $\bar\partial'$-parameter
\sn
\item[(b)]  $\bold S_{\bold y,\ell} = \{\partial'_\ell 
\alpha + i:\alpha \in S_{\bold x,\ell}$ and $i < \partial'_\ell\}$
\sn
\item[(c)]  $\Lambda_{\bold y} = \{g(\bar\eta):\bar\eta
  \in \Lambda\}$ where
\sn
\item[(d)]  $g:\bar S^{[\bar\partial]}_{\bold x}
  \rightarrow \bar S^{[\bar\partial']}_{\bold y}$ is defined as
  follows: for each $\ell < \bold k$ for some 
increasing continuous sequence 
$\langle \varepsilon_{\ell,i}:i \le \partial_\ell\rangle$
of ordinals with $\varepsilon_{\ell,0} =
  0,\varepsilon_{\ell,\partial_\ell} = \partial'_\ell$ we have 
 $g(\bar\eta) = \bar\nu$ iff $\bar\eta = \langle \eta_\ell:\ell <
  \bold k\rangle,\bar\nu = \langle \nu_\ell:\ell < k\rangle$
 and $\varepsilon_{\ell,i} \le \varepsilon < \varepsilon_{\ell,i+1}
  \Rightarrow \nu_\ell(\varepsilon) = \partial'_\ell \cdot \eta_\ell(i) +
  \varepsilon$ (of course, we could have ``economical")
\sn
\item[(e)]   if $\bold x$ has $\bar\chi-\BB$ and
  $\chi_\ell = \chi^{\partial'_\ell}_\ell$ for $\ell < \bold k$ then
  $\bold y$ has $\bar\chi-\BB$.
\end{enumerate}
\end{enumerate}
\end{discussion}

\begin{definition}
\label{a61}
We say a $\bold k$-c.p. $\bold x$ is $(\theta,\sigma)$-well orderable
$(\bar\chi,\bold k,1)$-BB \when \, there is a witness $\bar\Lambda$ which
means:
\mn
\begin{enumerate}
\item[$(a)$]  $\bar\Lambda = \langle \Lambda_\alpha:\alpha <
  \delta\rangle$
\sn
\item[$(b)$]  $\bar\Lambda$ is increasing continuous
\sn
\item[$(c)$]  $\cf(\delta) \ge \sigma$ and $\delta$ is divisible by $\theta$
\sn
\item[$(d)$]  if $\alpha < \delta$ then $\bold x \rest
  (\Lambda_{\alpha +1} \backslash \Lambda_\alpha)$ has
$\bar\chi$-pre-black box
\sn
\item[$(e)$]  if $\alpha < \delta,\bar\eta \in \Lambda_{\alpha +1}
  \backslash \Lambda_\alpha$ and $m < \bold k$ then the following set belongs
  to $J_{\bold x,m}$:
\sn
\begin{itemize}
\item    $\{i < \partial_{\bold x,m}$: for some $\bar\nu \in
  \Lambda_\alpha$ we have $\bar\eta \upharpoonleft (m,i) = \bar\nu
  \upharpoonleft (m,i)\}$.
\end{itemize}
\end{enumerate}
\end{definition}

\begin{claim}
\label{a63}
1) In Theorem \ref{a51}, for any $\theta = \cf(\theta) \le \chi_{\bold k-1}$
Clause (B)(b) can be strengthened to: $\bold x$ has $\theta$-well
orderable $\bar\chi$-black box.

\noindent
2) Parallely in Conclusion \ref{a52}.
\end{claim}
\newpage

\section {Building Abelian groups and modules with small dual} \label{Building}
\bigskip

\noindent
For transparency we restrict ourselves to hereditary rings.
\begin{convention}
\label{d0}
1) All rings $R$ are hereditary, i.e. if $M$ is a free $R$-module then
any pure sub-module $N$ of $M$ is free.

\noindent
2) An alternative is to interpret ``$G$ is a $\theta$-free ring" by
demanding $\cf(\theta) > \aleph_0$ and in the game of choosing $A_n \in [G]^{<
  \theta}$ increasing with $n$, the even player can guarantee the
sub-module $\langle \bigcup\limits_{n} A_n\rangle_G$ of $G$ is free.

We shall try to use a $\bar\partial$-BB to construct Abelian groups
and modules.  In \ref{d2} we present a quite clear case: the case
$\bigwedge\limits_{\ell} \partial_\ell = \aleph_0$, the ring is $\bbZ$ 
(and the equations are simple).  
Note that the addition of $z$ (in \ref{d2}(1)(b),
\ref{d4}(1)(a)) is natural when we are trying to prove $h \in \Hom(G,\bbZ)
\Rightarrow h(z)=0$ which is central in this section, but is not
natural for treating some other questions.  When dealing with $\TDC_\lambda$ we
may restrict ourselves to $G$ simply derived from $\bold x$, see
\ref{d2}(3), so can ignore \ref{d2}(1A),(2).
\end{convention}

\begin{definition}  
\label{d2}
Let $\bold x$ be a tree-like\footnote{In \cite{Sh:883} this was not
  necessary as the definition of $\eta \upharpoonleft (m,n)$ there is
  $\eta \upharpoonleft (m,<n+1)$ here.} (see Definition \ref{a3}(1))
combinatorial $\bar\partial$-parameter and let
$\bold k = \bold k_{\bold x}$.

\noindent
1) If $k < \bold k_{\bold x} \Rightarrow \partial_\ell = \aleph_0$, \then \,
   we say an Abelian group $G$ is derived from $\bold x$ \when \,
\mn
\begin{enumerate}
\item[$(a)$]  $G$ is generated by $X \cup Y$ where:
\sn
\begin{enumerate}
\item[$(\alpha)$]  $X = \{x_{\bar\eta \upharpoonleft (m,n)}:
\bar\eta \in \Lambda_{\bold x},
m < \bold k_{\bold x} \text{ and } n \in \bbN\} \cup \{z\}$
\sn
\item[$(\beta)$]  $Y = \{y_{\bar\eta,n}:\bar\eta \in 
\Lambda_{\bold x} \text { and } n \in \bbN\}$
\end{enumerate}
\sn
\item[$(b)$]  moreover generated freely except the following set of
equations

\begin{equation*}
\begin{array}{clcr}
\Xi_{\bold x} = \{(n+1) y_{\bar\eta,n+1} = y_{\bar\eta,n} -
&\Sigma\{x_{\bar\eta \upharpoonleft (m,n)}: m < \bold k\} -
a_{\bar\eta,n} z_{\bar\eta}: \\
  &\bar\eta \in \Lambda_{\bold x} \text{ and } n \in \bbN\}
\end{array}
\end{equation*}
\end{enumerate}
\mn
where
\mn
\begin{enumerate}
\item[$\bullet_1$]  $z_{\bar\eta} \in \Sigma\{\bbZ x_{\bar\eta
\upharpoonleft (m,n)}:\bar\eta \in \Lambda_{\bold x},m < \bold k,n \in
  \bbN\} \oplus \bbZ z$
\sn
\item[$\bullet_2$]  $a_{\bar\eta,n} \in \bbZ$.
\end{enumerate}
\mn
1A) We say the Abelian group $G$ is canonically derived from $\bold x$
\when \, above we omit the 
$z_{\bar\eta}$'s equivalently $a_{\bar\eta,n}=0$.  If
we omit $z$ we say strictly derived.

\noindent
2) We say the derivation of $G$ in part (1) is
well orderable (or ``$G$ or $\langle z_{\bar\eta}:\bar\eta \in 
\Lambda_{\bold x}\rangle$ universally respect $\bold x$") \when \, we
replace $\bullet_1$ above by:
\mn
\begin{enumerate}
\item[$\bullet'_1$]  there is a list $\langle \bar\eta_\alpha:\alpha <
\alpha_*\rangle$ of $\Lambda_{\bold x}$ such that $z_{\bar\eta_\alpha}
\in \Sigma\{\bbZ x_{\bar\eta_\beta \upharpoonleft (m,n)}:\beta <
\alpha,m < \bold k\} \oplus \bbZ z$ for every 
$\alpha < \alpha(*)$; such a sequence is called a witness.
\end{enumerate}
\mn
3) We add simply (derived from $\bold x$) \when \, $z_{\bar\eta} = z$ for every
$\bar\eta$.  
\end{definition}

\begin{remark}  
\label{d3}
1) We can replace $(n +1) y_{\bar\eta,n+1}$ by 
$k_{\bar\eta,n} y_{\bar\eta,n +1}$ with
$k_{\bar\eta,n} \in \{2,3,\ldots\}$.

\noindent
2) By combining Abelian groups, the ``simply derived" is enough for
 cases of the $\TDC_\lambda$.  Instead of ``simply derived" we 
may restrict $\langle z_{\bar\eta}:\bar\eta \in \Lambda_{\bold x}
\rangle$ more than in \ref{d2}(2). 
\end{remark}

\noindent
A more general case than \ref{d2} is:
\begin{definition}  
\label{d4}
1) We say an $R$-module $G$ is derived from a combinatorial
   $\bar\partial$-parameter $\bold x$ \when \, ($R$ is a ring and):
\mn
\begin{enumerate}
\item[$(a)$]  $G_*$ is an $R$-module freely generated by
\newline
$X_* = \{x_{\bar\eta \upharpoonleft (m,i)}:m < \bold k_{\bold x},i <
\partial_m$ and $\bar\eta \in \Lambda_{\bold x}\} \cup \{z\}$
\sn
\item[$(b)$]  the $R$-module $G$ is generated by $\cup\{G_{\bar\eta}:\bar\eta
\in \Lambda_{\bold x}\} \cup X_*$, also $G_* \subseteq G$
\sn
\item[$(c)$]  $G/G_*$ is the direct sum of $\langle (G_{\bar\eta} +
G_*)/G_*:\bar\eta \in \Lambda_{\bold x}\rangle$
\sn
\item[$(d)$]  $Z_{\bar\eta} \subseteq X_* \subseteq G_*$ 
for $\bar\eta \in \Lambda_{\bold x}$; if $Z_{\bar\eta} 
= \{z_{\bar\eta}\}$ we may write $z_{\bar\eta}$ instead of $Z_{\bar\eta}$
\sn
\item[$(e)$]  if $\bar\eta \in \Lambda_{\bold x}$ \then \, the
$R$-submodule $G_{\bar\eta} \cap G_*$ of $G$ is generated (not only
included in the submodule generated) by
$\{x_{\bar\eta \upharpoonleft (m,i)}:m < \bold k_{\bold x}$ and 
$i < \partial_{\bold x,m}\} \cup Z_{\bar\eta} \subseteq X_*$.
\end{enumerate}
\mn
1A) We say $\gx$ is an $R$-construction or $(R,\bold x)$-construction
\when \, it consists of $\bold x,R,G_*,G,\langle x_{\bar\eta}:\bar\eta \in
  \Lambda_{\bold x,< \bold k}\rangle,\langle
  G_{\bar\eta},Z_{\bar\eta}:\bar\eta \in \Lambda_{\bold x}\rangle$ as
above and we shall write $\Lambda_{\gx} = \Lambda_{\bold x},
G^{\gx}_* = G_*,G_{\gx} = G,G_{\gx,\bar\eta} =
G_{\bar\eta}$, etc., (so in \ref{d2}(1) we have a 
$\bbZ$-construction with $G_{\bar\eta}/(G_{\bar\eta} \cap G_*)$ being
isomorphic to $(\bbQ,+))$.  We may say $\gx$ is for $\bold x$ but we 
may write $G$ rather than $G_{\gx}$, etc.
  when $\gx$ is clear from the context.

\noindent
1B) For an $R$-construction $\gx$ we 
say: ``universally respecting $\bold x$" or ``$\gx$ is well
orderable" \when \, we can find $\bar\Lambda$ which $\gx$ obeys meaning:
\mn
\begin{enumerate}
\item[$(f)$]  $(\alpha) \quad \bar\Lambda = \langle \Lambda_\alpha:\alpha \le
  \alpha_*\rangle$ is increasing continuous
\sn
\item[${{}}$]  $(\beta) \quad \Lambda_{\alpha_*} = \Lambda_{\bold x}$
  and $\Lambda_0 = \emptyset$
\sn
\item[${{}}$]  $(\gamma) \quad$ if $\bar\eta \in \Lambda_{\alpha +1} \backslash
  \Lambda_\alpha$ and $m < \bold k$ then $\{i < \partial_m$:

\hskip25pt $(\exists \bar\nu \in \Lambda_\alpha)
(\bar\eta \upharpoonleft (m,i) = \bar\nu 
\upharpoonleft (m,i)\} \in J_{\bold x,m}$
\sn
\item[${{}}$]  $(\delta) \quad$ if $\bar\eta \in \Lambda_{\alpha +1} \backslash
  \Lambda_\alpha$ then $Z_{\bar\eta} \subseteq \langle
  \{G_{\bar\nu}:\bar\nu \in \Lambda_\alpha\} \cup \{z\}\rangle_G$.
\end{enumerate}
\mn
1C) We may say ``$G$ is derived from $\bold x$" and $\gx$ is derived
from $\bold x$.

\noindent
1D) We add ``simple" or ``simply derived" when $z_{\bar\eta} = z$
hence $Z_{\bar\eta} = \{z\}$ for every $\bar\eta \in \Lambda$.

\noindent
1E) We say $\gx$ is almost simple if $|Z_{\bar\eta} \backslash \{z\}|
\le 1$.  

\noindent
2) Above we say $\gx$ is a locally free derivation \underline{or} locally free
\underline{or} $G$ in part (1) is freely derived \when \, in addition:
\mn
\begin{enumerate}
\item[$(g)$]  if $\bar\eta \in \Lambda_{\bold x},m < \bold k$ and $w
  \in J_{\bold x,m}$ then $(G_{\bar\eta}/G_{\bar\eta,m,w})$ is 
a free $R$-module where $G_{\bar\eta,m,w}$ is 
the $R$-submodule of $G$ generated by
$\{x_{\bar\eta \upharpoonleft (m_1,i_1)}:m_1 < k,i_1 < \partial_{m_1}$
and $m_1 = m \Rightarrow i_1 \in w\} \cup Z_{\bar\eta}$ so
$G_{\bar\eta} = G^\perp_{\bar\eta,m,w} \oplus G_{\bar\eta,m,w}$ for
some $R$-submodule $G^\perp_{\bar\eta,m,w}$ and let $\gx$ determine it.
\end{enumerate}
\mn
3) Above we say $\gx$ is $(< \theta)$-locally free or $\gx$ is a free
$(< \theta)$-derivation \when \,\footnote{So ``$\gx$ is locally free"
  does not imply ``$\gx$ is $\theta$-free" because of clause (h).}
 in addition to part (1):
\mn
\begin{enumerate}
\item[$(g)^+$]  like (g) but the quotient
  $G_{\bar\eta}/G_{\bar\eta,m,w}$ is $\theta$-free
\sn
\item[$(h)$]  $\bold x$ is $\theta$-free.
\end{enumerate}
\mn
4) We say $\gx$ is a canonical $R$-construction or canonical $(R,\bold
x)$-construction \when \, $\bar\eta \in \Lambda_{\bold x} \Rightarrow
Z_{\bar\eta} = \emptyset$.  We say canonically$^*$ when we omit $z$
and we write $G^-_{\gx}$.

\noindent
5) We say $\gx$ or just $(\bold x,\bar Z)$ where $\bar Z = \langle
Z_{\bar\eta}:\bar\eta \in \Lambda_{\bold x}\rangle$ is $\theta$-well
orderable \when \, for every $\Lambda \subseteq \Lambda_{\bold x}$ of
cardinality $< \theta$ there is $\langle \bar\eta_\alpha:\alpha <
\alpha_*\rangle,\Lambda' \supseteq \Lambda$ witnessing which means:
\mn
\begin{enumerate}
\item[(a)]  $\bar\eta_\alpha \in \Lambda_{\bold x}$ with no
  repetitions
\sn
\item[(b)]  if $\bar\eta \in \Lambda$ then
\sn
\begin{itemize}
\item  $\bar\eta = \bar\eta_\alpha$ for some $\alpha$
\sn
\item  $Z_{\bar\eta} \subseteq \{\bar\eta_\beta:\beta < \alpha\}$
\sn
\item  for some $m_* < \bold k$ and $w \in J_{\bold x,m}$ we have $i
  \in \partial_{m_*} \backslash w \Rightarrow \eta \upharpoonleft
  (m_*,i) \notin \{\bar\eta_\beta \upharpoonleft (m_i,j):m < \bold k,j
  < \partial_m\}$.
\end{itemize}
\end{enumerate}
\end{definition}

\begin{remark}
In Definition \ref{d4} we may like in $G_{\bar\eta}$ to have more
elements from $G_*$.  This can be accomplished by replacing
$x_{\bar\nu},\bar\nu \in \Lambda_{\bold x,< \bold k}$ by
$x_{\bar\nu,t}$ for $t \in T_{m,i}$ when $\bar\nu = \bar\eta
\upharpoonleft (m,i),\bar\eta \in \Lambda_{\bold x}$.

However, we can just as well replace $\partial_\ell$ by
$\partial'_\ell = \gamma \cdot \partial_\ell$ for some non-zero
ordinal $\gamma$ (and $J_\ell$ by $J'_\ell = \{w \subseteq \partial'_\ell$: for
some $u \in J_\ell$ we have $w \subseteq \{\gamma i + \beta:\beta <
\gamma$ and $i \in u\}\}$).  
\end{remark}

\begin{claim}  
\label{d6}
Assume $\gx$ is a simple $R$-construction (see \ref{d4}(1A),(1D)) which is a 
 $(< \theta)$-locally-free (see \ref{d4}(2) respectively) 
and $G = G_{\gx}$ so it is
derived from $\bold x$ and $\bold x$ is $\theta$-free.

\noindent
1) $G$ is a $\theta$-free $R$-module.

\noindent
2) If in addition $(R,+)$, that is $R$ as an additive (so Abelian) group, 
is free \then \, $(G,+),G$ as an Abelian group, is $\theta$-free.

\noindent
3) In part (2) it suffices that $(R,+)$ is a $\theta$-free Abelian
   group.

\noindent
4) In (1),(2),(3) we can replace ``derived" by ``$(< \theta)$-derived".

\noindent
5) Instead assuming ``$\gx$ is \underline{simply} derived" 
we can demand ``$\gx$ is well orderable and almost simple", see
Definition \ref{d4}(1B),(1E).
\end{claim}

\begin{PROOF}{\ref{d6}}  
1) Let $X \subseteq G$ have cardinality $< \theta$.
By the Definition \ref{d4}(1) there are $\Lambda \subseteq
   \Lambda_{\bold x}$ of cardinality $< \theta$ and $\Lambda_*
   \subseteq \Lambda_{\bold x,<\bold k}$ of cardinality $< \theta$ such
   that $X \subseteq \langle \{x_{\bar\eta}:\bar\eta \in \Lambda_*\} \cup
\{G_{\bar\eta}:\bar\eta \in \Lambda\}\rangle_G$, recalling $\{z\} =
Z_{\bar\eta} \subseteq G_{\bar\eta}$ for every $\bar\eta \in
\Lambda_{\bold x}$ so \wilog \, $X = 
\{x_{\bar\eta}:\bar\eta \in \Lambda_*\} \cup
\{Y_{\bar\eta}:\bar\eta \in \Lambda\}$ where $Y_{\bar\eta} \subseteq
G_{\bar\eta},|Y_{\bar\eta}| < \theta$ for $\bar\eta \in \Lambda$ and 
$[m < \bold k \wedge i < \partial_m \Rightarrow \bar\eta \upharpoonleft (m,i)
\in \Lambda_*]$. 
  
As $\bold x$ is $\theta$-free we can find the following objects:
\mn
\begin{enumerate}
\item[$(a)$]  $\langle \bar\eta_\alpha:\alpha < \alpha_*\rangle$ list 
$\Lambda$ 
\sn
\item[$(b)$]  $m_\alpha < \bold k_{\bold x}$ and $w_\alpha \in
J_{\bold x,\alpha}$ for $\alpha < \alpha_*$
\sn
\item[$(c)$]  if $\alpha < \beta$ and $i \in \partial_{\bold
x,m_\beta} \backslash w_\beta$ then $\eta_{\beta,m}(i) \ne
\eta_{\alpha,m}(i)$.
\end{enumerate}
\mn
For $\alpha \le \alpha(*)$, let $G_\alpha = \langle \cup
\{G_{\bar\eta_\beta}:\beta < \alpha\}\rangle_G$ and let
$G_{\alpha(*)+1} = \langle G_{\alpha(*)} \cup \{x_{\bar\eta}:\bar\eta
\in \Lambda_*\}\rangle_G$.

So $\langle G_\alpha:\alpha \le \alpha(*)+1 \rangle$ is an
increasing continuous sequence of sub-modules, $G_0 = 0$ and
$G_{\alpha(*)+1}$ includes $X$.   Also $G_{\alpha(*)+1}/G_{\alpha(*)}$
is free by their choice above.

Lastly, if $\alpha < \alpha(*)$ then
$G_{\alpha +1}/G_\alpha$ is a $\theta$-free $R$-module because
it is isomorphic to $G_{\bar\eta_\alpha}/G_\alpha =
G_{\bar\eta_\alpha}/G_{\bar\eta_\alpha,m_\alpha,w_\alpha}$ which 
is $\theta$-free by Definition \ref{d4}(3)(g)$^+$.  

So clearly we are done.  

\noindent
2),3) Follows.

\noindent
4) Similarly.

\noindent
5) Let $X,\Lambda$ and $\Lambda_*$ be as in the proof of part (1) and
let $\langle \bar\eta_\alpha:\alpha < \alpha_*\rangle$ list $\Lambda$.

Let $\bar\Lambda$ witness the well orderability of $\gx$.
Then (recalling Definition \ref{d4}(1B)) there is a function $h$ such that:
\mn
\begin{enumerate}
\item[$(d)$]  $h:\alpha_* \rightarrow \ell g(\bar\Lambda)$
\sn
\item[$(e)$]  if $\alpha < \alpha_*$ then $\bar\eta_\alpha \in
  \Lambda_{h(\alpha)+1} \backslash \Lambda_{h(\alpha)}$.
\end{enumerate}
\mn
Let
\mn
\begin{enumerate}
\item[$(f)$]  $Z_{\bar\eta_\alpha} \backslash \{z\} \subseteq
\{\nu_\alpha\} \subseteq 
\{\bar\eta_\beta \upharpoonleft (m,i):\bold n < \bold k,i < 
\partial_m$ and $\beta < \alpha\} \subseteq \Lambda_*$.
\end{enumerate}
\mn
Also \wilog \, as in \S1
\mn
\begin{enumerate}
\item[$(g)$]  $h$ is non-decreasing.
\end{enumerate}
\mn
Now as $\Lambda$ is $\theta$-free; as in \S1, looking carefully at
\ref{d4}(1B), \wilog \, $|\Lambda_{\alpha +1} \backslash
\Lambda_\alpha| \le 1$, so \wilog \,
\mn
\begin{enumerate}
\item[$(g)'$]  $h$ is the identity.
\end{enumerate}
\mn
The rest is as before.
\end{PROOF}

\begin{definition}
\label{d14}
1) An Abelian group $H$ is $(\theta_2,\theta_1)-1$-free \when \,: if $X
   \subseteq H,|X| < \theta_2$ then we can find a $\bar G$ such that:
\mn
\begin{enumerate}
\item[$\bullet$]  $\bar G = \langle G_\alpha:
\alpha < \alpha(*)\rangle$ is a sequence of subgroups of $G$,
\sn
\item[$\bullet$]  $G := \sum\limits_{\alpha < \alpha(*)} G_\alpha
\subseteq H$, both $G$ and $H$ include $X$,
\sn
\item[$\bullet$]  $G_\alpha$ is generated by a set of $< \theta_1$
members,
\sn
\item[$\bullet$]  $G = \bigoplus\limits_{\alpha < \alpha(*)} G_\alpha$.
\end{enumerate}
\mn
2) Similarly for $R$-modules. 
\end{definition}

\begin{claim}
\label{d18}
1) If $\bold x$ is a $\bold k$-c.p., $(\theta_2,\theta_1)$-free (see
\ref{a12}) and $\gx$ is a canonical $(R,\bold x)$-construction which is
locally free simply 
derived from $\bold x$ \then \, $G$ is $(\theta_2,\theta_1)-1$-free.

\noindent
2) Similarly for modules.
\end{claim}

\begin{PROOF}{\ref{d18}}
1) By (2) using $R = \bbZ$.

\noindent
2) Let $G = G_{\gx}$ such that $|\Lambda|,|\Lambda_*| < \theta_2$
and let $X \subseteq G$ be of cardinality $< \theta_2$.  Choose
$\Lambda,\Lambda_*$ as in the proof of \ref{d6}(1).  As we are
assuming ``$\bold x$ is $(\theta_2,\theta_1)$-free" and $\Lambda
\subseteq \Lambda_{\bold x}$ has cardinality $< \theta_2$, there is a
sequence $\langle \bar\Lambda,g,\bar h\rangle$ witnessing it, see
\ref{a12}(4)(d) such that $\bar\Lambda = \langle \Lambda_\gamma:\gamma <
\gamma(*)\rangle$ and $\Lambda = \bigcup\limits_{\gamma}
\Lambda_\gamma$.  We define the sequence $\langle G_\gamma:\gamma
\le \gamma(*)+1\rangle$ as follows.

For $\gamma < \gamma(*)$ let $G_\gamma$ be the submodule of $G_{\bold
  x}$ generated by
$\cup\{G^{\perp}_{\bar\eta,g(\gamma),h_\gamma(\bar\eta)}:\bar\eta \in
\Lambda_\gamma\}$.  We may assume that $G_0 = \{0\}$.  For $\gamma =
\gamma(*)$ let $G_\gamma$ be the submodule of $G_{\bold x}$ generated
by $\{x_{\bar\nu}:\bar\nu \in \Lambda$ but for no $\gamma <
\gamma(*),\bar\eta \in \Lambda_\gamma,i \in \partial_{\bold
  x,g(\gamma)} \backslash h_\gamma(\bar\eta)$ do we have $\bar\nu =
\bar\eta \upharpoonleft (g(\gamma),i)\}$.  Finally, for $\gamma =
\gamma(*)+1$ let $G_\gamma$ be $\sum\limits_{\gamma' \le \gamma(*)}
G_\gamma$.  

For every $\gamma \le \gamma(*)+1$ let $G_{<\gamma}$ be the submodule
generated by $\cup\{G_\alpha:\alpha < \gamma\}$.  Notice that the
sequence $\langle G_{< \gamma}:\gamma \le \gamma(*)+1\rangle$ is
increasing and continuous.
It suffices to prove that $G_\gamma \cap G_{<\gamma} = \{0\}$. If not,
then for some $n$ and pairwise distinct
$\bar\eta_0,\dotsc,\bar\eta_{n-1} \in
\Lambda_\gamma,(\sum\limits_{\ell < n}
G^\perp_{\bar\eta_\ell,g(\gamma),h_\gamma(\bar\eta_\ell)}) \cap
G_{<\gamma} \ne \{0\}$, see \ref{d4}(2).

If $0 \ne x \in (\sum\limits_{\ell \le n}
G^\perp_{\bar\eta_\ell,g(\gamma),h_\gamma(\bar\eta_\ell)}) \cap
G_{<\gamma}$ then there are $x_\ell \in
G^\perp_{\bar\eta_0,g(\gamma),h_\gamma(\bar\eta_\ell)}$ for $\ell < n$
such that $x = \sum\limits_{\ell < n} x_\ell$.  Recalling
``$G_{\gx}/G_* = \oplus\{G_{\bar\eta}/(G_{\bar\eta} \cap G_*):\bar\eta
\in \Lambda_{\bold x}\}$ necessarily $x \in G_*$; moreover recalling
\ref{d4}(1)(c) for each $\ell < n$ we have 
$x_\ell \in G_*$ so $x_\ell \in
G^\perp_{\bar\eta,g(\gamma),h_\gamma(\bar\eta_\gamma)} \cap G_*$ which
  is $\subseteq \oplus\{x_{\eta_\ell \rest (m,i)}:m = g(\gamma)$ and 
$i \in \partial_{\bold x,m} \backslash h_\gamma(\bar\eta_\ell)\} \oplus
  Rz$, see \ref{d4}(2).

Hence $x = \sum\limits_{\ell < n} x_\ell \in H_1 := \oplus\{x_{\bar\eta_\ell
 \rest (m,i)}:\ell < n,m = g(\gamma),i \in \bigcup\limits_{\ell_1 < n}
h_\gamma(\eta_{\ell_1})\}$.  
By the choice of $(\bar\Lambda,g,\bar h),H_2 := G_{< \gamma} \cap G_* \subseteq
\oplus\{R x_{\bar\nu}$: for some $\alpha < \gamma,\bar\eta \in 
\Lambda_\alpha,m = g(\alpha),i < \partial_{\bold x,g(\alpha)},i \notin
h_\alpha(\bar\eta),\bar\nu = \bar\eta \upharpoonleft (m,i)\}$.  Hence
$x \in H_1 \cap H_2 = \{0\}$, contradiction.
\end{PROOF}

\begin{claim}
\label{d21}
Assume $\bold x$ is an $(\aleph_0,\bold k)$-c.p. with $(\aleph_0,\bold
k)$-BB.

\noindent
1) There is canonical $\bbZ$-construction $\gx$ such that:
\mn
\begin{enumerate}
\item[$(a)$]  $G = G_{\gx}$ so $G$ is an Abelian group of cardinality
 $|\Lambda_{\bold x}|$
\sn
\item[$(b)$]  $G$ is  not Whitehead
\sn
\item[$(c)$]  $G$ is $\theta$-free if $\bold x$ is $\theta$-free
\sn
\item[$(d)$]  $G$ is $(\theta_2,\theta_1)-1$-free if $\bold x$ is
  $(\theta_2,\theta_1)$-free, see \ref{d14}(1)
\sn
\item[$(e)$]  $G$ has a $\bbZ$-adic dense subgroup of cardinality
  $|\Lambda_{\bold x,< \bold k}|$.
\end{enumerate}
\mn
2) We can add:
\mn
\begin{enumerate}
\item[$(b)^+$]  $\Hom(G,\bbZ) = 0$.
\end{enumerate}
\end{claim}

\begin{remark}
\label{d22}
Recall that ``$\bold b$ is a $(\chi,\bold k)$-BB" 
means $\bold b$ is a function with range $\subseteq \chi$, see
Definition \ref{a9}.
\end{remark}

\begin{PROOF}{\ref{d21}} 
1) Let $G_0 = \oplus\{\bbZ x_{\bar\eta}:\bar\eta \in 
\Lambda_{\bold x,< \bold k}\} \oplus \bbZ z$ 
and $G_1$ be the $\bbZ$-adic closure of $G_0$ so 
$G_1$ is a complete metric space under the $\bbZ$-adic metric.

For $\bar\eta \in \Lambda_{\bold x}$ and $\bar a \in {}^\omega \bbZ$
and $n(*)$, in $G_1$ we let $y_{\bar a,\bar\eta,n(*)} = \sum\limits_{n \ge
  n(*)}(n!/n(*)!)(\sum\limits_{m < \bold k} 
x_{\bar\eta \upharpoonleft (m,n)} - \sum\limits_{m < \bold k} \bold
b(\eta,m,n) z + a_n z)$.

Let $\{b_i:i < \omega\}$ list the elements of $\bbZ$ and let 
$\bar{\bold c} = \langle \bold c_{\bar\eta}:
\bar\eta \in \Lambda_{\bold x}\rangle$ be 
an $(\aleph_0,\bold k)$-BB with $\bold c_{\bar\eta}$ a function from
$\{\bar\eta \upharpoonleft (m,n):m < \bold k,n < \omega\}$ to $\bbZ$.
Now for each
$\bar\eta \in \Lambda_{\bold x}$ let $G^0_{\bar\eta} = \Sigma\{\bbZ
x_{\bar\eta \upharpoonleft (m,n)}:m < \bold k,n < \omega\} \oplus \bbZ
z$ and $h_{\bar\eta} \in \Hom(G^0_{\bar\eta},\bbZ,z)$ be such that
$h_{\bar\eta}(z)=z,h_{\bar\eta}(x_{\bar\eta \upharpoonleft (m,n)}) = 
b_{\bold c_{\bar\eta}(\bar\eta \upharpoonleft (m,n))} z$
\mn
\begin{enumerate}
\item[$(*)_1$]  We can choose $\bar a = \bar a[\bar\eta] \in 
{}^\omega \bbZ$ such that there is no extension $h^1 \in \Hom(G^1_{\bar a,\bar\eta},\bbZ)$ of
$h_{\bar\eta}$ where $G^1_{\bar a,\bar\eta} = \langle G^0_{\bar\eta}
\cup \{y_{\bar a,\bar\eta,n}:n < \omega\}\rangle_{G_1}$.
\end{enumerate}
\mn
[Why?  Well known but we elaborate.  It suffices to prove that $\cA =
\{\bar a \in {}^\omega 2:h_{\bar\eta}$ has an extension in
$\Hom(G^1_{\bar a,\bar\eta},\bbZ)$ and $a_0 = 0 = a_1\}$ is a countable 
subset of ${}^\omega \bbZ$, we could have allowed $\bar a \in
{}^\omega \bbZ$ but this seems more transparent to restrict ourselves.  
For $\bar a \in \cA$ let $h_{\bar a,\bar\eta}$ be an extension witnessing it.

Now
\mn
\begin{enumerate}
\item[$\bullet$]  For each $b \in \bbZ$ the set $\cA_b = 
\{\bar a \in \cA \subseteq {}^\omega 2:
h_{\bar a,\bar\eta}(y_{\bar\eta,0}) = b$ and so
$a_0 = a_1 = 0\}$ has at most one member.
\end{enumerate}
\mn
[Why?  Toward contradiction assume $\bar a_1 \ne \bar a_2 \in \cA_b$ and
let $n$ be minimal such that $a_{1,n} \ne a_{2,n}$, now
$n=0,1$ is impossible as $\bar a_1,\bar a_2 \in \cA_b$, so $n \ge 2$.
Now prove by induction on $\ell < n$ that 
$h_{\bar a_1,\bar\eta}(y_{\bar\eta,\ell}) = 
h_{\bar a_2,\eta}(y_{\bar\eta,\ell})$; for $\ell=0,1$ use $\bar a_1,\bar a_2
\in \cA_b$ and for $\ell = j+1$ recall $\ell y_{\eta,\ell} = y_{\eta,j}
- (\sum\limits_{m < \bold k} x_{\bar\eta \upharpoonleft(m,j)} +
a_{\iota,j} z)$ for $\iota = 1,2$; apply $h_{\bar a_\iota,\bar\eta}$ and use
the induction hypothesis.  Now on this equation for $\ell=n,\iota =
1,2$ apply $h_{\bar a_\iota,\bar\eta}$ and substracting then we get
$a_{1,n} - a_{2,n}$ is divisible by $\ell$ and $\ell \ge 2$ but $a_{1,n} -
a_{2,n} \in \{1,-1\}$, contradiction.]

So clearly there is $\bar a \in {}^\omega 2 \backslash \cup\{\cA_b:b \in
\bbZ\}$ such that $a_0 = 0 = a_1$, it is as required.  So $(*)_1$ 
holds indeed.]

Lastly,
\mn
\begin{enumerate}
\item[$(*)_2$]  let $G_1 = \langle G_0 \cup \{y_{\bar a[\bar\eta],
\bar\eta,n}:n < \omega\}:\bar\eta \in \Lambda_{\bold
x}\}\rangle_{G_1}$.
\end{enumerate}
\mn
Now $G_1$ witnesses that $G_2 = G_1/\bbZ z$ is not a Whitehead group.
[Why?  Let $G_2 = G_1/\bbZ z$ and let $h_*$ be the canonical
homomorphism from $G_1$ onto $G_1/\bbZ z$, i.e. $h_*(x) = x + \bbZ z$
for $x \in G_1$.  Toward contradiction assume $G_2$ is a Whitehead
group, this means that there is a homomorphism $g_*$ from $G_2$ into
$G_1$ inverting $h_*$, that is, $y \in G_2 \Rightarrow h_*(g_*(y)) = y$.

As $\Ker(g_*) = \bbZ z$ clearly $x \in G_1 \Rightarrow g_*(h_*(x))-x
\in \bbZ z$ so let $h_\ell$ be the unique function from
$\Lambda_{\bold x,< \bold k}$ into $\bbZ$ defined by $h_0(\bar\nu) =
b$ iff $(\bar\nu \in \Lambda_{\bold x,< \bold k},k \in \bbZ$ and
$g_*(h_*(x_{\bar\nu})) - x_{\bar\nu} = b z$.  By the choice of $\bold
b$ there is $\bar\eta \in \Lambda_{\bold x}$ such that $m < \bold k
\wedge n < \omega \Rightarrow \bold k(\bar\eta,m,n) =
h_\bullet(\bar\eta \upharpoonleft (m,n))$.  So $x \mapsto x -
g_*(h_*(x))$ define a homomorphism from $G_{\bar
  a(\bar\eta),\bar\eta}$ onto $\bbZ z$ mapping $z$ to itself and
mapping $x_{\bar\eta \upharpoonleft (m,n)}$ to $b_{\bold
  c_\eta(\bar\eta \upharpoonleft (m,n))} z$, contradicting to the
choice of $\bar a(\eta)$.  So $G_2 = G_1/\bbZ z$ is Whitehead indeed.

Now clearly
for some canonical$^* \, \bbZ$-construction $\gx,G_{\gx} = G^-_{\gx} \oplus
\bbZ z$, and easily $G_2 \cong G^-_{\gx}$ and $G_2$ is a direct
summand of $G_{\gx}$ so (by the well known group theory) 
also $G_{\gx}$ is not a Whitehead group.  The
cardinality and freeness demands are obvious.

\noindent
2) For transparency we ignore the ``Whitehead".  Recall we assume $\bold
x$ has the $\aleph_0$-black box not just the $\aleph_0$-pre-black box
(see \ref{a9}(1),(4)).

Let $\langle \Lambda_\alpha:\alpha < |\Lambda_{\bold
  x}|\rangle,\langle \bar\nu_\alpha:\alpha < \alpha_*\rangle$ be as in
Definition \ref{a9}(4).  Let $\langle h_{\bar\eta}:\bar\eta \in
  \Lambda_\alpha\rangle$ be an $\aleph_0$-BB.  We choose 
$(\bbZ,\bold x)$-construction $\gx$ by choosing $(z_{\bar\eta},
\bar a_{\bar\eta})$ for
  $\bar\eta \in \Lambda_\alpha$ by induction on $\alpha$ such that:
\mn
\begin{enumerate}
\item[$\bullet_1$]  $z_{\bar\eta} = z_0 = z$ if $\alpha=0$
  (alternatively, omit $z$)
\sn
\item[$\bullet_2$]  $z_{\bar\eta} = z_\alpha = x_{\bar\nu_\alpha}$ 
if $\bar\eta \in \Lambda_{1 + \alpha}$
\sn
\item[$\bullet_3$]  $\bar a_{\bar\eta}$ for $\bar\eta \in
  \Lambda_\alpha$ is chosen such that: there is no homomorphism $h$
  from $G_{\bar\eta}$ into $\bbZ$ such that $(h(x_{\bar\eta
    \upharpoonleft (m,i)}),h(z_\alpha))$ is coded by
  $h_{\bar\eta}(\bar\eta \upharpoonleft (m,i))$.
\end{enumerate}
\mn
So if $h \in \Hom(G_{\gx},\bbZ)$ then $\alpha < \alpha_* \Rightarrow
h(z_\alpha) = 0$ but $\oplus \sum\limits_{\alpha} \bbZ z_\alpha = G_0$ so $h
\rest G_0$ is zero but $G_{\gx}/G_0$ is divisible hence $h$ is zero.

Alternatively omitting ``$G = G_{\gx}$", this follows easily by 
repeated amalgamation of the $G$ constructed in part (1) over pure 
subgroups isomorphic to $\bbZ$, see the proof of \ref{d24}(3) or, 
e.g. \cite[\S3]{Sh:1005}.
\end{PROOF}
\bigskip

\centerline {$* \qquad * \qquad *$}
\bigskip

Now claim \ref{d21} as stated is enough when we use \S1 to get
$\aleph_{\omega \cdot n}$-free $\bold x$ with $\chi$-BB, 
(see \ref{a52}(1),(2)) but
not for $\aleph_{\omega_1 \cdot n}$-free, because there we need
for $\partial = \aleph_1,J = J^{\bd}_\kappa \times J^{\bd}_\sigma,\sigma <
\kappa$ regular in particular $(\sigma,\kappa) = (\aleph_0,\aleph_1)$.
So we better use the
construction from Definition \ref{d4} rather than \ref{d2}.  Also we
prefer to have general $R$-modules and we formalize the relevant
property of $R,\bar\partial,\bar J,\theta$.  We use ${}_R R$ to denote
$R$ as a left $R$-module.

\begin{definition}
\label{d23}
1) We say that $(\bar\partial,\bar J)$ does $\theta$-fit $R$ \underline{or}
the triple $(\bar\partial,\bar J,\theta)$-fit $R$ (but if $\bar\partial =
\bar\partial_{\bold x},\bar J = \bar J_{\bold x}$ then we may write
$\bold x$ instead of $(\bar\partial,\bar J)$) \when \,:
\mn
\begin{enumerate}
\item[(A)]  
\begin{enumerate}
\item[(a)]  $R$ is a ring
\sn
\item[(b)]  $\bold k$ is a natural number $\ge 1$
\sn
\item[(c)]  $\bar\partial = \langle \partial_\ell:\ell <  \bold k\rangle$
\sn
\item[(d)]  $\partial_\ell$ is a regular cardinal
\sn
\item[(e)]  $\bar J = \langle J_\ell:\ell < \bold k\rangle$
\sn
\item[(f)]  $J_\ell$ is an ideal on $\partial_\ell$.
\end{enumerate}
\sn
\item[(B)]  If $G_0 = \oplus\{Rx_{m,i}:m < \bold k,i < \partial_m\} 
\oplus Rz$ and $h \in \Hom(G_0,{}_R R)$ and $h(z) \ne 0$ \then \,
there is $G_1$ such that
\sn
\begin{enumerate}
\item[$(*)$]
\begin{enumerate}
\item[$(\alpha)$]  $G_1$ is an $R$-module extending $G_0$
\sn
\item[$(\beta)$]  $G_1$ has cardinality $< \theta$
\sn
\item[$(\gamma)$]   there is no homomorphism from $G_1$ to
${}_R R$ (i.e. $R$ as a left $R$-module) extending $h$.
\end{enumerate}
\end{enumerate}
\end{enumerate}
\mn
1A) We replace ``fit" by ``weakly fit" when in clause (B) we further
demand on $h,h(x_{m,2i}) = h(x_{m,2i+1})$.

\noindent
2) We say $(\bar\partial,\bar J)$ freely $\theta$-fits $R$ or
 $(\bar\partial,\bar J,\theta)$-fit $R$ (but if $\bar\partial =
 \bar\partial_{\bold x},\bar J = \bar J_{\bold x}$ we may replace
 $(\bar\partial,\bar J)$ by $\bold x$) \when \,:
\mn
\begin{enumerate}
\item[$(A)$]  $(a)-(f)$ as above
\sn
\item[$(B)$]   as above adding
\sn
\begin{enumerate}
\item[${{}}$]  $(\delta) \quad$ if $m_* < \bold k \wedge w \in J_{m_*}$ 
\then \, $G_1$ is free over 

\hskip25pt  $\oplus\{Rx_{m,i}:m < \bold k,
i < \partial_m$ and $m=m_* \Rightarrow i \in w\} \oplus Rz$.
\end{enumerate}
\end{enumerate}
\mn
3) In part (1) above and also parts (4)-(6) below we may 
write $(\partial,J,\bold k)$ instead of 
$(\langle \partial_\ell:\ell < \bold k\rangle,\langle
J_\ell:\ell < \bold k\rangle)$ when
$\ell < \bold k \Rightarrow \partial_\ell = \partial \wedge J_\ell =J$.  
Also we may write $J$ if $m < \bold k \Rightarrow J_m = J$ and
omit $\bar J$ when $\ell < \bold k \Rightarrow J_\ell =
J^{\bd}_{\partial_\ell}$.

\noindent
4) We may above replace ``$J_\ell$ is an ideal on $\partial_\ell$" by
$J_\ell \subseteq \cP(\partial_\ell)$.

\noindent
5) We may omit $\theta$ when $\theta = |R|^+ + \max\{\partial^+_m:
m < \bold k\}$.

\noindent
6) We replace fit by $\bbI$-fit when:
\mn
\begin{enumerate}
\item[$(a)$]  $\bbI$ is a set of ideals of $R$ closed under
  intersection of two including $I_0 = \{0_R\}$
\sn
\item[$(b)$]  replace $R z$ by $(R/I)z,I \in \bbI$.  The default value
of $\bbI$ is $\{\{a:ab=0\}:b \in R\}$
\sn
\item[$(c)$]  in $(B)(*)$, if $x \in G_1 \backslash \{0\}$ then
  $\ann(x,G_1) = \{a \in R:ax=0\} \in \bbI$.
\end{enumerate}
\end{definition}

\begin{claim}
\label{d24}
1) Assume $\bold x$ is a $\bold k$-c.p., $R$ is a ring, $\bold x$ does
   $\theta$-fit $R,\chi^+ \ge \theta + |R|^+$ and $\bold x$
 has $(\chi,\bold k,1)$-BB.

There is $\gx$ such that:
\mn
\begin{enumerate}
\item[$(a)$]  $\gx$ is an $(R,\bold x)$-construction 
\sn
\item[$(b)$]  $G = G_{\gx}$ is an $R$-module of cardinality
  $|\Lambda_{\bold x}|$
\sn
\item[$(c)$]  there is no $h \in \Hom(G,{}_R R)$ such that $h(z) \ne 0$
\sn
\item[$(d)$]  $\gx$ is simple, that is, 
$z_{\bar\eta} = z$ for $\bar\eta \in \Lambda_{\bold x}$.
\end{enumerate}
\mn
2) If in addition $\bold x$ freely $\theta$-fit $R$, \then \, we can add:
\mn
\begin{enumerate}
\item[$(e)$]  $G$ is $\sigma$-free if $\bold x$ is $\sigma$-free
  (holds always for $\sigma = \min(\bar\partial_{\bold x}))$
\sn
\item[$(f)$]  $G$ is $(\theta_2,\theta_1)-1$-free if $\bold x$ is
  $(\theta_2,\theta_1)$-free.
\end{enumerate}
\mn
3) In (2) we can add:
\mn
\begin{enumerate}
\item[$(g)$]  $\Hom(G,{}_R R) = 0$.
\end{enumerate}
\mn
4) We can use above ``weakly fit".
\end{claim}

\begin{PROOF}{\ref{d24}}
Let $G_* = \oplus\{R x_{\bar\eta}:\bar\eta \in \Lambda_{\bold x,<
  \bold k}\} \oplus R z$.

\noindent
1) Let $\{(a^1_\varepsilon,a^2_\varepsilon):\varepsilon < \chi\}$
   list, possibly with repetitions the members of $R \times 
(R \backslash \{0_R\})$ and let 
$\bold b$ be a $(\chi,\bold k,1)$-BB for $\bold x$ 
and let $\bold b',\bold b''$ 
be defined such that: $\varepsilon = \bold b_{\bar\eta}(m,i)$
implies $\bold b'_{\bar\eta}(m,i) = a^1_\varepsilon,\bold
   b''_{\bar\eta}(m,i) := a^2_\varepsilon$.

For $\bar\eta \in \Lambda_{\bold x}$ let $G^0_{\bar\eta} = \Sigma\{R
x_{\bar\eta \upharpoonleft (m,i)}:m < \bold k,i < \partial_m\}
\oplus R z \subseteq G_*$ and let $h_{\bar\eta}$ be the 
unique homomorphism from $G^0_{\bar\eta}$
  into ${}_R R$ satisfying $h_{\bar\eta}(x_{\bar\eta \upharpoonleft
    (m,i)}) = \bold b'_{\bar\eta}(m,i)$ and $h_{\bar\eta}(z) =
  \bold b''_{\bar\eta}(0,0)$ and let $G^1_{\bar\eta}$ be an $R$-module
  extending $G^0_{\bar\eta}$ such that
  $(G^1_{\bar\eta},G^0_{\bar\eta},h_{\bar\eta})$ here are like
  $(G_1,G_0,h)$ in Definition \ref{d23}(1)(B)$(*)$, so in particular
there is no homomorphism from
  $G^1_{\bar\eta}$ into ${}_R R$ extending $h_{\bar\eta}$.  \Wilog
  \, $G^1_{\bar\eta} \cap G_0 = G^0_{\bar\eta}$ and $\langle
  G^1_{\bar\eta} \backslash G^0_{\bar\eta}:\bar\eta \in \Lambda_{\bold
    x}\rangle$ is a sequence of pairwise disjoint sets.  Let $G$ be
  the $R$-module generated by $\cup\{G^1_{\bar\eta}:\bar\eta \in
  \Lambda_{\bold x}\} \cup G_0$ extending each $G^1_{\bar\eta}$ and
$G_*$, freely except this.  Clearly we have defined an $R$-construction
  $\gx$ with $\bold x_{\gx} = \bold x,G_{\gx} = G,z_{\gx,\bar\eta} =
  \{z\}$ and clauses $(a),(b),(d)$ of the desired conclusion hold.
  To prove clause (c) toward contradiction assume that $h \in
  \Hom(G,{}_R R)$ satisfies $h(z) \ne 0$.  Let $g:\Lambda_{\bold x,< \bold
    k} \rightarrow \chi$ be defined by $g(\bar\nu) = \min\{\varepsilon
< \chi:(h(x_{\bar\nu}),h(z)) = (a^1_\varepsilon,a^2_\varepsilon)\}$.

Clearly the function is well defined hence as $\bold x$ has
$(\chi,\bold k,1)$-BB, that is by the choice of $\bold b$ there is
$\bar\eta \in \Lambda_{\bold x}$ such that $m < \bold k \wedge i <
\partial_m \Rightarrow g(\bar\eta \upharpoonleft (m,i)) = \bold
b_{\bar\eta}(m,i)$.  We get easy contradiction.

What about the cardinality $|G|$?  
Note that $|G^1_{\bar\eta}| < \theta$ and $\theta \le \chi^+$.

\noindent
2) In the proof of part (1), choosing $G^1_{\bar\eta}$ we add the parallel
of clause $(*)(\delta)$ of \ref{d23}(B).  Now clause (e) of \ref{d24}(2) holds 
by \ref{d6}(1) and clause (f) by \ref{d18}(2).

\noindent
3) Let $G$ be as constructed in part (1), and let $Y = \{y \in G:G/Ry$
   is $\aleph_1$-free or even $\min(\bar\partial)^+$-free$\}$ (recall
\ref{d6} + freeness of $\bold x$).

So by part (2) the set $Y$ generates $G$, let 
$\langle G_\varrho,h_\rho:\varrho \in
{}^{\omega >}Y \rangle$ be such that $G_\varrho$ is an $R$-module,
$h_\varrho$ is an isomorphism from $G$ onto $G_\varrho$, \wilog \,
$0_{G_\varrho}=0$ for every $\varrho$ and 
$G_{\varrho_1} \cap G_{\varrho_2} = \{0\}$ for $\varrho_1 \ne
\varrho_2$.

Let $H_1 = \oplus\{G_\varrho:\varrho \in {}^{\omega >}Y\}$ and let $H_0$ be the
$R$-submodule of $H_1$ generated by

\[
X = \{h_{\varrho \char 94 \langle y \rangle}(z) - h_\varrho(y):\varrho
\in {}^{\omega >}Y \text{ and } y \in Y\}.
\]

\mn
That is, toward contradiction $f_g \in \Hom(H,{}_R R)$ is not zero and
$f_1 \in \Hom(H_1,{}_R R)$ be defined by
$f_1(x) = h(x+H_0)$, so also $f_1$ is not zero but $x \in X
\Rightarrow f_1(x) = 0$.  By the choice of $H_1$, there is $\varrho
\in {}^{\omega >} Y$ such that $f_1 \rest G_\varrho$ is not zero.  But
recall that $G$ is generic by $Y$ hence $G_\varrho$ is generated by
$\{f,h_\varrho(y):y in Y\}$, hence for some $n \ge 1$ and
$y_0,\dotsc,y_{n-1} \in Y$ and $b_0,\dotsc,b_{n-1} \in R \backslash
\{0_R\}$ we have $f_1(h_\rho(\sum\limits_{\ell < n} b_\ell,y_\ell))
\in R \backslash \{0\}$ hence for some $\ell < n,0 \ne
f_1(h_\varrho(b_\ell y_\ell)) = f_1(b_\ell h_\varrho(y_\ell))$.  So
letting $y = h_\varrho(y_\ell)$ we have $y \in G_\varrho$ and for some
$\ell \in R \backslash \{0\},c = f_1(b_\ell h_\varrho(y_\ell)) =
f_2(b_\ell y))$.

As said above about $f_1$ we have $f_1(y) = f_1(h_\varrho(y_\ell)) =
f_1(h_{\varrho \char 94 \langle y_\ell\rangle}(z))$ so 
$f_1(h{\varrho \char 94 \langle y_\ell\rangle}(z)) = b \in R
\backslash \{0\}$.  So $h_{\varrho \char 94 \langle y_\ell\rangle}
\circ f_1 \in \Hom(G,{}_R R)$ maps $z$ into $h \in R \backslash
\{0\}$, contradiction.

\noindent
4) Similarly but repalcing $x_{\bar\eta}(\bar\eta \in \Lambda_{\bold
  x,< \bold k})$ by $x_{\bar\eta,\zeta}(\zeta < |R|^+)$.

\noindent
5) Let $\langle (\bar\alpha_\varp,a_\varp):\varp < \chi\rangle$ list,
possibly with repetitions the members of
$\{(\bar\alpha,a_\varp):\bar\alpha = (\alpha_0,\alpha_1)$ such that
$\alpha_0 < \alpha_1 < \chi$ and $a \in R \backslash \{0_R\}\}$,
possibly with repetitions, and $\bold b$ be a $(\chi,\bold k,1)$-BB
for $\bold x$ and let $\bold b^\iota$ for $\iota =0,1,2$ be the
functions with the same domain as $\bold b$ (writing $\bold
b^\iota_{\bar\eta}(m,i)$ or $\bold b_{\eta,\iota}(m,i)$ for $\bold
b^\iota(\bar\eta,m,i))$ such that $\varp = \bold b_{\bar\eta}(m,i)$
implies $9\alpha_{\varp,\ell}:\ell < 3\rangle = (\bold
b^0_{\bar\eta}(m,i),\bold b^1_{\bar\eta}(m,i),\bold
b^2_{\bar\eta}(m,i))$.

Let $G_0 = \oplus\{R x_{\bar\eta,\varp}:\bar\eta \in \Lambda_{\bold
  x,< \bold k}$ and $\varp < \chi\}$ and
\mn
\begin{enumerate}
\item[$(*)_1$]  for $\bar\eta \in \Lambda_{\bold x}$ let
\sn
\begin{enumerate}
\item[(a)]  $G^0_{\bar\eta} = \Sigma\{R x_{\bar\eta \upharpoonleft
    (m,i),\varp}: m < \bold k,i < \partial_m$ and $\varp < \chi\}$
\sn
\item[(b)]  $G^{0,0}_{\bar\eta} = \Sigma\{R x_{\bar\eta \upharpoonleft
(m,i),\bold b_{\bar\eta,1}(m,i)} - x_{\bar\eta \upharpoonleft
(m,i),\bold b_{\bar\eta,0}(m,i)}): m < \bold k,i < \partial_m\} \oplus
R z$
\sn
\item[(c)]  $G^{0,1}_{\bar\eta} = G^{0,0} \oplus R z$
\sn
\item[(d)]  $h_{\bar\eta}$ be the homomorphism from
  $G^{0,0}_{\bar\eta}$ into $R$ such that:
\begin{itemize}
\item  $h_{\bar\eta} \rest G^{0,0}$ is constantly zero
\sn
\item  $h_{\bar\eta}(z)$ is $\bold b_{\bar\eta,2}(0,0) \in R
  \backslash \{0\}$
\end{itemize}
\sn
\item[(e)]  let $\bold h_{\bar\eta}$ be the isomorphism from $G_0 = 
\oplus\{R x_{\bar\eta \upharpoonleft (m,i)}:m < \bold k,i
< \partial_m\}$ onto $G^{0,1}_{\bar\eta}$ such that
$\bold h_{\bar\eta}(z)=z,\bold h_{\bar\eta}(x_{m,i}) = (x_{\bar\eta
  \upharpoonleft (m,i),\bold b_{\bar\eta,1}(m,i)} - x_{\bar\eta
  \upharpoonleft (m,i),\bold b_{\bar\eta,0}(m,i)}$
\sn
\item[(f)]  $G^\bullet_{\bar\eta,1}$ be an $R$-module extending the
  $R$-module $G^\bullet_{\bar\eta}$ such that the triple
$G^\bullet_0,G_{\bar\eta,1},h^\bullet_{\bar\eta} \circ \bold
h_{\bar\eta})$ is as in \ref{d23}(1)(B)$(*)$
\sn
\item[(g)]  let $\bold h ^+_{\bar\eta},G^1_{\bar\eta}$ are such that
  $G^1_{\bar\eta}$ is an $R$-module extending $G^0_{\bar\eta}$ and
  $h^+_{\bar\eta}$ is an isomorphism from $G^\bullet_{\bar\eta,1}$
  onto $G^1_{\bar\eta}$ extending $\bold h_{\bar\eta}$.
\end{enumerate}
\end{enumerate}
\mn
Lastly, let
\mn
\begin{enumerate}
\item[$(*)$]  \wilog \, $G^1_{\bar\eta} \cap G_0 =
G^{0,0}_{\bar\eta},\langle G^1_{\bar\eta} \backslash
G^{0,0}_{\bar\eta}:\bar\eta \in \Lambda_{\bold x}$ are pairwise
disjoint and $G^*_1$ is an $R$-module extending $G_0$ and
$G^1_{\bar\eta}$ for $\bar\eta \in \Lambda_{\bold x}$ and generated by
their union freely (except the equations implicit in ``extending"
above).
\end{enumerate}
\mn
Note
\mn
\begin{enumerate}
\item[$(*)$]  if $h \in \Hom(G,{}_R R)$ satisfies $h(z) \ne 0_R$ then
  we define a function $\bold c:\Lambda_{\bold x,< \bold k}
  \rightarrow \chi$ as follows: $\bold c$ is the minimal $\varp <
  \chi$ such that:
\sn
\begin{itemize}
\item  $h(x_{\bar\eta,\alpha_{\varp,0}}) =
  h(x_{\bar\eta,\alpha_{\varp,1}})$
\sn
\item  $h(z) = a_\varp$.
\end{itemize}
\end{enumerate}
\mn
The rest should be clear.
\end{PROOF}

\begin{remark}
We can use a $2^\chi$-BB $\bold b$ and then let $\bold c(\bar\eta)$
code $(h \rest \{x_{\bar\eta,\varp}:\varp < \chi\},h(z))$.

Lastly, $H := H_1/H_0$ is as required.
\end{remark}

\begin{remark}
\label{d25}
1) There is an 
alternative to the proof of \ref{d24}(3), assume that $\bold x$ has
$\aleph_0$-well orderable $(\chi,\bold k,1)$-BB $\bar\alpha$ as
witnessed by $\bar\Lambda$, see Definition \ref{a61}.  We then can
find a $(R,\bold x)$-construction obeying $\bar\Lambda$, see
\ref{d4}(1B).

\noindent
2) It may suffice for us to prove in \ref{d24} that $\gx$ is simple
and $Rz$ is not a direct summand of the $R$-module $G_{\gx}$.  For
this we can weaken the demand in Definition \ref{d23}(1)(B) demanding
$h(z) = 1_R$.
\end{remark}

\begin{claim}
\label{d29}
1) Let $\partial = \aleph_0,J = J^{\bd}_\partial$ and $\bold k=1$,
\then \, $(\partial,J,\bold k_\theta)$ freely fit $R$ \when \,:
\mn
\begin{enumerate}
\item[$\oplus_1$]
\begin{enumerate}
\item[(a)]    $R$ is an infinite ring
\sn
\item[(b)]  if $d \in R \backslash \{0\}$ and $\bar d \in {}^\omega
  R$, \then \, we can find $a^\iota_n \in R$ for $\iota = 1,2,3$ and
  $n < \omega$ such that the following set $\Gamma$ of equations
  cannot be solved in $R$:
\[
\Gamma = \{a_n x_{n+2} = x_n + d_n + b_n d:n < \omega\}.
\]
\end{enumerate}
\end{enumerate}
\mn
2) For $\partial,J,\bold k$ as above, $(\partial,J,\bold k,\theta)$
freely weakly fit $R$ \when \,:
\mn
\begin{enumerate}
\item[$\oplus_2$]
\begin{enumerate}
\item[(a)]    as above
\sn
\item[(b)]  for every $d \in R \backslash \{0\}$ letting
  $\bigwedge\limits_{n} d_n = 0_R$, the demand in $\oplus$ above
  holds, i.e. there are $a_n,b_n \in R$ for $n < \omega$ such that the
  following set  $\Gamma$ of equations is not solved in $R$:
\begin{itemize}
\item  $\Gamma = \{a_n x_{n+1} = x_n + b_n d:n < \omega\}$.
\end{itemize}
\end{enumerate}
\end{enumerate}
\mn
3) In part (2), $\oplus^*_1$ holds when:
\mn
\begin{enumerate}
\item[$\oplus_3$]
\begin{enumerate}
\item[(a)]    as above
\sn
\item[(b)]   $(R,+),R$ as an additive group is $\aleph_1$-free.
\sn
\item[(c)]  $(R,+)$ is $\aleph_1$-free 
or at least $\cap\{nR:n \ge 2\} = \{0\}$.
\end{enumerate}
\end{enumerate}
\end{claim}

\begin{PROOF}{\ref{d29}}
We should check all the clauses in Definition \ref{d23}.  First,
Clause (A) is obvious: $R$ is a ring by $\oplus(a), \bold k=1 >0$ by a
\ref{d29} assumption, of course, letting $\bar\partial =
\langle \partial_0 \rangle,\partial_0 = \partial,\partial$ is regular
being $\aleph_0$ and $\bar J = \langle J_0 \rangle,J_0 = J$ is
$J^{\bd}_\partial = J^{\bd}_{\aleph_0}$ so an ideal on $\partial$.

Second, toward proving Clause (B), assume $G_0 = \oplus\{R x_{m,i}:m <
\bold k=1$ so $m=0$ and $i < \partial\} \oplus \bbZ z,h_0 \in
\Hom(G_0,{}_R R)$ and $d := h_0(z) \ne 0_R$ and let $d-n =
h(x_{0,n})$.  We should find $G_1$ satisfying $(*)$ there.  Let
$\langle (a^\bullet_n,b_n):n < k\rangle$ be as guaranteed by
$\oplus(b)$ of the claim for $d,\langle d_n:n < \omega\rangle$ from
above.

For each $i < \partial$ let $G^*_n = G_0 \oplus(R y_n)$ and
$R$-module; clearly there is an embedding $g_n:G^*_n \rightarrow
G^*_{n+1}$ such that $g_n \rest G_0 = \id_{G_0}$ and $g_n(y_n) =
a^\bullet_n y_{n+1} + x_{0,n} + b_n z$ where the $a_n,b_n \in R$ are
from $\oplus(b)$ for our $h$.

Renaming \wilog \, $G^*_n \subseteq G^*_{n+1}$ and $g_n$ is the
identity on $G^*_n$.  Lastly, let $G_1 = \bigcup\limits_{n} G^*_n$ and
it suffices to prove that $(*)$ of Definition \ref{d23} is satisfied.
Clearly $G_1$ is an $R$-module extending $G_0$, i.e. $(*)(\alpha)$
holds.  Also $|G_1| \le \aleph_0 + |G_0| = \aleph_0 + \aleph_0 \cdot
|R| = |R| < |R|^+ = \theta$, recalling $R$ is an infinite ring, so
also $(*)(\beta)$ holds.

Lastly, to prove $(*)(\gamma)$, toward contradiction assume $h_2 \in
\Hom(G_1,{}_R R)$ extends $h$.  Let $c_n := h_2(y_n) \in R$, now
\mn
\begin{enumerate}
\item[$(*)$]
\begin{enumerate}
\item[(a)]  $\bar c = \langle c_n:n < \omega\rangle \in {}^\omega
  \bbR$
\sn
\item[(b)]  $a_n c_{n+1} = a_n h_1(y_{n+1}) = h_2(a^1_\eta y_{n+1}) =
  h_2(y_n + x_{0,n} + b_n z) = h_2(y_n) + h_1(x_{0,n}) + b_n h_2(z) =
  c_n + d_n + b_n d$.
\end{enumerate}
\end{enumerate}
\mn
So $\bar c$ solves (in $R$) the set of equations $\Gamma = \{a_n
z_{n+1} = z_n + d_n + b_n d:n < \omega\}$, contradicting the choice of
$\langle (a_n,h_n):n < \omega\rangle$.

We still have to justify the ``freely", i.e. clause $(\delta)$ of
\ref{d23}(2).  So let $m_* < \bold k$, i.e. $m_*=0$ and $w \in J_0 =
J^{\bd}_\partial$ so $w$ is finite and let $G_0 = \oplus\{R x_{0,i}:i
\in w\}$, let $n_*$ be such that $\sup(w) < n_*$ and we easily finish
by noting:
\mn
\begin{enumerate}
\item[$(*)$]  the sequence $\langle y_n:n > n_*\rangle \char 94
  \langle x_{0,m}:m \le n^*\rangle \char 94 \langle z \rangle$
  generate $G_1$.
\end{enumerate}
\mn
[Why?  Freely, it generates $G_1$ because $x_{0,m} = a_n y_{m+2} -b_m
y_m$ for $m > n_*$, use $y_n = a_n y_{n+1} - x_{0,n} - b_n z$ by
downward induction on $n \le n_*$; translating the equations they
become trivial.]

\noindent
2) Similarly but we choose $g_n$ such that $g_n(y_n) = a_n y_{n+1} +
(x_{0,2n} - x_{0,2n+1}) + b_n z_n$.

\noindent
3) Choose $b_n = 1_R,a_n:n! \cdot 1_R$.
\end{PROOF}

\begin{claim}
\label{d31}
1) The triple $(\partial,J,\bold k,\theta)$ freely fit $\bbZ$ \when \,:
\mn
\begin{enumerate}
\item[$(a)$]  $\theta = \aleph_2,\partial = \aleph_2$ and $\bold k > 0$
\sn
\item[$(b)$]  $J = J^{\bd}_{\aleph_1} \times J^{\bd}_{\aleph_0}$, but
  pedantically use the isomorphic copy $J_{\aleph_1 * \aleph_0} 
= \{A$: for some $n_\alpha < \omega$ for $\alpha < \omega_1,i_* <
\omega_1$ we have $A \subseteq \{\omega \cdot i+n:i < i_*$ or $n < 
n_\alpha\}\}$; hence it is also O.K. to use 
$J = J^{\bd}_{\aleph_0} * J^{\bd}_{\aleph_1}$.
\end{enumerate}
\mn
2) The triple $(\aleph_1,J,\bold k,\theta)$ freely fits $R$ \when \,:
\mn
\begin{enumerate}
\item[$(a),(b)$]  as above 
\sn
\item[(c)]  $\theta = \aleph_2$
\sn
\item[$(d)$]  given $b_{\alpha,n} \in R$ for $\alpha < \omega_1,n <
  \omega$ and $t \in R \backslash \{0_R\}$, there are pairwise distinct
  $\rho_\alpha \in {}^\omega 2$ for $\alpha < \omega_1$ and
  $a_{\alpha,n},d_{\alpha,n} \in R$ such that the following set of
  equations is not solvable in $R$:
\sn
\begin{itemize}
\item[$\bullet$]  $d_{\alpha,n+1} y^1_{\alpha,n+1} = y^1_{\alpha,n} -
  y^2_{\rho_\alpha \rest n} - b_{\alpha,n} - a_{\alpha,n} t$.
\end{itemize}
\end{enumerate}
\mn
3) Similarly for ``weakly" fit
\end{claim}

\begin{remark}
%\label{}
1) Probably we can use $\bar\partial = \langle \partial_\ell:\ell < \bold
   k\rangle$ with $\partial_\ell \in \{\aleph_0,\aleph_1\}$ but no
   real need so far.

\noindent
2) This is essentially \cite[\S4]{Sh:98} and
\cite[4.10(C)=L5e.28]{Sh:898}.  
\end{remark}

\begin{PROOF}{\ref{d31}}
1) Proving clause (A) of \ref{d23}(1) and
clause (B)$(\delta)$ of \ref{d23}(c) is easy as in \ref{d29}, so we
concentrate on \ref{d23}(B).

So let $G_0,h$ be as in \ref{d23}(1)(B).  Choose $p_n$ by induction
on $n$ as follows: $p_0 = 2,p_{n+1}$ the first prime $> p_n+n$ such
that $p_{n+1}!/(c_{n+1}-n) > \sqrt{p_{n+1!}}$; where we let $c_n =
\prod\limits_{m < n}(p_m!)$.

Now observe that:
\mn
\begin{enumerate}
\item[$\boxplus$]  for $n \ge 100$ there is $C_n \subseteq
\{0,1,\dotsc,(p_n!)-1\}$ such that: if $b \in \bbZ$ and $t \in \bbZ$
satisfies $0 < |t| < n$ then for some $a_0,a_1 \in \bbZ$ we have
\sn
\begin{enumerate}
\item[$\bullet$]   $b + c_n a_0 t \in \cup\{i+(p_{n+1}!-1)\bbZ:i \in C_n\}$
\sn
\item[$\bullet$]  $b + c_n a_1 t \notin \cup\{i + (p_{n+1}!)\bbZ:i \in C_n\}$.
\end{enumerate}
\end{enumerate}
\mn
[Why?  It suffices to consider $b \in \{0,\dotsc,p_n!-1\},t
\in\{\ell,-\ell:\ell \le n,\ell \ne 0\}$ and let $A_{b,t} = \{b + c_n a t:a
\in \bbZ\} \cap \{0,\dotsc,p_{n+1}!-1\}$.
Clearly $|A_{b,t}| = (p_n!)/(c_n \cdot |t|) > \sqrt{p_n!}$.
The family $\{A_{b,t}:b \in \{0,\dotsc,p_{n+1}!-1\},
t \in \{\ell,-\ell:\ell \le n,\ell \ne =0\}\}$ has at most $2 n(p_n!)$
members.   Easily the number of $C \subseteq
\{0,\dotsc,p_n!-1\}$ such that $(C \supseteq A_{b,t}) \vee (C \cap
A_{b,t}) = \emptyset)$ for some pair $(b,t)$ as above is\footnote{In
 other words, for each $b,t a$ above a random 
$C \subseteq \{0,\dotsc,p_{n+1}!-1\}$ has probability 
$\le \frac{2}{2^{|A_{b,t}|}} \le \frac{2}{2^{\sqrt p_n!}}$ to include
  $A_{b,t}$ or to be disjoint to it.  So the probability that this
  occurs for some pair $(b,t)$ in $\le 2 \cdot |\{A_{b,t}:b,t$ is as
  above$\}|/2^{\sqrt p_n!} \le 4 n(p_n!)/2^{\sqrt{p_n!}}$ which is $\ll 1$.}
$< 2^{\sqrt p_{n+1}!}$ hence there is $C_n$ as required.]

Let $\Omega \subseteq {}^\omega 2$ be of cardinality $\aleph_1$ and
$\langle \rho_\alpha:\alpha < \omega_1\rangle$ list $\Omega$ without
repetitions.

Let $G$ be generated by $\{x_{m,\alpha}:\alpha < \aleph_1,m < \bold k\} \cup
\{y^1_{\rho,n}:\rho \in \Omega$ and $n < \omega\} \cup \{y^2_\varrho:
\varrho \in {}^{\omega >}2\} \cup \{z\}$ freely except the equations:
\mn
\begin{enumerate}
\item[$(*)^1_{\alpha,n}$]  $p_n! y^1_{\alpha,n+1} = y^1_{\alpha,n} -
  y^2_{\rho_\alpha \rest n} - \sum\limits_{m < \bold k}
x_{m,\omega \cdot \alpha +n} - a_{\alpha,n} z$
\end{enumerate}
\mn
where $a_{\alpha,n} \in \bbZ$ are chosen below; let $\bar a = \langle
a_{\alpha,n}:\alpha < \omega_1,n < \omega\rangle$, so really $G =
G_{\bar a}$ and let $\bar a_{\alpha,<n} = \langle a_{\alpha,n_1}:n_1 <
n\rangle$.

Note that in $G$
\mn
\begin{enumerate}
\item[$(*)^2_{\alpha,n}$]  $y^1_{\alpha,0} = c_n y^1_{\alpha,n} +
\sum\limits_{n_1 < n} \, c_{n_1}(y^2_{\rho_\alpha \rest n_1} + 
\sum\limits_{m < \bold k} x_{m,\omega \cdot \alpha +n} + a_{\alpha,n_1} z)$.
\end{enumerate}
\mn
Define
\mn
\begin{enumerate}
\item[$(*)^3_{\alpha,n}$]  $b_{\alpha,n} = 
\sum\limits_{n_1 \le n} h(\sum\limits_{m < \bold k} c_{n_1} 
\,x_{m,\omega \cdot \alpha + n}) \in \bbZ$.
\end{enumerate}
\mn
Recall $G_0,h$ are as in \ref{d23}(1)(B).
Let $n_* = |h(z)|$ so $n_* > 0$.  We choose $a_{\alpha,n} \in \bbZ$ by
induction on $n$ such that: if $n > |h(z)|$ then
\mn
\begin{enumerate}
\item[$(*)^5_{\alpha,n}$]  $\rho_\alpha(n)=1$ \Iff \, $(b_{\alpha,n}  + 
\sum\limits_{n_1 \le n} c_{n_1} 
a_{\alpha,n_1} h(z))$ is equal to some $a \in C_n$ modulo $< p_n!$.
\end{enumerate}
\mn
[Why possible?  Arriving at $n$, the sum on the right side is $b_{\alpha,n} +
\sum\limits_{n_1 < n} c_{n_1} a_{\alpha,n_1} h(z)) + 
c_n a_{\alpha,n} h(z) \in \bbZ$, with the first two summands being
already determined, i.e. are computable from $\bar a_{\alpha,<n}$ and 
$|h(z)| \le n$, applying $\boxplus$ with $(n,h(z),b_{\alpha,n} +
\Sigma\{c_{n_1} a_{\alpha,n_1} h(z):n_1 < n\})$ here standing for
$(n,t,b)$ there we get there $a_0,a_1$ and let
  $a_{\alpha,n}$ be $a_0$ if $\rho_\alpha(n)=0$ and $a_1$ if
$\rho_\alpha(n)=1$. 
So for every $n,a_{\alpha,n}$ is as required and can be chosen.]

Having chosen $\bar a = \langle a_{\alpha,m}:\alpha < \omega_1,m <
\omega\rangle$, the Abelian group $G = G_{\bar a}$ is chosen.
Hence we just have to prove that $G$ is as required in clause (B) of
\ref{d23}(1),(2).  First, for \ref{d23}(1)(B)
\mn
\begin{enumerate}
\item[$\odot$]   toward contradiction assume
that $f \in \Hom(G,\bbZ)$ extend $h$ and $n_* = |f(z)|$ is $>0$
\end{enumerate}
\mn
 hence (for every $\alpha,n$ applying $f_n$ to the equation in
 $(*)^2_{\alpha,n}$):
\mn
\begin{enumerate}
\item[$(*)^6_{\alpha,n}$]  $f(y^1_{\alpha,0}) = c_n! f(y^1_{\alpha,n}) +
 \sum\limits_{n_1 < n} c_{n_1} \, f(y^2_{\rho_\alpha \rest n_1}) +
\sum\limits_{n_1<n} \, \sum\limits_{m < \bold k} c_{n_1} \,
f(x_{m,\omega \cdot \alpha +n_1}) + \sum\limits_{n_1<n}
  c_{n_1} \, a_{\alpha,n} f(z)$.
\end{enumerate}
\mn
So recalling $|h(z)| = n_*$ for some $\rho_* \in {}^{n_* +100}2$ and 
$a_* \in \bbZ$ we have $|S| =
\aleph_1$ where $S = \{\alpha < \aleph_1:f(y^1_{\alpha,0}) \equiv a_*$
and $\rho_\alpha \rest (n_* +1)) = \rho_*\}$.

So choose $\alpha < \beta$ from $S$ and let $n =
\min\{\ell:\rho_\alpha(\ell) \ne  \rho_\beta(\ell)\}$, clearly 
we have $n > n_*$ hence $n \ge n_* +1 \ge 2$ and substructing
the equations $(*)^6_{\alpha,n+1}, (*)^6_{\beta,n+1}$, in the left
side we get a multiple of $c_{n+1}$ so a
number divisible by $p_n!$ and in the right side we
get the sum of the following four differences:
\mn
\begin{enumerate}
\item[$\odot_1$]  $f(y^1_{\alpha,0})- (f(y^1_{\beta,0})$ which is zero
 by the choice of $S$ and the demand $\alpha,\beta \in S$
\sn
\item[$\odot_2$]  $\sum\limits_{n_1 \le n} c_{n_1} \,
f(y^2_{\rho_\alpha \rest n_1}) - \sum\limits_{n_1 \le n}
c_{n_1} \, f(y^2_{\rho_\beta \rest n_1})$
  which is zero as $n_1 \le n \Rightarrow \rho_\alpha \rest n_1 =
  \rho_\beta \rest n_1$
\sn
\item[$\odot_3$]  $\sum\limits_{n_1 \le n} \, \sum\limits_{m < \bold k} 
c_{n_1} \, f(x_{m,\omega \cdot \alpha + n}) -
\sum\limits_{n_1 \le n} \, \sum\limits_{m < \bold k}
c_{n_1} \, f(x_{m,\omega \cdot \beta +n})$ which, recalling
$(*)^3_{\alpha,n} + (*)^3_{\beta,n}$ is equal to $b_{\alpha,n} -
b_{\beta,n}$ by the choice of $b_{\alpha,n},b_{\beta,n}$ as $f,h$
agree on $G_0$
\sn
\item[$\odot_4$]  $\sum\limits_{n_1 \le n} c_{n_1} \,
a_{\alpha,n_1}f(z) - \sum\limits_{n_1 \le n} c_{n_1} \, 
a_{\beta,n_1} f(z)$.
\end{enumerate}
\mn
Hence (recalling $f(z) = h(z)$)
\mn
\begin{enumerate}
\item[$\boxdot$]  $(b_{\alpha,n} + \sum\limits_{n_1 \le n}
 c_{n_1} \, a_{\alpha,n_1} f_\alpha(z)) - 
(b_{\beta,n} + \sum\limits_{n_1 \le n} c_{n_1} \,
a_{\beta,n} f(z))$ is divisible by $p_n!$ in $\bbZ$.
\end{enumerate}
\mn
But by the choice of $a_{\alpha,n}$, i.e. by $(*)^5_{\alpha,n}$
 we know that $(b_{\alpha,n} + \sum\limits_{n_1 \le n} c_{n_1} \, 
a_{\alpha,n} f(z))$ is equal modulo $p_n!$
to some $i \in C_n$ iff $\rho_\alpha(n)=1$.  Similarly 
for $\beta$, but $\rho_\alpha(n) \ne
\rho_\beta(n)$ contradiction to $\boxdot$.  So indeed, $\odot$ leads
to contradiction.  This means that the demand in \ref{d23}(1)(B) is
satisfied.  Second, recall that we need to verify the ``freely fit".
This means that for
\mn
\begin{enumerate}
\item[$\circledast_1$]  for $\bar a$ as above and $w \in J$, the
 Abelian group $G_{\bar a}/\oplus\{\bbZ x_\alpha:\alpha \in w\}$ is free
\sn
\item[$\circledast_2$]   $G_{\bar a}$ is free.
\end{enumerate}
\mn
[Why?  Easy.]

Hence
\mn
\begin{enumerate}
\item[$\circledast_3$]  \wilog \, $w = \{w \alpha +n:\alpha <
  \alpha_*$ or $\alpha < \omega_1 \upharpoonleft n < n^*_\alpha\}$ for
  some $\alpha_* < w_1$ and $n^*_\alpha < w$ for $\alpha < \omega_1$.
\end{enumerate}
\mn
Now
\mn
\begin{enumerate}
\item[$\circledast_4$]   letting $G_* = \oplus\{\bbZ
  y^2_\varrho:\varrho \in {}^{\omega >}2\} \oplus \bigoplus\{\bbZ
X_\alpha:\alpha < \omega \alpha_*\},B_\omega = \oplus\{\bbZ
  X_\alpha:\alpha \in \omega,\alpha \ge \omega \alpha_*\}$ we have
\sn
\begin{enumerate}
\item[(a)]  $G_\omega + G_* = G_\omega \oplus G_*$
\newline
(Why?  check.)
\sn
\item[(b)]  it suffices to prove $G_{\bar a}/(G_\omega \oplus G_*)$ is
  free
\newline
(Why?  By (a).]
\sn
\item[(c)]   $G_{\bar a}/(G_\omega \oplus G_*)$ is the direct such of
  $H'_\alpha := \langle H_\alpha + (G_\omega \oplus
  G_*\rangle/G_\omega \oplus G_*:\alpha \in [\omega
  \alpha_*,\omega_1]\rangle$ where $H_\alpha$ is the subgroup of
  $G_{\bar a}$ generated by $\{X_{\omega \alpha +n}:n < \omega\} \cup
  \{y^1_{\alpha,n}:n < \omega\} \cup \{y^2_{\rho_\alpha \rest n}:n <
  \omega\}$
\newline
(Why?  Check.)
\sn
\item[(d)]  it suffices to prove each $H'_\alpha$ is a free Abelian
  group
\newline
(Why?  By (c).)
\sn
\item[(e)]  $H'_\alpha$ is isomorphic to $H_\alpha/\oplus(\cup\{\bbZ
  X_{\omega \alpha,n}:n < n_\alpha\} \cup\{\bbZ y^2_{\rho_\alpha \rest
    n}:n < \omega\})$
\newline
(Why?  check)
\sn
\item[(f)]  $H'_\alpha$ is indeed free
\newline
(Why?  By the same proof as in \ref{d29}.)
\end{enumerate}
\end{enumerate}
\mn
So $(\partial,J,\bold k,\theta)$ freely fit $\bbZ$ indeed.

\noindent
2) We can fix $G_0 = \oplus\{R X_{m,i}:m < \bold k,i < \partial_m\}
\oplus Rz,h \in \Hom(G_0,{}_R R)$ such that $h(z) \ne 0$.  Let
$\Omega,\langle \rho_\alpha:\alpha < \omega_1\rangle$ be as in the
proof of part (1).

 We are given $b_{\alpha,n} = h(x_{m,\omega\alpha +n}
 (\alpha < \aleph_1,n \in \bbN)$ and $t = h(z)$
from $R$.  We shall choose $\langle (a_{\alpha,n},d_{\alpha,m}):
\alpha < \omega_1,n < \omega\rangle$ and will let $G$ be the 
$R$-module generated by $\{x_{m,\alpha}:\alpha <
\aleph_1,m < \bold k\} \cup \{y^1_{\alpha,n}:\alpha < \aleph_1$ and $n
 < \omega\} \cup \{y^2_\varrho:\varrho \in {}^{\omega >}2\} \cup 
\{z\}$ freely except the equations
\mn
\begin{enumerate}
\item[$(*)_{\alpha,n}$]  $d_{\alpha,n} y^1_{\alpha,n+1} =
  y^1_{\alpha,n} + y^2_{\rho_\alpha \rest n} + \sum\limits_{m <
  \bold k} x_{m,\omega \cdot \alpha +n} - a_{\alpha,n} z$.
\end{enumerate}
\mn
Hence
\mn
\begin{enumerate}
\item[$(*)'_{\alpha,n}$]  $y^1_{\alpha,0} =
(\prod\limits_{\ell < n} d_{\alpha,\ell}) y^1_{\alpha,n} + 
\sum\limits_{n_1 < n} (\prod\limits^{n-1}_{\ell = n_1} d_{\alpha,\ell})
y^2_{\rho_\alpha \rest \ell} + \sum\limits_{n_1 < n} \,
\sum\limits_{m < \bold k} (\prod\limits^{n-1}_{\ell = n_1} d_{\alpha,\ell})
x_{m,\omega \cdot \alpha +n} + \sum\limits_{n_1 < n}
(\prod\limits^{n-1}_{\ell = n_1} d_\ell) a_{\alpha,n_1} z$.
\end{enumerate}
\mn
Now continue as in the proof of part (1).
\end{PROOF}

\noindent
We now can put things together
\begin{theorem}
\label{d35}
1) For every $k \ge 1$ there is an $\aleph_{\omega_1 \cdot k}$-free
Abelian group $G$ which is not Whitehead and even $\Hom(G,\bbZ) = 0$.

\noindent
2) If the ring $R$ satisfies the demands in clause (c) part (2) 
from \ref{d31} \then \, for every $k$ there is an 
$\aleph_{\omega_1 \cdot k}$-free $R$-module such that 
$\Hom(G,{}_R R) = 0$ and $\Ext(G,{}_R R) \ne 0$.
\end{theorem}

\begin{PROOF}{\ref{d35}}
1) Given $k$ we use \ref{a52} to find a c.p. $\bold x$ which is
$\aleph_{\omega_1 \cdot k}$-free and has $\chi$-BB where $\chi
= |R| + \aleph_1$ and $J = J^{\bd}_{\aleph_0 \times \aleph_1}$.  Now
apply \ref{d31}(1) so $(\aleph_1,J,k)$ fits $\bbZ$ and by
\ref{d24}(1),(2) we get the desired conclusion.

\noindent
2) Similarly, but now we use \ref{d31}(2) rather than \ref{d31}(1).
\end{PROOF}
\newpage

\section {Forcing} \label{3}
\bigskip

The main result of the former section is the existence in ZFC, of
$\aleph_{\omega_1 \cdot n}$-free Abelian groups $G$ (for every $n \in
\omega$) such that $\Hom(G,\bbZ) = 0$.  The purpose of this section is
to show that this result is best possible in the sense of freeness
amount.  Assuming the existence of $\aleph_0$-many supercompact
cardinals in the ground model, we shall force the following
statement.  For every $\aleph_{\omega_1 \cdot \omega}$-free Abelian
group $G,\Hom(G,\bbZ) \ne 0$.

This section is divided into two subsections.  In \S(3A) like \S1 is
combinatorial; we describe a general framework for dealing with
freeness of $R$-modules (this continues \cite{Sh:161}, \cite{Sh:521}
and \cite{Sh:266}; but have to work more).

In \S(3B) we rely on forcing, we focus on $R = \bbZ$ (hence
$R$-modules are simply Abelian groups), and we prove the main
consistency result in Theorem \ref{f2} which relies on Magidor-Shelah
\cite{MgSh:204}.  The proof is based on the context of \S(3A), with
double meaning.
\bigskip

\subsection {Freeness Classes} \label{Freeness}\
\bigskip

\begin{context}  
\label{g2}
1) $R$ is a ring with no zero divisors and is hereditary (see \ref{d0}(1A)).

\noindent
2) $\bold K$ is the class of $R$-rings.

\noindent
3) $\bold K_*$ will denote a class $\subseteq \bold K$.
\end{context}

\begin{definition}  
\label{g8}
0) $\bold K^w = \{M \in \bold K:M$ a Whitehead module that is,
   $\Ext(M,{}_R R)=0$ equivalently, if $N_1 \subseteq N_2$ are
   $R$-modules, $N_2/N_1 \cong M$ and $h_1 \in \Hom(N_1,{}_R R)$ \then
   \, there is $h_2 \in \Hom(N_2,{}_R R)$ extending $h_1\}$.

\noindent
1) We say $\bold K_*$ is a $\lambda$-freeness class inside $\bold K$ \when \,:
\mn
\begin{enumerate}
\item[$(a)$]  $\bold K_* \subseteq \bold K_{< \lambda}$ where for any
cardinality $\theta$ we let 
\newline
$\bold K_{< \theta} := \{M \in \bold K:
\|M\| < \theta\}$ 
\sn
\item[$(b)$]  $\bold K_*$ is closed under isomorphisms
\sn
\item[$(c)$]   for simplicity $\lambda > |R|$.
\end{enumerate}
\mn
1A) We say $\bold K_*$ is hereditary \when \,  
$\bold K_*$ is closed under pure submodules, i.e. $M \subseteq_{\pr} N
\in \bold K_* \Rightarrow M \in \bold K_*$.

\noindent
2) We say $M \in \bold K$ is $\bold K_*$-free \when \, there is
$\bar M$ such that $\bar M = \langle M_\alpha:\alpha \le 
\alpha_*\rangle$ is purely increasing continuous,
$M_0$ is the zero module and $\alpha < \alpha_* \Rightarrow 
M_{\alpha+1}/M_\alpha \in \bold K_*$ and
$M_{\alpha_*} = M$.

\noindent
2A) $M \in \bold K$ is $(\lambda,\bold K_*)$-free \when \, every $M'
\subseteq_{\pr} M$ of cardinality $< \lambda$ is $\bold K_*$-free.

\noindent
3) $\bold K^*_{< \theta} = \bold K_* \cap \bold K_{< \theta}$ for any
   cardinal $\theta$.

\noindent
4) The class $\bold K_*$ is called a $(\lambda,\kappa)$-freeness class
\when \,: $\bold K_*$ is a $\lambda$-freeness class, $\bold K_*$ is
hereditary and if $M \in \bold K_{< \lambda} \backslash \bold K_*$ \then \,
there is $N \subseteq_{\pr} M$ from $\bold K_{< \kappa} \backslash \bold
K_*$.
\end{definition}

\noindent
The main example here is:
\begin{claim}  
\label{g16}
Assume $R = \bbZ,\lambda \ge \aleph_1$ and $\bold K =$ the class of 
$R$-modules and let $\bold K_{whu} = \bold K_* = 
\{M \in \bold K_{< \lambda}:M$ is a Whitehead module, equivalently satisfies 
condition $\oplus(b)$ of \ref{f2} below$\}$ and $\bold K_{\fr} = \{M \in
\bold K_{< \aleph_1}:M$ free$\}$.

\noindent
0) $\bold K_{\fr}$ is a hereditary $\aleph_1$-freeness class.

\noindent
1) If $\lambda > \aleph_2$ and $\MA_{< \lambda}$ \then \,
$\bold K_*$ is a hereditary $(\lambda,\aleph_2)$-freeness class.

\noindent
2) If $M \in \bold K$ is $\bold K_*$-free \then \, $M$ is a Whitehead
   group.

\noindent
3) If $M_1 \subseteq_{\pr} M_2$ and $M_2/M_1$ is $\bold K_*$-free and
   $h_1 \in \Hom(M_1,{}_R R)$ \then \, there is $h_2 \in \Hom(M_2,{}_R R)$
   extending $h_1$.

\noindent
4) If $\bold K_{**} = \{M \in K_{< \lambda}$: for every c.c.c. forcing
$\bbP_1$ for some c.c.c. forcing notion $\bbP_2$ satisfying $\bbP_1
\lessdot \bbP_2$ we have $\Vdash_{\bbP_2}$ ``$M$ is a Whitehead group"$\}$ 
\then \, $\bold K_{**}$ is $(\lambda,\aleph_2)$-freeness class.
\end{claim}

\begin{PROOF}{\ref{g16}}
0) Obvious as $\bbZ$ is countable.

\noindent
1) The first property in \ref{g8}(4) holds trivially by the choice of
 $\bold K_*$.  As for the second property it is well known 
that $\bold K_*$ is a hereditary class, see \cite{Fu}.  
The third property in \ref{g8}(4) follows from the full 
characterization of being Whitehead for Abelian
   group $G$ of cardinality $< \lambda$ when $\MA_{< \lambda}$ holds,
(not just proving ``strongly $\aleph_1$-free is enough");
 in particular $G$ is Whitehead if every subgroup of cardinality
 $\le \aleph_1$ is Whitehead; see \cite{EM02}.

\noindent
2) Follows by (3).

\noindent
3)  \Wilog \, let $M = M_2/M_1$ and $\pi \in \Hom(M_2,M)$
 be onto with kernel $M$. Let $\langle M'_\alpha:\alpha \le
 \alpha_*\rangle$ be as in \ref{g8}(4) for $M$ and let $N_\alpha =
 \pi^{-1}(M'_\alpha)$ so $\langle N_\alpha:\alpha \le
\alpha_*\rangle$ is purely increasing continuous, $N^*_0 =
M_1,N_{\alpha_*} = M_2$ and $N_{\alpha +1}/N_\alpha \in \bold K_*$.

Given $h_1 \in \Hom(M_1,{}_R R)$ by induction on $\alpha$ 
we choose $f_\alpha \in
\Hom(N_\alpha,{}_R R)$ increasing continuous with $\alpha$.  For
$\alpha = 0$ let $f_\alpha = h_1$, for $\alpha$ limit let $f_\alpha =
\cup\{f_\beta:\beta < \alpha\}$ and for $\alpha = \beta +1$ use
$N_\alpha/N_\beta \cong M'_\alpha/M'_\beta \in \bold K_*$ and the
choice of $\bold K_*$.  

Lastly, $h_2 = f_{\alpha_*}$ is as required.

\noindent
4) Easy.  
\end{PROOF}

\noindent
Now on those freeness contexts see \cite{Sh:52} or better
\cite{Sh:266} and history there.  Note that we shall in \S(3B) use
\ref{g44}(B)(c), and for this we need witnesses $\bold s$ from those
references. 
Recall (see \cite{Sh:266}):
\begin{definition}
\label{g38}
1) We say $\bold c$ is a pre-1-freeness context \when \, $\bold c$ consists of:
\mn
\begin{enumerate}
\item[$(a)$]   $\cU$ is a fixed set (we shall deal with subsets of it) or
$\gU$ is an algebra with universe $\cU$ (maybe with empty 
set of functions); let $c \ell_{\bold c}(A)$ be the closure of the 
set $A \subseteq \cU$ in the algebra $\gU$; but we may
sometime say $\cU$ instead of $\gU$
\sn
\item[$(b)$]    $\cF$ a family of pairs
 of subsets of $\cU$; we may write ``$A/B$ is free" or ``$A$ is 
free over $B$" for $(A,B)$ in $\cF$
\sn
\item[$(c)$]  $\chi,\mu$ will be fixed cardinals such that 
$|\tau(\gU)| \le \chi < \mu \le \infty$ and $(A,B) \in \cF \Rightarrow
|A| + |B| < \mu$, but if $\mu = \infty$
  (equivalently, $\mu > |\cU|$) we may omit it.
\end{enumerate}
\mn
2) We say ``for the $\chi$-majority of $X \subseteq
A,P(X)$" (for a property $P$) \when \, there is an algebra $\gB$ with
universe $A$ and $\chi$ functions, such that any $X \subseteq A$ closed under
those functions satisfies $P$.  We can replace $X \subseteq A$ by $X \in
{\cP}(A)$ or $X \in {\cP}_{< \lambda}(A)$: alternatively we may say
$\{X \subseteq A:P(A)\}$ is a $\chi$-majority.

\noindent
3) We say $\bold c$ is a freeness context \when \, in addition to (a),(b),(c)
of part (1) it satisfies the following; 
adding a superscript$^+$ to an axiom means that
whenever ``$A/B \in \cF$" or its negation appear in the assumption then we
demand $B$ to be free over $\emptyset$. Of course, $\cF_{\bold c} =
\bold F,\chi_{\bold c} = \chi$, etc.
\bigskip

\noindent
$\Ax II_\mu$:
\mn
\begin{enumerate}
\item[$(a)$]   $A/B$ is free \Iff \, $A \cup B/B$ is free
\sn
\item[$(b)_\mu$]   $A/B$ is free \when \, $|B| < \mu$ and $A \subseteq B$.
\end{enumerate}
\medskip

\noindent
$\Ax III$:  [2-transitivity]  If $A/B$ and $B/C$ are free and 
$C \subseteq B \subseteq A$ \then \, $A/C$ is free.
\medskip

\noindent
$\Ax IV_{\lambda,\mu}$:  [continuous transitivity]  If $A_i(i < \lambda)$
is increasing, for $i \le \gamma < \lambda$ we have 
$A_\gamma/ \bigcup\limits_{j < i} A_j \cup B$ is
free, $\lambda < \mu$ and $|\bigcup\limits_{i < \lambda} A_i| < \mu$
\then \, $\bigcup\limits_{i < \lambda} A_i/B$ is free. 

\noindent
Let $\Ax(IV_{< \lambda,\mu})$ mean $\theta < \lambda \Rightarrow
\Ax(IV_{\theta,\mu})$; and $\Ax IV_\mu$ will mean $\Ax IV_{<\mu,\mu}$ 
and $IV$ means $IV_\infty$.
\medskip

\noindent
$\Ax VI$:  If $A$ is free over $B \cup C$ \then \, for the 
$\chi_{\bold c}$-majority of $X \subseteq A \cup B \cup C$ the pair
$A \cap X/(B \cap X) \cup C$ is free.
\medskip

\noindent
$\Ax VII$:  If $A$ is free over $B$, \then \, for the 
$\chi_{\bold c}$-majority of $X \subseteq A \cup B$ the 
pair $A/(A \cap X) \cup B$ is free.

\noindent
4) We say $\bold c$ is a freeness$^+$ context \when \, in addition
\medskip

\noindent
$\Ax I^{**}$:  If $A/B$ is free and $A^* \subseteq A$ \then \,
$A^*/B$ is free.

\noindent
5) We say $\bold c$ is a $(\lambda,\kappa)$-freeness context \when \,:
in addition $\chi_{\bold c} \le \kappa$, AxI$^{**}$ and 
if $A/B$ is not $\bold c$-free and $|A| < \lambda$ then for some $A'
\subseteq A$ of cardinality $< \kappa,A'/B$ is not $\bold c$-free
\end{definition}

\begin{definition}
\label{g41}
For a freeness class $\bold K_*$ and $R$-module $G$ and $\chi \ge
|R| + \aleph_0$ (if equal then $\chi$ may be omitted) we define what we call a 
pre-freeness context $\bold c = \bold c_G = \bold c_{\bold K_*,G,\chi}$,
(this is proved in \ref{g18}) as the triple $(\cU,\gA,\cF,\chi) = 
(\cU_{\bold x},\gA_{\bold c},\cF_{\bold c},\chi_{\bold c})$ where:
\mn
\begin{enumerate}
\item[$(a)$]  $\cU = G$ as a set and $\gA$ is an expansion of $G$ by
  $F^{\gA}_a(a \in R)$ such that: if $G \models nx = y$ and $y' =
  F_a(y)$ then $G \models ax'=y$
\sn
\item[$(b)$]  $\cF = \{A/B:B,A \subseteq \cU$ and $\langle A \cup
B\rangle_G/\langle B \rangle_G$ is $\bold K_*$-free$\}$, we may say
$A/B$ is $\bold c$-free so $A/B$ stands for the formal quotient, 
so pedantically is just the pair $(A,B)$; where $\langle B \rangle_G$
is the minimal pure\footnote{Our modules are torsion free (i.e. $a \in
  R \wedge x \in G \Rightarrow (ax=0 \Leftrightarrow (a = 0_{\bbR}
  \vee x = 0_G))$, holds when $R =\bbZ$ this has no problem.
  Otherwise, recall we have expanded $G$ to an algebra $\gA$ such that $A = c
  \ell_{\cU}(A) \Rightarrow A \subseteq_{\pr} G$} 
sub-module of $G$ which includes $B$
\sn
\item[$(c)$]  $\chi_{\bold c} = \chi$ so $\ge |R| + \aleph_0$, 
(and $\mu_{\bold c}=\infty$).
\end{enumerate}
\end{definition}

\begin{fact}  
\label{g18}
Assume $\bold K_*$ is a $\lambda$-freeness class and $\chi = |R| + \aleph_0$.

\noindent
1) Being $\bold K_*$-free has compactness in singular cardinals $> \lambda$.

\noindent
2) For any $R$-module $G_*,\bold c = \bold c_{\bold K_*,G_*,\chi}$ 
defined in \ref{g41} above is a freeness context and satisfies
$\Ax I^{**}$.

\noindent
3) If $\bold K_*$ is moreover a $(\lambda,\kappa)$-freeness class (see
\ref{g8}(4)) \then \, $\bold c$ is a $(\lambda,\kappa)$-freeness
context (see \ref{g38}(5)).
\end{fact}

\begin{PROOF}{\ref{g18}}
1) By part (2) and \cite{Sh:266}, see history there.

\noindent
2) Check.

\noindent
3) Easy.
\end{PROOF}

\begin{claim}
\label{g44}
If (A) then (B) where:
\mn
\begin{enumerate}
\item[$(A)$]  $(a) \quad \bold K_*$ is a $(\lambda,\kappa)$-freeness
class, see Definition \ref{g8}(4)
\sn
\item[${{}}$]  $(b) \quad G \in \bold K_*$ is $(\bold
K_*,\lambda)$-free not free, see Definition \ref{g8}(2A); 

\hskip25pt  fix such $G$ of minimal cardinality called $\mu$
\sn
\item[${{}}$]  $(c) \quad \bold c = \bold c_{\bold K_*,G,\kappa}$, see
Definition \ref{g41}(1)
\sn
\item[$(B)$]  there is a witness $\bold s$ for $G$ in the context $\bold c$
(see \cite[\S2]{Sh:161} and better \cite[\S3]{Sh:521}) such that:
\sn
\item[${{}}$]  $(a) \quad B^{\bold s}_{<>} = \emptyset,B^{\bold
s}_{<>^+} \subseteq G$ so $\lambda(<>,S_{\bold s}) \le \|M\|$
\sn
\item[${{}}$]  $(b) \quad$ if $\eta \notin \fin(S_{\bold s})$ then
  $\lambda_{\bold s,\eta} \ge \lambda$
\sn
\item[${{}}$]  $(c) \quad$ if $\eta \char 94 \langle \delta \rangle
 \in S_{\bold s}$ then $\cf(\delta) \notin [\kappa,\lambda)$
\sn
\item[${{}}$]  $(d) \quad$ if $\eta \in \fin(S_{\bold s})$ then
  $B_{\bold s,\eta^+} \backslash B_{\bold s,\eta}$ has cardinality $<
  \kappa$
\end{enumerate}
\end{claim}

\begin{PROOF}{\ref{g44}}
By \ref{g18} we can apply \ref{g47} below.
\end{PROOF}

\noindent
\begin{claim}
\label{g47}
If (A) then (B) where:
\mn
\begin{enumerate}
\item[$(A)$]  $(a) \quad \bold c$ is a freeness context satisfying
$\Ax I^{**}$
\sn
\item[${{}}$]  $(b) \quad \bold c$ is $(\lambda,\kappa)$-freeness context
\sn
\item[${{}}$]  $(c) \quad A/B$ is $\lambda$-free not free pair;
 and with $|A|$ minimal
\sn
\item[$(B)$]  there is a witness $\bold s$ such that
\sn
\item[${{}}$]  $(a) \quad B^{\bold s}_{<>} = B$ and $B^{\bold s}_{<>^+} 
\subseteq A$ so $\lambda(<>,S_{\bold s}) \le |A|$ but $\ge \lambda$
\sn
\item[${{}}$]  $(b) \quad$ if $\eta \notin \fin(S_{\bold s})$ \then \,
 $\lambda_{\bold s,\eta} \ge \lambda$
\sn
\item[${{}}$]  $(c) \quad$ if $\eta \char 94 \langle \delta \rangle
  \in S_{\bold s}$ \then \, $\cf(\delta) \notin [\kappa,\lambda)$
\sn
\item[${{}}$]  $(d) \quad$ if $\eta \in \fin(S_{\bold s})$ then
  $B_{\bold s,\eta^+} \backslash B_{\bold s,\eta}$ has cardinality $<
  \kappa$.
\end{enumerate}
\end{claim}

\begin{PROOF}{\ref{g47}}
Now (see \cite[\S3]{Sh:521} or better yet see
\cite[4.5=Ld15]{Sh:F13}), there is a
disjoint witness $\bold s$ for $A/B$ being
non-$\bold c$-free.  So \wilog \, ($n = n(\bold s)$ is well defined
and) for some $\bar\lambda^* = \langle \lambda^*_\ell:\ell <
n\rangle,\bar\kappa^* = \langle \kappa^*_\ell:\ell < n\rangle$ we have
\mn
\begin{enumerate}
\item[$(*)_1$]
\begin{enumerate}
\item[(a)]  for each $\ell < n$ one of the following holds:
\begin{enumerate}
\item[$(\alpha)$]  $\lambda_\ell$ is a regular cardinal and 
$\eta \in S_{\bold s,\ell} \Rightarrow \lambda(\eta,S_{\bold s}) 
= \lambda^*_\ell$ 
\sn
\item[$(\beta)$]  $\lambda_\ell = *$ and $\eta \in S_{\bold s,\ell} 
\Rightarrow \lambda(\eta,S_{\bold s})$ is (possibly weakly)
inaccessible
\end{enumerate}
\sn
\item[(b)]   if for each $\ell < n$, either $\kappa_\ell$ is a
 regular cardinal and $\eta \in S_{\bold s,\ell} \wedge \delta \in 
W(\eta,S_{\bold s}) \Rightarrow \cf(\delta) = \kappa_\ell$
\underline{or}  $\kappa_\ell = *$ and $\lambda_{\ell +1}$ is $*$.
\end{enumerate}
\end{enumerate}
\mn
See more there; naturally \wilog \,
\mn
\begin{enumerate}
\item[$(*)_2$]  $\bold s$ is minimal which means that (fixing $A$ and $B$)
\sn
\begin{enumerate}
\item[$(a)$]  $n = n(\bold s)$ is minimal
\sn
\item[$(b)$]  under (a), $\bar\lambda^*$ is minimal under the
  lexicographical order
\sn
\item[$(c)$]  under (a) + (b), $\bar\kappa^*$ is minimal under the
  lexicopgrahical order.
\end{enumerate}
\end{enumerate}
\mn
Now
\mn
\begin{enumerate}
\item[$(*)_3$]  if $\eta \in \ini(S_{\bold s})$ then
  $\lambda(\eta,S_{\bold s}) \ge \lambda$.
\end{enumerate}
\mn
[Why?  Otherwise chose a counterexample $\eta$ with
$\lambda(\eta,S_{\bold s})$ minimal so
 by the definition of a witness as $\chi_{\bold c} \le \kappa$ we
  have $B^{\bold s}_{\eta^+}/B^{\bold s}_{\le \eta}$ is not free 
$B^{\bold s}_{\eta^+} \backslash B^{\bold s}_\eta$ has cardinality
  $\lambda(\eta,S_{\bold s})$ so $< \lambda$ and
$\alpha < \lambda(\eta,S_{\bold s}) \Rightarrow B^{\bold s}_{\eta
\char 94 <\alpha +1>} \backslash B^{\bold s}_{\eta \char 94 <\alpha>}$ 
 has cardinality $< \kappa$.  
Recalling ``$\bold c$ is a $(\lambda,\kappa)$-freeness
  context", see Definition \ref{g38}(5) and \ref{g18}(3), there is $C_\eta 
\subseteq B^{\bold s}_{\eta^+}$ of cardinality $\le \kappa$ such 
that $C_\eta/B^{\bold s}_{\le \eta}$ is not $\bold c$-free.  
So (follows by minimality of $\bold s$) we get contradiction, so
$\lambda(\eta,S_{\bold s}) \ge \lambda$ as promised in $(*)_3$.]
\mn
\begin{enumerate}
\item[$(*)_4$]  if $\eta \char 94 \langle \delta\rangle \in S_{\bold
  s}$ then $\cf(\delta) \notin [\kappa,\lambda)$.
\end{enumerate}
\mn
[Why?  As in the proof there for each $\eta \in S_{\bold s}$
satisfying $\cf(\delta) \ge \kappa$ by the
 minimality, $\cf(\delta) \in \{\lambda(\nu,S_{\bold s}):\nu \in
  S_{\bold s}$ satisfies $\eta \triangleleft \nu\}$, so $(*)_4$
  follows by $(*)_3$.]

So we are done.
\end{PROOF}
\bigskip

\subsection {The Main Independence Result} \label{The} \
\bigskip

\noindent
Below, it is reasonable to assume that the ring $R$ is $\bbZ$ and we assume
this is the nice version.  Note that we prove that a non-Whitehead group has a
non-free subgroup of small cardinality, not necessarily a
non-Whitehead one.  This is connected to the black boxes here having
cardinality (much) bigger than the amount of freedom.  For simplicity,
presently we deal with freeness only in hereditary cases.

Recall that $\mu$ is supercompact iff for every $\partial$ there
exists an elementary embedding $j:\bold V \rightarrow M$ such that $M$
is a transitive class satisfying ${}^\partial M \subseteq M$.
 
\begin{theorem}  
\label{f2}
If in $\bold V$ there are $\aleph_0$-many supercompact cardinals, 
\then \, in some
 forcing extension we have for $\mu_* = \aleph_{\omega_1 \cdot \omega}$:
\mn
\begin{enumerate}
\item[$\oplus_{\mu_*}$]
\begin{enumerate}
\item[(a)] if $G$ is a $\mu_*$-free Abelian group, 
\then \, $\Hom(G,\bbZ) \ne 0$, 
\sn
\item[(b)]   if $G \subseteq H$ are Abelian groups and 
$H/G$ is $\mu_*$-free and $h \in \Hom(G,\bbZ)$ \then \,
  $h$ can be extended to a homomorphism from 
$H$ to $\bbZ$, (this is an equivalent definition of ``$H/G$ is Whitehead", 
the reader may use it here as a definition).
\end{enumerate}
\end{enumerate}
\end{theorem}

\noindent
As usual in such proof, we make the collapse large cardinal into quite
small ones, so they cannot be real large but some remnant of their
early largeness remains and is enough for our purpose.  This is the
rationale of Definition \ref{f1g} below.
\begin{definition}
\label{f1g}
Let $\Pr_{\lambda_*,\mu_*,\kappa_*}$ means\footnote{May allow
$\lambda_* = \mu_*$ here and in \ref{f7} but then have to say somewhat
more.}
\mn
\begin{enumerate}
\item[(A)]
\begin{enumerate}
\item[(a)]  $\lambda_* > \mu_* > \kappa_*$
\sn
\item[(b)]  $\lambda_*,\kappa_*$ are regular uncountable cardinals
\sn
\item[(c)]  $\mu_*$ is a limit cardinal
\end{enumerate}
\sn
\item[$(B)$]  if $(a)$ then $(b)$ where
\begin{enumerate}
\item[(a)]
\begin{enumerate}
\item[$(\alpha)$]  $\lambda$ is a regular cardinal $\ge \lambda_*$
\sn
\item[$(\beta)$]  $\chi > \lambda$ and $\mu < \mu_*$ and $x \in \cH(\chi)$
\sn
\item[$(\gamma)$]  $S \subseteq \{\delta < \lambda:
\cf(\delta) < \kappa_*\}$ is a stationary subset of $\lambda$
\sn
\item[$(\delta)$]  $u_\alpha \in [\alpha]^{\le \mu}$
  for $\alpha \in S$
\end{enumerate}
\sn
\item[(b)]   there are a regular $\lambda' \in 
(\mu + \kappa_*,\mu_*)$ and an increasing continuous
sequence $\langle \alpha_\varepsilon:\varepsilon < \lambda'\rangle$ of
ordinals $< \lambda$ such that the set
$\{\varepsilon < \lambda':\alpha_\varepsilon \in S$ and $u_{\alpha_\varepsilon}
\subseteq \{\alpha_\zeta:\zeta < \varepsilon\}\}$ is a stationary
subset of $\lambda'$.
\end{enumerate}
\end{enumerate}
\end{definition}

\noindent
On the strong hypothesis above, see \cite{Sh:410}; a sufficient
condition for $\partial \ge \lambda \Rightarrow 2^\partial
= \partial^+$.
\begin{definition}
\label{f4d}
We say the universe $\bold V$ satisfies the strong hypothesis above
$\lambda$ \when \, if $\lambda > \lambda + \mu_1$ and for some
$\lambda \in (\mu_1,\lambda_1)$ we have $\cf([\chi]^{<
  \mu_1},\subseteq) \ge \lambda_2$, \then \, $\lambda_1 = \chi^+$ and
$\cf(\chi) < \mu_1$.
\end{definition}

\begin{theorem}  
\label{f5}
1) Assume in $\bold V_0$ there are infinitely many supercompact
cardinals $> \theta$ and $\theta = \cf(\theta) \in
[\aleph_1,\aleph_{\omega_1})$.  Then for some forcing notion $\bbQ$
 not adding new subsets to $\theta,\bold V_1 = \bold V^{\bbQ}_0$
satisfies $\Pr_{\lambda_*,\mu_*,\kappa_*}$ \where \, 
$\lambda_* = \cf(\lambda_*) = \mu^+_*,\mu_* = \aleph_{\theta \cdot
\omega}$ and $\kappa_* = \theta^+$.

\noindent
1A) We can (by preliminary forcing) assume that the 
universe $\bold V$ above satisfies also G.C.H. above $\theta$ (we use
just ``above $\mu_*$) and
$\diamondsuit^*_\lambda$ for every regular uncountable $\lambda$ above
$\mu_*$ and $|\bbQ_*| \le \mu^+_*$.

\noindent
2) If $\Pr_{\lambda_*,\mu_*,\kappa_*}$ holds in $\bold V$ and
 the c.c.c. forcing $\bbP$ has cardinality $\lambda_*$ \then \, in
   $\bold V^{\bbP}$ still $\Pr_{\lambda_*,\mu_*,\kappa_*}$ holds.

\noindent
3) Part (1) holds for any freeness$^+$ context (see Definition
\ref{g38}(3),(4)). 
\end{theorem}

\begin{PROOF}{\ref{f5}}
1), 1A)  Similarly to \cite[\S4,Th.1,pg.807]{MgSh:204}.  As there let
$\langle \kappa_n:n < \omega\rangle$ be an increasing sequence of
supercompact cardinals.  \Wilog \, GCH holds above $\mu = \sum\limits_{n}
\kappa_n$ (called $\kappa$ there) and $\diamondsuit^*_\chi$ holds for
$\chi = \cf(\chi) > \mu$.  Also for each $n$, the supercompactness of
$\kappa_n$ is preserved by forcing notions which are $\kappa_n$-directed
closed.

Choose $h_*:\omega \rightarrow \theta \cap \Reg$ such that for every
regular $\sigma < \theta$, the set $\{n:h_*(n) = \sigma\}$ is
infinite.  We proceed as there but now in the interval
$(\kappa_{n-1},\kappa_n)$, the set of cardinals we do not collapse has
order type $h_*(n)+1$.

\noindent
2),3)  Easy.
\end{PROOF}

\begin{PROOF}{\ref{f2}}  
\underline{Proof of \ref{f2}}

Let $\bold V_1 = \bold V^{\bbQ}_0$ be as in \ref{f5}(1)(1A) with $\theta
= \aleph_1$ so $\kappa_* = \aleph_2,\mu_* =
\aleph_{\omega_1 \cdot \omega},\lambda_* = \mu^+_*$  and let 
$\bbP$ be a c.c.c. forcing notion of cardinality
 $\lambda_*$ such that $\Vdash_{\bbP} ``\MA + 2^{\aleph_0} =
 \lambda_*"$.  The result follows from Theorem
\ref{f7} below.  Clause (d) there holds because $\bold V = 
\bold V^{\bbP}_1$, see \ref{f5}(2).
\end{PROOF}

\noindent
Recall that the strong hypothesissays that $\pp(\lambda) = \lambda^+$
for every singular cardinal $\lambda$; we rely on \S(3A).
\begin{theorem}
\label{f7}
The statement $\oplus_{\mu_*}$  from \ref{f2} holds \when \,
 $\bold V$ satisfies:
\mn
\begin{enumerate}
\item[$(a)$]   the statement
$\Pr_{\lambda_*,\mu_*,\kappa_*}$ from Definition \ref{f1g}
\sn
\item[$(b)$]  $\lambda_* = \lambda^{< \lambda_*}_* > \mu_*$
\sn
\item[$(c)$]  $\kappa_* = \aleph_2$
\sn
\item[$(d)$]  $\MA + 2^{\aleph_0} = \lambda_*$ and
$\bold V$ satisfies the strong hypothesis above $\lambda_*$, 
see \cite{Sh:410}. 
\end{enumerate}
\end{theorem}

\begin{PROOF}{\ref{f7}}
We rely on \ref{g2} - \ref{g47}.
First clause (b) of $\oplus_{\mu_*}$ implies clause (a); why?
because if $H$ is
a $\mu_*$-free Abelian group, let $x \in H
\backslash \{0_H\}$ and \wilog \, $x$ is not divisible by any $n \in
   \{2,3,\ldots\}$ hence $K := \bbZ x$ is a pure subgroup of $H$, 
let $h$ be an isomorphism from $K$ onto $\bbZ$.  As $H$ is
$\mu_*$-free easily also $H/K$ is $\mu_*$-free, hence by
$\oplus_{\mu_*}(b)$ there is a homomorphism $h^+$ from $H$ to
$\bbZ$ extending $h$ so $h^+(x) \ne 0_{\bbZ}$ hence $h^+ \in 
\Hom(H,\bbZ)$ is non-zero, as required.

So it suffices to prove clause (b) of $\oplus_{\mu_*}$.

Let $R = \bbZ$ and let $\bold K,\bold K_*$ be as in 
Claim \ref{g16} for $\lambda_*$ so $\bold K_*$ is a
hereditary $(\mu_*,\aleph_2)$-freeness class (see Definition
\ref{g8}(1),(1A),(4))
 by \ref{g16}(1).  So toward contradiction assume $G \in \bold K$ is 
a counterexample of minimal cardinality called $\lambda$ so 
$G$ is $\mu_*$-free.  To get a contradiction and finish the proof it
suffices to assume $G_1 \subseteq_{\pr} G_2,G_2/G_1 \cong G$ and $h_1 \in 
\Hom(G_1,\bbZ)$ and prove that there is $h_2
\in \Hom(G_2,\bbZ)$ extending $h_1$.  If $G$ is $\bold K_*$-free (see
Definition \ref{g8}(2)) then by 
\ref{g16}(3) a homomorphism $h_2$ as required exists.

Hence $G$ is not $\bold K_*$-free and let $\bold c
 = \bold c_{\bold K_*,G,\theta}$, see Definition \ref{g41} so by
 \ref{g18}(3), $\bold c$ is a $(\lambda_*,\kappa_*)$-freeness context
 and by \ref{g44}(2),(3) (with $\lambda_*,\kappa_*$ here standing for $\lambda,
\kappa$ there) there is a witness $\bold s$ as there.  By \ref{g16}(1)
we have $\lambda(\langle \rangle,S_{\bold s}) \ge \lambda_*$.

Let $\bold c_1 = \bold c_{\bold K_{\fr},G,\theta}$, it is a
$(\lambda,\aleph_1)$-freeness context.  (Why?  By \ref{g18} with
$K_{\fr}$ (see \ref{g16}) playing the role of $K_*$.)

Let $S_1 = W(<>,S_{\bold s})$ so for each $\delta \in S_1,
B^{\bold s}_{<\delta +1>} /B^{\bold s}_{<\delta>}$ is not free for
$\bold c$ so cannot be $\mu_*$-free for $\bold c_1$ 
(as we have chosen a counter-example of minimal cardinality).
Hence there is $A_\delta \subseteq B^{\bold s}_{< \delta +1>}$ of
cardinality $< \mu_*$ such that $A_\delta/B^{\bold s}_{<\delta>}$ is
not free for $\bold c_1$.

Let $B'_\delta \subseteq B^{\bold s}_{<\delta>}$ be of cardinality
$\le |A_\delta| + \kappa_*$ such that $B'_\delta 
\subseteq B' \subseteq B^{\bold s}_{<\delta>} \Rightarrow 
A_\delta/B'$ is not free for $\bold c_1$; exists by
properties of Abelian groups as $B^{\bold s}_{\langle \delta\rangle}
\subseteq B^{\bold s}_{\langle \delta +1\rangle}$ are free (for $\bold c_1$).

So for some $\mu < \mu_*$ the set $S_2 = \{\delta \in S_1:|A_\delta \cup
B'_\delta| + \kappa_* = \mu\}$ is a stationary subset of $\lambda
(\langle \rangle,S_{\bold s})$.  
Let $h$ be a one-to-one function from $\lambda(\langle
\rangle,S_{\bold s})$ onto $B^{\bold s}_{<\lambda>}$ and let
$C := \{\delta < \lambda(\langle \rangle,S_{\bold s}):h$ 
maps $\delta$ onto $B^{\bold s}_{\langle \delta\rangle}\}$,
it is a club of $\lambda(\langle \rangle,S_{\bold s})$ hence $S_3 := 
S_2 \cap C$ is a stationary subset of $\lambda(\langle
\rangle,S_{\bold s})$.  Also for $\delta \in S_3$ let 
$u_\delta = \{\alpha < \delta:h(\alpha) \in B'_\delta\}$.

By clause (B)(c) of \ref{g44}, i.e. the choice of $\bold s$, \wilog \,
one of the following occurs:
\mn
\begin{enumerate}
\item[$(a)$]   $\delta \in S_3 \Rightarrow \cf(\delta) = \kappa_1$ for
  some regular $\kappa_1 < \kappa_*$
\sn
\item[$(b)$]   every $\delta \in S_3$ has cofinality $\ge \lambda_*$.
\end{enumerate}
\medskip

\noindent
\underline{Case 1}:  $\kappa_1 < \kappa_*$ is as in clause (a)

Just use $\Pr_{\lambda_*,\mu_*,\kappa_*}$ for $\lambda,S_3,\langle
u_\delta:\delta \in S_3\rangle$ to prove $G$ is not a
$\mu_*$-free, a contradiction.
\medskip

\noindent
\underline{Case 2}:  Clause (b) above holds

For $\delta \in S_3$ clearly $|u_\delta| = |A_\delta \cup B'_\delta| +
 \kappa_* = \mu < \mu_* \le \lambda_* \le \cf(\delta)$ hence there is
$\gamma_\delta < \delta$ such that $u_\delta \subseteq
 \gamma_\delta$ hence for some $\gamma_* < \lambda$
the set $S_4 = \{\delta \in S_3:u_\delta \subseteq \gamma_*\}$ is
stationary.
\medskip

\noindent
\underline{Subcase 2A}:  $\cf([\gamma_*]^{\le \mu_*},\subseteq)$ is $<
\lambda(\langle \rangle,S_{\bold s})$.

So for some $u_* \in [\gamma_*]^{\le \mu}$ the set 
$S_5 = \{\delta < \lambda:u_\delta \subseteq u_*\}$ is a stationary
subset of $\lambda$.  Let $S_6 \subseteq S_5$ be of cardinality
$\mu^+$ and let $A^* = \cup\{A_\delta:\delta \in S_6\} \cup
\{h(\alpha):\alpha \in u_*\}$.  Clearly $A^* \subseteq G$ is of
cardinality $< \mu$ and $A^*/\emptyset$ is not free for $\bold c_1$.

So $G$ has a non-free subgroup of cardinality 
$< \mu_*$, contradiction to the assumption ``$G = G_2/G_1$ is
$\mu_*$-free".
\medskip

\noindent
\underline{Subcase 2B}:  $\cf([\gamma_*]^{\le \mu},\subseteq) \ge
 \lambda(\langle \rangle,S_{\bold s})$.

Note that because $\bold V$ satisfies the strong hypothesis 
(see \cite{Sh:410}), necessarily
for some cardinal $\partial$ of cofinality $< \kappa_*$ we have
$\lambda(\langle \rangle,S_{\bold s}) = \partial^+$.

In any case clearly for every $\alpha \in [\gamma_*,\lambda)$, letting
$\beta_\alpha = \min(S_4 \backslash \alpha)$, the pair
$A_{\beta_\alpha}/B_{<\alpha>}$ is not $\bold c_1$-free.  So renaming
\wilog \, $\alpha \ge \gamma_* \wedge \cf(\alpha) = \aleph_0
\Rightarrow \langle \alpha \rangle \in S$ and we continue as in Case
1, so this works also in Subcase 2A.
\end{PROOF}
\bigskip

\newpage

\bibliographystyle{amsalpha}
\bibliography{lista,listb,listx,listf,liste,listz}

\end{document}